\documentclass{amsart}
\usepackage{amsfonts,amscd,amsthm,amsgen,amsmath,amssymb}
\usepackage[all]{xy}
\usepackage{epsfig,color}
\usepackage{tikz}
\usetikzlibrary{calc}
\usepackage[vcentermath]{youngtab}

\newtheorem{theorem}{Theorem}[section]
\newtheorem{lemma}[theorem]{Lemma}
\newtheorem{corollary}[theorem]{Corollary}
\newtheorem{proposition}[theorem]{Proposition}

\theoremstyle{definition}
\newtheorem{definition}[theorem]{Definition}
\newtheorem{example}[theorem]{Example}

\theoremstyle{remark}
\newtheorem{remark}[theorem]{Remark}

\numberwithin{equation}{section}

\begin{document}

\setlength\parskip{0.5em plus 0.1em minus 0.2em}

\title{Equivariant Seiberg--Witten theory}

\author{David Baraglia}
\address{School of Computer and Mathematical Sciences, The University of Adelaide, Adelaide SA 5005, Australia}
\email{david.baraglia@adelaide.edu.au}


\date{\today}

\begin{abstract}
We introduce and study equivariant Seiberg--Witten invariants for $4$-manifolds equipped with a smooth action of a finite group $G$. Our invariants come in two types: cohomological, valued in the group cohomology of $G$ and $K$-theoretic, valued in the representation ring of $G$. We establish basic properties of the invariants such as wall-crossing and vanishing of the invariants for $G$-invariant positive scalar curvature metrics. We establish a relation between the equivariant Seiberg--Witten invariants and families Seiberg--Witten invariants. Sufficient conditions are found under which equivariant transversality can be achieved leading to smooth moduli spaces on which $G$ acts. In the zero-dimensional case this yields a further invariant of the $G$-action valued in a refinement of the Burnside ring of $G$. We prove localisation formulas in cohomology and $K$-theory, relating the equivariant Seiberg--Witten invariants to moduli spaces of $G$-invariant solutions. We give an explicit formula for the invariants for holomorphic group actions on K\"ahler surfaces. We also prove a gluing formula for the invariants of equivariant connected sums. Various applications and consequences of the theory are considered.

\end{abstract}

\maketitle


\section{Introduction}

The goal of this paper is to define an equivariant version of the Seiberg--Witten invariants in the presence of a smooth group action and to establish their fundamental properties. Given a compact oriented smooth $4$-manifold $X$ equipped with an action of a finite group $G$ by diffeomorphisms, there are two natural approaches to incorporate the group action into Seiberg--Witten theory. The first approach is to consider only those solutions of the Seiberg--Witten equations on $X$ which are invariant under the action of $G$. Equivalently, one can interpret this as the Seiberg--Witten equations on the quotient orbifold $X/G$. This leads to a series of integer invariants $\overline{SW}_{G,X,\mathfrak{s}}$ which we refer to as the {\em reduced Seiberg--Witten invariants} of the group action. These invariants have been studied by various authors \cite{rua,cho,ish,sung0,sung}. 

The second approach, which is the main focus of this paper, is to consider the full set of solutions to the Seiberg--Witten equations while keeping track of the group action. However any attempt carry this out is confronted with the problem of equivariant transversality. We can not expect to get a smooth $G$-invariant moduli space in all cases (although we will see in Section \ref{sec:eswm} that this be achieved in some cases). The transversality problem can be evaded by the techinique of finite-dimensional approximation, as in the construction of the Bauer--Furuta invariant \cite{bf}. In this approach one forgoes the construction of smooth moduli spaces, working instead with a finite-dimensional approximation of the Seiberg--Witten equations. An alternative approach to addressing the problem of transversality which will occasionally be useful makes use of families Seiberg--Witten theory, see Section \ref{sec:eqfam}. Finite-dimensional approximation can be carried out equivariantly and this produces an invariant $BF_{G,X,\mathfrak{s}}$ which we refer to as the $G$-equivariant Bauer--Furuta invariant \cite{szy}. A significant problem with the equivariant Bauer--Furuta invariants is that they take values in certain twisted equivariant stable cohomotopy groups, which are very difficult to calculate. One of the motivations for this paper is to construct invariants of smooth group actions which are easier to work with.

Throughout the paper we assume that our $4$-manifold $X$ satisfies $b_1(X) = 0$. We have made this choice for simplicity, but it is not essential to the theory developed in this paper. Most of the results will carry over to the case $b_1(X) > 0$ with suitable modifications. Let $\mathfrak{s}$ be a spin$^c$-structure on $X$ whose isomorphism class is fixed by $G$. Associated to $\mathfrak{s}$ is a certain $S^1$-central extension of $G$
\[
1 \to S^1 \to G_{\mathfrak{s}} \to G \to 1
\]
The extension has the property that the action of $G$ on $X$ lifts to an action of $G_{\mathfrak{s}}$ on the spinor bundles corresponding to $\mathfrak{s}$. Assume the action of $G$ on $H^+(X)$ is orientation preserving. We will define the {\em (cohomological) equivariant Seiberg--Witten invariant} of $(X,\mathfrak{s})$ as a certain map
\[
SW_{G,X,\mathfrak{s}} : H^*( G_{\mathfrak{s}} ; \mathbb{Z}) \to H^{* - d(X,\mathfrak{s})}( G ; \mathbb{Z} )
\]
where $d(X,\mathfrak{s})$ is the expected dimension of the Seiberg--Witten moduli space. The definition can be extended to other coefficient groups and also extended to the case that $G$ does not preserve orientation on $H^+(X)$ using local coefficients. The definition requires that $b_+(X)^G > 0$, where $b_+(X)^G = dim( H^+(X)^G )$. In the case $b_+(X)^G = 1$ the definition depends also on a choice of chamber and there is a wall-crossing formula relating the invariants for the two chambers. The map $SW_{G,X,\mathfrak{s}}$ is defined by imitating the formula for the families Seiberg--Witten invariants given in \cite{bk}. In the case that a smooth $G$-invariant moduli space can be constructed, the map is given by evaluating cohomology classes over the moduli space. The cohomological equivariant Seiberg--Witten invariants were originally defined by the author and Konno in \cite{bk1}, although they were not the main focus of that paper and many of their basic properties were not established. We will also define a second type of invariant which we wall the {\em $K$-theoretic equivariant Seiberg--Witten invariant} of $(X , \mathfrak{s})$ as a map
\[
SW^K_{G,X,\mathfrak{s}} : R(G_{\mathfrak{s}}) \to R(G)
\]
where $R(G_{\mathfrak{s}}), R(G)$ denote the representation rings of $G_{\mathfrak{s}}, G$. The definition will require that $b_+(X)$ is odd and that $H^+(X)$ admits a $G$-equivariant spin$^c$-structure. It is sometimes more convenient to work with the $K$-theoretic invariants, for instance the representation ring $R(G)$ is typically easier to work with than the cohomology ring $H^*(G ; \mathbb{Z})$. Furthermore, the $K$-theoretic invariants behave better under localisation.

When the extension $G_{\mathfrak{s}}$ is split, a choice of splitting $G_{\mathfrak{s}} \cong S^1 \times G$ yields isomorphisms $H^*(G_{\mathfrak{s}} ; \mathbb{Z}) \cong H^*(G ; \mathbb{Z})[x]$ and $R(G_\mathfrak{s}) \cong R(G)[\xi,\xi^{-1}]$. In this case $SW_{G,X,\mathfrak{s}}$ and $SW_{G,X,\mathfrak{s}}^K$ are determined by the classes
\begin{align*}
SW_{G,X,\mathfrak{s}}(x^m)& \in H^{2m-d(X,\mathfrak{s})}(G ; \mathbb{Z}), \; m \ge 0 \\
SW_{G,X,\mathfrak{s}}^K( \xi^m )& \in R(G), \; m \in \mathbb{Z}.
\end{align*}

\subsection{Motivation for the invariants}

We are motivated to study equivariant Seiberg--Witten invariants for a variety of reasons. As previously mentioned, the equivariant Bauer--Furuta invariants $BF_{G,X,\mathfrak{s}}$ are valued in equivariant stable cohomotopy groups, which are very difficult to work with. The cohomological and $K$-theoretic equivariant Seiberg--Witten invariants $SW_{G,X,\mathfrak{s}}, SW_{G,X,\mathfrak{s}}^K$ are valued in group cohomology $H^*(G ; \mathbb{Z})$ and the representation ring $R(G)$, both of which are considerably easier to understand. The equivariant Seiberg--Witten invariants also have certain advantages over the reduced Seiberg--Witten invariants $\overline{SW}_{G,X,\mathfrak{s}}$. For instance $SW_{G,X,\mathfrak{s}}$ behaves well under restriction to subgroups of $G$ whereas $\overline{SW}_{G,X,\mathfrak{s}}$ does not. This is useful because it provides obstructions to extending a group action on $X$ to a larger group. Restricting to the trivial subgroup we recover the ordinary Seiberg--Witten invariants. On the other hand the reduced invariants are related to the equivariant invariants through localisation (as shown in Section \ref{sec:red}). From this we obtain non-trivial relations between the ordinary Seiberg--Witten invariants and the reduced Seiberg--Witten invariants, which leads to non-trivial constraints on the possible smooth group actions on $X$.

Another motivation comes from families Seiberg--Witten theory. If $P \to B$ is a principal $G$-bundle, then we can associate a family of $4$-manifolds over $B$ by taking the induced fibre bundle $E = P \times_G X$. We will show that the families Seiberg--Witten invariants of $E$ are given by pullback by the classifying map $f : B \to BG$ of $P$. Hence the equivariant Seiberg--Witten invariants can be used to compute families Seiberg--Witten invariants.

In \cite{bh}, the author and Hekmati introduced an equivariant version of Seiberg--Witten Floer homology for rational homology $3$-spheres. The equivariant Seiberg--Witten invariant is the $4$-dimensional counterpart to this Floer theory. In future work we aim to prove a gluing formula for the equivariant Seiberg--Witten invariants of $4$-manifolds obtained by gluing along rational homology $3$-spheres, extending the result of \cite{man,kls} to the equivariant setting. In this paper we consider only gluing along $3$-spheres.

\subsection{Main results}

Let $X$ be a compact, oriented smooth $4$-manifold with $b_1(X) = 0$ and let $G$ be a finite group which acts on $X$ by orientation preserving diffeomorphism. Let $\mathfrak{s}$ be a $G$-invariant spin$^c$-structure on $X$ and denote by $D$ the index of the corresponding spin$^c$ Dirac operator. Note that $D$ is a virtual representation of $G_{\mathfrak{s}}$ of rank $d = (c(\mathfrak{s})^2 - \sigma(X))/8$. Assume that $b_+(X)^G > 0$. We denote by $s_j(D) \in H^{2j}(G ; \mathbb{Z})$ the $j$-th Segre class of $D$. We summarise below the main properties of the equivariant Seiberg--Witten invariants. We state most results only for the cohomological invariants.

\noindent {\bf Diffeomorphism invariance.} Let $Y$ be another compact oriented smooth $4$-manifolds with $b_1(Y) = 0$ and let $G$ act smoothly on $Y$. Let $f : X \to Y$ be a $G$-equivariant diffeomorphism. Then
\[
SW_{G , X , f^*(\mathfrak{s}_Y)}^{f^*(\phi_Y)} = SW_{G , Y , \mathfrak{s}_Y}^{\phi_Y},
\]
where $\mathfrak{s}_Y$ is a $G$-invariant spin$^c$-structure on $Y$ and $\phi_Y$ is a chamber on $Y$.

\noindent {\bf Change of group.} Let $\psi : K \to G$ be a group homomorphism. Then $K$ acts on $X$ through $\psi$ and
\[
SW_{K , X , \mathfrak{s} }^\phi( \psi^*(\theta)) = \psi^*(SW_{G , X , \mathfrak{s}}^\phi(\theta))
\]
where $\mathfrak{s}$ is a $G$-invariant spin$^c$-structure and $\psi^* : H^*(G) \to H^*(K)$ is the pullback map.

\noindent {\bf Trivial group.} If $G$ is the trivial group then
\[
SW_{G,X,\mathfrak{s}}^\phi(x^m) = \begin{cases} SW(X,\mathfrak{s} , \phi) & \text{if } m = d(X,\mathfrak{s})/2 \\ 0 & \text{otherwise}, \end{cases}
\]
where $SW(X,\mathfrak{s} , \phi)$ denotes the ordinary Seiberg--Witten invariant of $(X , \mathfrak{s})$ with respect to the chamber $\phi$.

\noindent {\bf Finiteness.} Assume $b_+(X)^G > 1$. Then there exists only finitely many $G$-invariant spin$^c$-structures $\mathfrak{s}$ for which $SW_{G,X,\mathfrak{s}}$ is non-zero.

\noindent {\bf Charge conjugation.} For a spin$^c$-structure $\mathfrak{s}$ we let $-\mathfrak{s}$ denote the charge conjugate of $\mathfrak{s}$. There is a charge conjugation isomorphism $\psi : G_{\mathfrak{s}} \to G_{-\mathfrak{s}}$ which satisfies $\psi(u) = \overline{u}$ for all $u \in S^1$. This induces isomorphisms $H^*(G_{\mathfrak{s}}) \to H^*(G_{-\mathfrak{s}})$ and $R(G_{\mathfrak{s}}) \to R(G_{-\mathfrak{s}})$ which we shall denote by $\theta \mapsto \overline{\theta}$. Then
\[
SW_{G,X,-\mathfrak{s}}^\phi( \theta ) = (-1)^{d+b_+(X)+1} SW_{G,X,\mathfrak{s}}^{-\phi}( \overline{\theta} ).
\]

\noindent {\bf Wall crossing.} Suppose that $b_+(X)^G = 1$. Suppose that $G_{\mathfrak{s}}$ is split and choose a splitting. Let $\phi$ be a chamber and set $H_0 = H^+(X)/\langle \phi \rangle$. Assume that $G$ acts orientation preservingly on $H^+(X)$ and orient $H_0$ so that $H^+(X) \cong \mathbb{R}\phi \oplus H_0$. Then
\[
SW_{G,X,\mathfrak{s}}^\phi( x^m) - SW_{G,X,\mathfrak{s}}^{-\phi}(x^m) = e( H_0 ) s_{m-(d-1)}(D)
\]
where $e(H_0)$ denotes the Euler class of $H_0$.

\noindent {\bf Positive scalar curvature.} Suppose that $X$ admits a $G$-invariant metric $g$ of positive scalar curvature. If $b_+(X)^G > 1$ then $SW_{G,X,\mathfrak{s}} = 0$. If $b_+(X)^G = 1$ and $\phi$ is a chamber such that $\langle \phi , c(\mathfrak{s}) \rangle \le 0$, then $SW^\phi_{G,X,\mathfrak{s}} = 0$.

\noindent {\bf Mod 2 invariants for spin structures.} Suppose that $\mathfrak{s}$ is a spin structure on $X$ and the action of $G$ lifts to the principal $Spin(4)$-bundle of $\mathfrak{s}$. If $m\ge 0$ is even and $\phi$ is any chamber, then 
\[
SW^\phi_{X,\mathfrak{s}}(x^m) = w_{b_+(X)-3}( H^+(X) ) s_{m-(d-2)}(D) \in H^*(G ; \mathbb{Z}_2).
\]

\noindent {\bf Relations to families Seiberg--Witten invariants.} 

\begin{theorem}
Assume that $G_{\mathfrak{s}}$ is split and choose a splitting. Let $P \to B$ be a principal $G$-bundle over a compact smooth manifold $B$ and $E = P \times_G X$ the associated family of $4$-manifolds. Then
\[
SW^{\phi_{E}}_{E , \mathfrak{s}_{E}}( x^m ) = \varphi_P^*( SW^{\phi}_{G,X,\mathfrak{s}}( x^m))
\]
for all $m \ge 0$, where $\varphi_P : B \to BG$ is the classifying map of $P$. Moreover $SW^{\phi}_{G,X , \mathfrak{s}}$ is completely determined by the classes $SW^{\phi_{E}}_{E , \mathfrak{s}_{E}}$ as $(B , \varphi_P)$ ranges over all compact smooth manifolds and continuous maps $\varphi_P : B \to BG$.
\end{theorem}

\noindent {\bf Equivariant transversality.} 

\begin{theorem}
Let $\mathcal{L}'$ denote the collection of conjugacy classes of subgroups of $G_{\mathfrak{s}}$ whose intersection with $S^1$ is the trivial group, partially ordered by inclusion. Suppose that for each maximal element $K$ of $\mathcal{L}'$ and for each non-trivial real irreducible representation $\chi$ of $K$, we have $D_\chi \ge H^+_\chi$, then equivariant transversality can be achieved. Here for a representation $W$ of $K$, $W_\chi$ denotes the multiplicity of $\chi$ in $W$.
\end{theorem}

The precise meaning of achieving equivariant transversality is explained in Section \ref{sec:et}. Here we simply note that as a consequence we obtain a compact, smooth $G$-manifold $\mathcal{M}$ such that the equivariant Seiberg--Witten equations are given by evaluating chomology classes over $\mathcal{M}$.

Suppose in addition we also have $D_\chi > H^+_\chi$ whenever $dim_{\mathbb{R}}( Hom_K(\chi , \chi) ) = 1$. In Section \ref{sec:et} we prove that under this condition the equivariant cobordism class of $\mathcal{M}$ is well-defined. In the case that $d(X , \mathfrak{s}) = 0$ this means that $\mathcal{M}$ is a finite oriented $G$-set up to equivariant cobordism, hence defines a class $[\mathcal{M}]$ in the Burnside ring $A(G)$ of $G$. Furthermore there is a $G_\mathfrak{s}$-equivariant circle bundle $\widetilde{\mathcal{M}} \to \mathcal{M}$ whose $G_\mathfrak{s}$-equivariant cobordism class is well defined. This allows us define an invariant which takes values in a refinement $\widehat{A}_{\mathfrak{s}}(G)$ of the Burnside ring. The cohomological and $K$-theoretic equivariant Seiberg--Witten invariants can be obtained from this class.

\noindent {\bf Free actions.}

\begin{proposition}
Let $X$ be a compact, oriented, smooth $4$-manifold with $b_1(X) = 0$. Let $G$ be a finite group which acts smoothly, orientation preservingly and freely on $X$. Let $\mathfrak{s}$ be a spin$^c$-structure whose isomorphism class is fixed by $G$. If $d(X,\mathfrak{s}) \ge 0$ then equivariant transversality of the Seiberg--Witten moduli space can be achieved for $(X , \mathfrak{s})$ with respect to any chamber.
\end{proposition}

Studying the equivariant Seiberg--Witten invariants in the case of a free action, we obtain a congruence relation between the Seiberg--Witten invariants of a $4$-manifold $X$ and of its quotients. It generalises to arbitrary finite group $G$ a result of Nakamura \cite{na1,na2} for the case $G$ is cyclic of prime order.

\begin{theorem}\label{thm:swfree}
Let $X$ be a compact, oriented, smooth $4$-manifold with $b_1(X) = 0$. Let $G$ be a finite group which acts smoothly and freely on $X$. Let $\mathfrak{s}$ be a spin$^c$-structure whose isomorphism class is fixed by $G$ and suppose that $d(X,\mathfrak{s}) \ge 0$. Assume that $b_+(X)^G > 0$ and if $b_+(X)^G = 1$ then fix a chamber $\phi$. Then we have
\[
\sum_{g \in G} \sum_{s'} SW(X/\langle g \rangle , \mathfrak{s}' ) = 0 \; ({\rm mod} \; |G|)
\]
where the second sum is over spin$^c$-structures on $X/\langle g \rangle$ whose pullback to $X$ is isomorphic to $\mathfrak{s}$. If $b_+(X/H) = 1$ then $SW(X/H , \mathfrak{s}')$ is defined using the chamber $\phi$.
\end{theorem}

Note that for the Seiberg--Witten invariants $SW(X/\langle g \rangle , \mathfrak{s}')$ to be defined we need to choose orientations on $H^+(X/\langle g \rangle)$ for each $g$. In order for Theorem \ref{thm:swfree} to hold it is important to choose these orientations in a compatible way. This is explained in Section \ref{sec:far}.

\noindent {\bf Localisation theorems.} The localisation theorems are formulas which relate the equivariant Seiberg--Witten invariants $SW_{G,X,\mathfrak{s}}, SW_{G,X,\mathfrak{s}}^K$ to the reduced Seiberg--Witten invariants $\overline{SW}_{G,X,\mathfrak{s}}^{s}, \overline{SW}_{G,X,\mathfrak{s}}^{K,s}$. The localisation theorem in cohomology holds only when $G$ is cyclic of prime power order whereas the location theorem in $K$-theory applies to any finite group. This is one of the main advantages of the $K$-theoretic invariants compared to the cohomological invariants.

\begin{theorem}
Let $G = \mathbb{Z}_n$ where $n = p^k$ is a prime power. Let $H$ be the unique subgroup of $G$ of order $p$. Then
\[
SW^\phi_{G,X,\mathfrak{s}}(\theta) = \sum_{s} \overline{SW}^{s,\phi}_{H,X,\mathfrak{s}}(  e(H^+(X)/(H^+(X))^H)  e_{G_{\mathfrak{s}}}(D/D^{sH})^{-1} \theta )
\]
where the sum is over splittings $s : H \to H_{\mathfrak{s}} \cong S^1 \times H$.
\end{theorem}

The precise meaning of the right hand side will be clarified in Section \ref{sec:locc}. We turn now to the $K$-theoretic invariants. For $V \in R(G)$ and $g \in G$, we let $\chi(V,g)$ denote the character of $V$ evaluated at $g$.

\begin{theorem}
Assume that $\mathbb{R} \oplus H^+(X)$ can be given a $G$-invariant complex structure which we use to $K$-orient $H^+(X)$. For any $g \in G$, we have
\[
\chi( SW^{K,\phi}_{G,X,\mathfrak{s}}(\theta) , g ) = \chi( e^K_G(H^+/(H^+)^g) , g) \sum_{s} \chi( \overline{SW}^{K,\phi}_{\langle g \rangle , X , \mathfrak{s}}( e^K_{G_{\mathfrak{s}}}(D/D^{sg})^{-1} \theta , g)
\]
where the sum is over splittings $s : \langle g \rangle \to \langle g \rangle_{\mathfrak{s}}$.
\end{theorem}

As a consequence of the localisation theorem in $K$-theory, we obtain the following constraint on smooth group actions. This result can be regarded as a generalisation of a result of Fang \cite{fan} and Nakamura \cite{na3} in the case $G$ is cyclic of prime order.

\begin{theorem}
Let $X$ be a compact, oriented, smooth $4$-manifold with $b_1(X) = 0$. Let $G$ be a finite group acting on $X$ by orientation preserving diffeomorphisms. Let $\mathfrak{s}$ be a $G$-invariant spin$^c$-structure. Assume that $b_+(X)^G > 0$ and that $d(X,\mathfrak{s}) = 0$, hence $b_+(X)$ is odd. Suppose also that $\mathbb{R} \oplus H^+(X)$ can be equipped with a $G$-invariant complex structure. Suppose that $SW(X,\mathfrak{s}) \neq 0 \; ({\rm mod} \; |G|)$. Then for some non-trivial cyclic subgroup $\{1 \} \neq H \subseteq G$ and some splitting $s : H \to G_{\mathfrak{s}}$, we have $2 \, dim_{\mathbb{C}} ( D^{sH} ) > dim_{\mathbb{R}}( H^+(X)^H)$.
\end{theorem}

\noindent {\bf Divisibility condition.} In Section \ref{sec:divc} we prove a divisibility condition for monopole maps. The divisibility considition says that is the Seiberg--Witten invariant is non-zero mod $p$, then $(b_+-1)/2$ must be divisible by a certain power of $p$. The result can also be interpreted as a bound on the dimension of the moduli space for non-zero Seiberg--Witten invariant.

\begin{theorem}
Let $f : S^{V,U} \to S^{V',U'}$ be an ordinary monopole map over a point and let $p$ be a prime. If $SW(f) \neq 0 \; ({\rm mod} \; p)$, then $(b_+ -1)/2$ is divisible by $p^{e+1}$ whenever $p^e \le \left \lfloor \frac{m}{p-1} \right\rfloor$, where $2d-b_+-1 = 2m$. In particular, if $SW(f) \neq 0 \; ({\rm mod} \; p)$ and $b_+ \neq 1 \; ({\rm mod} \; 2p)$ then $m \le p-2$.
\end{theorem}

The above result can be applied to the ordinary Seiberg--Witten invariant of $4$-manifolds with $b_1(X) = 0$. It can also be applied to the reduced Seiberg--Witten invariants $\overline{SW}_{G,X,\mathfrak{s}}$ in which case we obtain constraints on the structure of smooth group actions on $X$. 

\noindent {\bf $\mathbb{Z}_p$-actions.} Consider the case that $G = \mathbb{Z}_p$ is cyclic of prime order $p$. In this case the localisation theorem can be worked out explicitly. We have $H^{ev}(\mathbb{Z}_p ; \mathbb{Z}_p) \cong \mathbb{Z}_p[v]$ where $deg(v) = 2$. Let $d_0, \dots , d_{p-1}$ be the dimensions of the weight spaces of the $\mathbb{Z}_p$-action on $D$. Define $c_j( n ; n_0 , \dots , n_{p-1}) \in H^*(\mathbb{Z}_p ; \mathbb{Z}_p)$ to be the coefficient of $(x+jv)^n$ in the Laurent expansion of $\prod_{i=0}^{p-1} (x+iv)^{n_i}$, see Section \ref{sec:zpact} for more explanation. 

\begin{theorem}
For any non-negative integers $m_0, \dots , m_{p-1}$, we have the following equality in $H^*(\mathbb{Z}_p ; \mathbb{Z}_p)$:
\begin{align*}
& SW^\phi_{\mathbb{Z}_p , X , \mathfrak{s}}( x^{m_0} (x+v)^{m_1} \cdots (x+(p-1)v)^{m_{p-1}}) \\
& \quad \quad = e_{\mathbb{Z}_p}( H^+/(H^+)^g ) \sum_{j=0}^{p-1} c_j\left(  -\dfrac{(b_0+1)}{2} ; m_0-d_0 , \dots , m_{p-1} - d_{p-1} \right) \overline{SW}^{s_j,\phi}_{G,X,\mathfrak{s}}.
\end{align*}
\end{theorem}

\noindent {\bf K\"ahler actions.} If $X$ is K\"ahler with K\"ahler form $\omega$ then $b_+(X)^G > 0$ and $\omega$ defines a chamber. Furthermore there is a canonical spin$^c$-structure $\mathfrak{s}_{can}$ determined by the complex structure and this sets up a bijection $L \mapsto \mathfrak{s}_L = L \otimes \mathfrak{s}_{can}$ between spin$^c$-structures and complex line bundles. Furthermore, lifting $G$ to the spinor bundles of $\mathfrak{s}_L$ is equivalent to lifting $G$ to $L$.

\begin{theorem}
Let $X$ be a compact complex surface with $b_1(X) = 0$. Let $G$ be a finite group that acts on $X$ by biholomorphisms. Let $L$ be a $G$-equivariant line bundle. If $L$ is not holomorphic, or if $L$ is holomorphic and $h^0(L) = 0$, then $SW^\omega_{G , X , \mathfrak{s}_L} = 0$. If $L$ is holomorphic and $h^0(L) > 0$, then
\begin{align*}
& SW_{G , X , \mathfrak{s}_L}^\omega( x^m ) \\
& \quad = \sum_{\substack{i+j=r \\ i,j \ge 0}} \sum_{l=0}^{i} \binom{h^1(L)-h^2(L)-l}{i-l} s_{i-l+m - (h^0(L)-1)}(V^0) c_{l}(V^1-V^2) c_{j}( H^2(X , \mathcal{O}))
\end{align*}
where $V^i = H^i(X , L)$, $d = h^0(L) - h^1(L) + h^2(L)$, $r = h^1(L) - h^2(L) + h^{2,0}(X)$. 
\end{theorem}

\noindent {\bf Gluing formulas.} We prove a gluing formula for the equivariant Seiberg--Witten invariants of an equivariant connected sum. In fact we consider a more general type of connected sum as follows. Let $Y$ be a compact, oriented, smooth $4$-manifolds with $b_1(Y) = 0$. Let $H$ be a subgroup of $G$ and suppose that $H$ acts smoothly and orientation preservingly on $Y$. Suppose that there is an $x \in X$ and $y \in Y$ with stabiliser group $H$ and an orientation reversing isomorphism $\psi : T_x X \to T_y Y$ of representations of $H$. Using $\psi$ we can form the connected sum of $X$ with $|G/H|$ copies of $Y$ equivariantly, which we denote by $Z = X \# ind^G_H(Y)$. Suppose that $\mathfrak{s}_X$ is a $G$-invariant spin$^c$-structure on $X$ and $\mathfrak{s}_Y$ is a $H$-invariant spin$^c$-structure on $Y$. Then as explained in Section \ref{sec:ecs} we obtain a $G$-invariant spin$^c$-structure $\mathfrak{s} = \mathfrak{s}_X \# ind^G_H(\mathfrak{s}_Y)$ by attaching $\mathfrak{s}_X$ to $|G/H|$ copies of $\mathfrak{s}_Y$.

\begin{theorem}
Let $(X , \mathfrak{s}_X), (Y , \mathfrak{s}_Y)$ be as above. Then:
\begin{itemize}
\item[(1)]{If $b_+(X)^G, b_+(Y)^G$ are both non-zero then the equivariant Seiberg--Witten invariants of $Z$ vanish.}
\item[(2)]{If $b_+(X)^G \neq 0$ and $b_+(Y)^H = 0$, then for any chamber $\phi \in H^+(X)^G \setminus \{0\}$ we have
\[
SW^\phi_{G , Z , \mathfrak{s}}( \theta ) = SW^\phi_{G , X , \mathfrak{s}_X}( e(ind^G_H(H^+(Y))) s_{G_{\mathfrak{s}} , -d_Y}( ind^{G_{\mathfrak{s}}}_{H_{\mathfrak{s}}}( D_Y ) ) \theta ).
\]
}
\end{itemize}

\end{theorem}

This generalises the gluing formula given in \cite{bk1} in the $G$-equivariant case. We also have a gluing formula in the case $b_+(X)^G = 0$, $b_+(Y)^H \neq 0$. However this case is more difficult to give a general formula, so for simplicity we restrict to the case $G = \mathbb{Z}_p$ for a prime $p$ and $H = 1$. Note that $H^*(\mathbb{Z}_p ; \mathbb{Z}) \cong \mathbb{Z}[v]/(pv)$, where $deg(v) = 2$.

\begin{theorem}
Let $(X , \mathfrak{s}_X), (Y , \mathfrak{s}_Y)$ be as above, where $G = \mathbb{Z}_p$ and $H=1$. Suppose that $b_+(X)^G = 0$ and $b_+(Y) > 0$. Assume further that $d(Y,\mathfrak{s}_Y) = 0$. Then
\[
SW_{G , Z , \mathfrak{s}}^\phi( x^m ) = (-1)^{d_Y+1} h {\sum_l}' e(H^+(X)) s_{-d_X-l}(D_X) SW(Y , \mathfrak{s}_Y , \phi_Y) v^{m+l - (p-1)/2}
\]
where the sum ${\sum_l}'$ is over $l$ such that $0 \le l \le -d_X$, $m+l > 0$, $m+l = 0 \; ({\rm mod} \; p-1)$ and $h = \prod_{j=1}^{(p-1)/2} j^{b_+(Y)}$ for $p \neq 2$, $h=1$ for $p=2$.
\end{theorem}

\subsection{Structure of the paper}

In Section \ref{sec:esw} we introduce the equivariant Seiberg--Witten invariants. We introduce the notion of an abstract equivariant monopole map in Section \ref{sec:emm}. To such a map we define cohomological and $K$-theoretic Seiberg--Witten invariants in Sections \ref{sec:ci} and \ref{sec:ki} respectively. In Section \ref{sec:prop} we prove some basic properties of these invariants including wall-crossing (\textsection \ref{sec:wcf}), positive scalar curvature (\textsection \ref{sec:psc}) and mod $2$ invariants for spin structures (\textsection \ref{sec:mod2}). In Section \ref{sec:eqfam}) we establish the relation between the equivariant Seiberg--Witten invariants and the families Seiberg--Witten invariants. In Section \ref{sec:eswm} we consider the problem of constructing smooth equivariant moduli spaces. We find sufficient conditions under which equivariant transversality can be achieved (\textsection \ref{sec:et}). In the zero-dimensional case we use equivariant transversality to obtain an invariant valued in a refinement of the Burnside ring (\textsection \ref{sec:zero}). We prove that equivariant transversality can always be achieved for free actions and use this to prove a congruence formula relating the Seiberg--Witten invariants of a $4$-manifold $X$ and its quotients under some restrictions on the group (\textsection \ref{sec:trfree}). In Section \ref{sec:red} we introduce the reduced Seiberg--Witten invariants of a group action. We then use localisation in cohomology (\textsection \ref{sec:locc}) and $K$-theory (\textsection \ref{sec:lock}) to relate the equivariant and reduced invariants. As an application, we obtain another congruence formula for the Seiberg--Witten invariants of a $4$-manifold $X$ and its quotients under a free action without restriction on the group (\textsection \ref{sec:far}). We also use localisation to obtain a constraint on smooth group actions (\textsection \ref{sec:constraint}) and to give an explicit formula relating the $\mathbb{Z}_p$-equivariant and reduced Seiberg--Witten invariants (\textsection \ref{sec:zpact}). In Section \ref{sec:kahler} we prove a formula for the equivariant Seiberg--Witten invariants in the case where $X$ is a K\"ahler surface and the action preserves the K\"ahler structure. In Section \ref{sec:glue} we prove a connected sum formula for the equivariant Seiberg--Witten invariants. We first prove an abstract gluing formula for the Seiberg--Witten invariants of a smash product of monopole maps (\textsection \ref{sec:agf}) and then apply this to the case of connected sums (\textsection \ref{sec:ecs}). We finish with some examples to illustrate the main result of the paper in Section \ref{sec:ex}.

\noindent{\bf Acknowledgments.} The author was financially supported by an Australian Research Council Future Fellowship, FT230100092.

\subsection{Orientation conventions}

Our orientation conventions are such that push-forward maps in cohomology are {\em right} module homomorphisms $p_*( a \smallsmile p^*(b) ) = p_*(a) \smallsmile b$. This means for example whenever we apply the Thom isomorphism, the Thom class goes on the left: $a \mapsto \tau \smallsmile p^*(a)$, where $\tau$ is the Thom class. For a fibre bundle $p : E \to B$ with compact fibres, $p_*$ is given by evaluation over the fibres, where $E,B$ and $T(E/B) = Ker(p_*)$ are oriented so that $TE = T(E/B) \oplus p^*(TB)$ (vertical bundle first). For an embedding $\iota : Y \to X$, $\iota_*(1)$ is Poincar\'e dual to $Y$, where $X,Y$ and the normal bundle $N$ are oriented so that $TX|_Y = N \oplus TY$ (normal bundle first).

\section{Equivariant Seiberg--Witten invariants}\label{sec:esw}

In this section we define the equivariant Seiberg--Witten invariants for group actions on smooth $4$-manifolds. In order to prove various properties of the invariants it is convenient to study them in a more abstract setting. Therefore we introduce an abstract notion of monopole maps and their Seiberg--Witten invariants.

\subsection{Equivariant monopole maps}\label{sec:emm}

We introduce an abstract formalism for equivariant monopole maps. Let $B$ be a compact smooth manifold. Let $G$ be a finite group which acts smoothly on $B$. Consider an $S^1$-central extension
\[
1 \to S^1 \to G_{\mathfrak{s}} \to G \to 1.
\]
Suppose that $V,V' \to B$ are complex vector bundles over $B$ and $U,U' \to B$ are real vector bundles over $B$. Assume that $V,V',U,U'$ are $G_{\mathfrak{s}}$-equivariant, where the $S^1$ subgroup of $G_{\mathfrak{s}}$ acts by scalar multiplication on $V,V'$ and acts trivially on $U,U'$.

\begin{definition}\label{def:mm}
An {\em abstract $G$-monopole map over $B$} consists of a central extension $G_{\mathfrak{s}}$, vector bundles $V,V',U,U'$ as above and a $G_{\mathfrak{s}}$-equivariant map
\[
f : S^{V,U} \to S^{V',U'}
\]
where $S^{V , U}$ denotes the fibrewise one point compactification of $V\oplus U$ and $S^{V',U'}$ is defined similarly, satisfying the following conditions:
\begin{itemize}
\item[(1)]{$f$ sends the section of $S^{V,U}$ at infinity to the section of $S^{V',U'}$ at infinity.}
\item[(2)]{The restriction of $f$ to $U$ defines a linear injection $f|_U : U \to U'$.}
\end{itemize}

When $G = \{1\}$, $f$ is an $S^1$-equivariant map and we will refer to $f$ as an {\em ordinary monopole map over $B$} in such cases.

\end{definition}

We will be mainly interested in the case that $B$ is a point, in which case $f$ is simply a map of spheres. However the general formalism will be occasionally useful. For the Bauer--Furuta map of $4$-manifolds with $b_1(X) > 0$, one takes $B$ to be a torsor for the Jacobian torus of $X$.

By condition (2) of Definition \ref{def:mm}, we can identify $U$ with a subbundle of $U'$ and then $f|_U$ is given by the inclusion map $U \to U'$. Choose an equivariant splitting $U' \cong U \oplus H^+$, where $H^+ = U'/U$. This is always possible by choosing a $G$-invariant metric on $U'$ and taking $H^+$ to be the orthogonal complement of $U$. We let $D$ denote the equivariant virtual vector bundle $D = V - V'$. Let $a,a'$ denote the complex dimensions of $V,V'$ and $b,b'$ denote the real dimensions of $U,U'$. The (real) dimension of $H^+$ is $b^+ = b'-b$ and the virtual (complex) dimension of $D$ is $d = a-a'$.

Our interest in abtract monopoles is that they arise as the Bauer--Furuta invariants of smooth $4$-manifolds. Let $X$ be a compact oriented smooth $4$-manifold with $b_1(X)=0$ and suppose that a finite group $G$ acts on $X$ by diffeomorphism. Suppose that $\mathfrak{s}$ is a spin$^c$-structure whose isomorphism class is preserved by $G$. By taking a finite dimensional approximation of the Seiberg--Witten equations for $(X,\mathfrak{s})$, we obtain a $G$-monopole map $f : S^{V,U} \to S^{V',U'}$ over $B = pt$ \cite{bf,szy,bar}. In this case $D = V-V'$ is the index of the spin$^c$ Dirac operator associated to $\mathfrak{s}$ and $H^+ = H^+(X)$ is the space of harmonic self-dual $2$-forms on $X$ (with respect to a choice of $G$-invariant metric). The group $G_{\mathfrak{s}}$ can be defined as follows. Let $P_{\mathfrak{s}} \to X$ denote the principal $Spin^c(4)$-bundle associated to $\mathfrak{s}$. Choose a spin$^c$-connection $A$ on $P_{\mathfrak{s}}$ whose curvature $F_A$ is $G$-invariant (this is always possible by averaging over $G$). Then $G_{\mathfrak{s}}$ is the group of lifts of elements of $G$ to $P_{\mathfrak{s}}$ which preserve $A$. The homotopy class of $f$ depends on the choice of finite dimensional approximation, but the stable equivariant cohomotopy class of $f$ is an invariant, the {\em $G$-equivariant Bauer--Furuta invariant of $(X,\mathfrak{s})$ (with respect to the given $G$-action)}. One of the main objectives of this paper is to extract invariants from $f$ which are easier to work with than the stable equivariant cohomotopy class of $f$ itself (which is an element of a group which is generally very difficult to calculate). In the case $b_1(X) > 0$ it is still possible to define equivariant Bauer--Furuta invariants, but the details are more complex. In this paper we will restrict to the case $b_1(X) = 0$ for simplicity.

In the following sections we define a series of invariants of $G$-monopole maps. The invariants that we consider will come in two types: cohomological and $K$-theoretic and they will generally depend on an additional choice known as a chamber.

\begin{definition}
Let $f : S^{V,U} \to S^{V',U'}$ be a $G$-monopole map. A {\em chamber} for $f$ is a homotopy class of $G$-equivariant map $\phi : B \to S(H^+)$ from $B$ to the unit sphere bundle of $H^+$. If $B = pt$ is a point, then $\phi$ is equivalent to specifying a path-component of $S(H^+)$.
\end{definition}

Observe that when $B = pt$ there are three possibilities depending on the dimension $b_+^G = dim_{\mathbb{R}}( (H^+)^G )$ of the $G$-invariant subspace of $H^+$. If $b_+^G = 0$ then there are no chambers. If $b_+^G = 1$, then there are two chambers. If $\phi$ is one chamber then $-\phi$ is the other. If $b_+^G > 1$, then there is a unique chamber.

\subsection{Cohomological invariants}\label{sec:ci}

In general, to define cohomological invariants we will need to use equivariant local systems. Suppose $G$ acts on a space $M$. A class $\alpha \in H^1_G(M ; \mathbb{Z}_2)$ defines a $G$-equivariant principal $\mathbb{Z}_2$-bundle $M_\alpha \to M$. Then we obtain an equivariant local system $\mathbb{Z}_\alpha = M_\alpha \times_{\mathbb{Z}_2} \mathbb{Z}$. More generally, for any equivariant local system $A$, we set $A_\alpha = A \otimes_{\mathbb{Z}} \mathbb{Z}_\alpha$. If $W \to M$ is a real, finite rank equivariant vector bundle, then its equivariant first Stiefel--Whitney class $w_1(W) \in H^1_G(M ; \mathbb{Z}_2)$ defines an equivariant local system $\mathbb{Z}_{w_1(W)}$. We sometimes write $\mathbb{Z}_W$ for $\mathbb{Z}_{w_1(W)}$. In the case $M = pt$, a $G$-equivariant local system is an abelian group $A$ equipped with an action of $G$, i.e. a $G$-module.

Let $f : S^{V,U} \to S^{V',U'}$ be a $G$-monopole map. A chamber $\phi$ for $f$ defines a refined Thom class 
\[
\tau^\phi_{V',U'} \in H^{2a'+b'}_{G_{\mathfrak{s}}}( S^{V',U'} , S^{U} ; \mathbb{Z}_{U'}).
\]
This was shown for families in \cite[\textsection 3]{bk} and the same construction works in the equivariant setting. The pullback of $\tau^{\phi}_{V',U'}$ is valued in $H^{2a'+b'}_{G_{\mathfrak{s}}}( S^{V,U} , S^{U} ; \mathbb{Z}_{U'})$, which by excision may be identified with $H^{2a'+b'}_{G_{\mathfrak{s}}}( \widetilde{Y} , \partial \widetilde{Y} ; \mathbb{Z}_{U'})$, where $\widetilde{Y}$ is the complement in $S^{V,U}$ of a $G_{\mathfrak{s}}$-invariant tubular neighbourhood of $S^{U}$. As in \cite[Section 3]{bk} one can show that $f^*(\tau^\phi_{V',U'})$ is in the image of the coboundary $\delta : H^*_{G_{\mathfrak{s}}}( \partial \widetilde{Y} ; \mathbb{Z}_{U'} ) \to H^{*+1}_{G_{\mathfrak{s}}}( \widetilde{Y} , \partial \widetilde{Y} ; \mathbb{Z}_{U'})$. We have $\widetilde{Y} = S(V) \times S^U$ and so we find that
\begin{equation}\label{equ:eta}
f^*( \tau^\phi_{V',U'} ) = (-1)^b \delta \tau_U \eta^\phi
\end{equation}
for some
\[
\eta^\phi \in H^{2a'+b^+-1}_{G_{\mathfrak{s}}}( S(V) ; \mathbb{Z}_{w} ) \cong H^{2a'+b^+-1}_{G}( \mathbb{P}(V) ; \mathbb{Z}_{w})
\]
and where $w = w_1(H^+)$. The sign factor $(-1)^b$ is chosen for a subtle reason related to the process of orienting the Seiberg--Witten moduli space. It makes $\eta^\phi$ behave nicely under suspension. Let $\pi : \mathbb{P}(V) \to B$ denote the projection map. This induces a pushforward
\[
\pi_* : H^{*}_{G}( \mathbb{P}(V) ; A ) \to H^{*-2a+2}_{G}( B ; A )
\]
for any equivariant local system $A$. 

\begin{definition}\label{def:swc}
Let $f : S^{V,U} \to S^{V',U'}$ be a $G$-monopole map over $B$, $\phi : B \to S(H^+)$ a chamber for $f$ and $A$ a $G$-module. The {\em (cohomological) $G$-equivariant Seiberg--Witten invariant of $f$ with respect to the chamber $\phi$ and with coefficient group $A$} is the map
\[
SW_{G,f}^\phi : H^*_{G_{\mathfrak{s}}}(B ; A) \to H^{*-\delta}_G(B ; A_w)
\]
where $\delta = 2d-b_+-1$, given by
\[
SW^\phi_{G,f}( \theta ) = \pi_*( \eta^\phi \pi^*(\theta))
\]
where by $\pi^*(\theta)$ we mean the pullback of $\theta$ to $H^*_{G_\mathfrak{s}}( S(V) ; A) \cong H^*_G( \mathbb{P}(V) ; A)$.
\end{definition}

When $G$ acts orientation preservingly on $H^+$, we have that $w = w_1(H^+) = 0$. A choice of orientation on $H^+$ determines a trivialisation $\mathbb{Z}_w \cong \mathbb{Z}$ and the $G$-equivariant Seiberg--Witten invariants (with coefficient group $A$) are valued in $H^{*-\delta}_G(B ; A)$. To make this construction completely unambiguous we need to establish orientation conventions. Suppose $G$ acts orientation preservingly on $H^+$ and fix a choice of orientation. Suspending $f$ if necessary, we can assume $G$ acts orientation preservingly on $U$ and we fix a choice of orientation on $U$. Then we can also orient $U'$ according to the convention that $U' \cong U \oplus H^+$. The representations $V,V'$ are complex, so they carry natural orientations. We now have orientations on $S^{V',U'}$ and $S^U$ which we use to define the Thom classes $\tau^\phi_{V',U'}, \tau_{U}$. If we change orientation on $U$, then since $U' \cong U \oplus H^+$, the orientation on $U'$ also changes. Hence $\tau^\phi_{V',U'}$ and $\tau_U$ change sign, but $\eta^\phi$ remains unchanged.

By a straightforward computation (c.f. \cite[Proposition 3.8]{bk}) the invariants $SW_{G,f}^{\phi}(\theta)$ do not change under suspension $f \mapsto id_{V'' \oplus U''} \wedge f$ (note that the sign factor $(-1)^b$ in Equation (\ref{equ:eta}) is needed for this to hold).

If $H$ is a subgroup of $G$ and we write $SW^\phi_{H,f}$, then it means that we are restricting the $G_{\mathfrak{s}}$-action to the subgroup $H_{\mathfrak{s}}$, the preimage of $H$ under $G_{\mathfrak{s}} \to G$. The invariants $SW^\phi_{G,f}$ and $SW^\phi_{H,f}$ are related by a commutative square
\[
\xymatrix{
H^*_{G_{\mathfrak{s}}}(B ; A) \ar[r]^-{SW^\phi_{G,f}} \ar[d] & H^{*-\delta}_G(B ; A_w) \ar[d] \\
H^*_{H_{\mathfrak{s}}}(B ; A) \ar[r]^-{SW^\phi_{H,f}} & H^{*-\delta}_H(B ; A_w)
}
\]
where the vertical maps are restriction maps. If we write $SW^\phi_{f}$, then it means we are restricting to $1_{\mathfrak{s}} = S^1$.

When $f$ is the monopole map of a $4$-manifold $X$ satisfying $b_1(X)=0$, equipped with a $G$-action and a $G$-invariant spin$^c$-structure $\mathfrak{s}$, we write the $G$-equivariant Seiberg--Witten invariant of $f$ as $SW_{G,X,\mathfrak{s}}^{\phi}$. Thus $SW_{G,X,\mathfrak{s}}^{\phi}$ (with integral coefficients) is a map
\[
SW_{G,X,\mathfrak{s}}^\phi : H^*_{G_{\mathfrak{s}}}(pt ; \mathbb{Z}) \to H^{*-d(X,\mathfrak{s})}_G(pt ; \mathbb{Z}_w)
\]
where $w = w_1(H^+(X))$ and $d(X,\mathfrak{s}) = 2d - b_+ - 1$ is the expected dimension of the Seiberg--Witten moduli space for $(X,\mathfrak{s})$. An important feature of these invariants is that they can be non-zero even when the expected dimension $d(X,\mathfrak{s})$ is negative.

If the pair $(X,\mathfrak{s})$ or the chamber $\phi$ is understood we may sometimes omit them from the notation and write $SW_{G,X,\mathfrak{s}}, SW_G^\phi$, etc. 

From Definition \ref{def:swc}, it is immediate that $SW_{G,f}^\phi$ is a morphism of right $H^*_{G}(pt ; \mathbb{Z})$-modules in the sense that $SW^\phi_{G,f}( \theta \lambda ) = SW^\phi_{G,f}(\theta) \lambda$ for all $\theta \in H^*_{G_{\mathfrak{s}}}(pt ; A), \lambda \in H^*_G(pt ; \mathbb{Z})$. In fact the same is true more generally when $\lambda$ is valued in a coefficient group other than $\mathbb{Z}$. Suppose now that the extension $G_{\mathfrak{s}}$ is split and choose a splitting $s : G \to G_{\mathfrak{s}}$. This induces an isomorphism $G_{\mathfrak{s}} \cong S^1 \times G$ and hence an isomorphism $\psi_s : H^*_{G_{\mathfrak{s}}}(pt ; A) \cong H^*_G(pt ; A)[x]$, where $x$ is the generator of $H^2_{S^1}( pt ; \mathbb{Z})$ given by $x = c_1( \mathbb{C}_1 )$, where $\mathbb{C}_1$ denotes the $1$-dimensional representation of $S^1$ where $S^1$ acts with weight $1$. In this case $SW^\phi_{G,f}$ is completely determined by the classes 
\[
SW^\phi_{G,f}(x^m) \in H^{2m - d(X,\mathfrak{s})}_G(pt ; \mathbb{Z}_w), \; m \ge 0.
\]
If $s' : G \to G_{\mathfrak{s}}$ is another splitting then $s'(g) = \chi(g)s(g)$ for some homomorphism $\chi : G \to S^1$. Note that $\chi$ defines a class in $H^1(G ; S^1) \cong H^2(G ; \mathbb{Z})$. A simple calculation shows that the class $x$ defined by the splitting $s$ corresponds to $x + \chi$ under the splitting $s'$, in other words, $\psi_{s'}( \psi^{-1}_s(x) ) = x + \chi$.

\begin{remark}
Suppose that for some $n \ge 0$, we have that $SW^\phi_{G,f}( x^m ) = 0$ for $m < n$. Then $SW^\phi_{G,f}( (x + \chi)^n ) = SW^\phi_{G,f}( x^n )$ for all $\chi \in H^2(G ; \mathbb{Z})$. Hence the invariant $SW^\phi_{G,f}(x^n)$ does not depend on the choice of splitting. In particular, $SW^\phi_{G,f}(1)$ is always independent of the choice of splitting (in fact $SW^\phi_{G,f}(1)$ is defined without requiring a splitting to exist at all).
\end{remark}

In order to facilitate the computation of $SW^\phi_{G,f}$, it will be useful to understand in more detail the push-forward map 
\[
\pi_* : H^{*}_{G}( \mathbb{P}(V) ; A ) \to H^{*-2a+2}_{G}( B ; A ).
\]
Observe that $H^*_G( \mathbb{P}(V) ; A) \cong H^*_{G_{\mathfrak{s}}}( S(V) ; A)$. Since $S(V) \to B$ is a sphere bundle, the cohomology of the fibres is concentrated in degrees $0$ and $2a-1$. The Serre spectral sequence for $S(V) \to B$ reduces to a long exact sequence, the {\em Gysin sequence} (suppressing coefficients):
\[
\cdots \to H^i_{G_{\mathfrak{s}}}(B) \buildrel \pi^* \over \longrightarrow H^i_{G_{\mathfrak{s}}}( S(V) ) \buildrel \pi_* \over \longrightarrow H^{i-(2a-1)}_{G_{\mathfrak{s}}}(B) \buildrel e(V) \over \longrightarrow H^{i+1}_{G_{\mathfrak{s}}}(B) \to \cdots
\]
where $e(V) \in H^{2a}_{G_{\mathfrak{s}}}(B ; \mathbb{Z})$ denotes the $G_{\mathfrak{s}}$-equivariant Euler class of $V$. Suppose that $G_{\mathfrak{s}} \to G$ splits. Upon choosing a splitting, we have
\[
e(V) = x^a + x^{a-1} c_1(V) + \cdots + c_a(V) \in H^*_{G_{\mathfrak{s}}}(B ; \mathbb{Z}) \cong H^*_G(B ; \mathbb{Z})[x].
\]
Multiplication by $e(V)$ is clearly injective, and so
\[
H^*_{G}( \mathbb{P}(V) ; A) \cong H^*_{G_{\mathfrak{s}}}( S(V) ; A) \cong H^*_{G_{\mathfrak{s}}}(B ; A)/ ( e(V) ) \cong H^*_G(B ; A)[x]/ ( e(V) ).
\]
The pushforward map is easily calculated: $\pi_*( x^j ) = 0$ for $j < a$ and $\pi_*( x^{a-1} ) = 1$. Then it follows from the formula for $e(V)$ that $\pi_*( x^{k+(a-1)}) = s_k(V)$ for $k \ge 0$, where $s_k(V)$ is the $k$-th $G$-equivariant Segre class of $V$ \cite[Proposition 3.5]{bk}. We recall that the Segre classes $s_k(V)$ of a vector bundle $V$ are defined by $c(V)s(V) = 1$, where $c(V) = 1 + c_1(V) + c_2(V) + \cdots $ is the total Chern class and $s(V) = 1 + s_1(V) + \cdots $ is the total Segre class.

\subsection{$K$-theoretic invariants}\label{sec:ki}

We introduce the $K$-theoretic invariants of a $G$-equivariant monopole map. For simplicity we will avoid the use of twisted $K$-theory. Thus, to define pushforward maps in equivariant complex $K$-theory we require that spaces are $K$-oriented. A $G$-equivariant $K$-orientation on a $G$-equivariant vector bundle $V$ is a $G$-equivariant spin$^c$-structure. A $G$-equivariant $K$-orientation on a smooth manifold $M$ is a $G$-equivariant $K$-orientation on $TM$. If $V$ is a complex equivariant vector bundle, then it has a canonical spin$^c$-structure determined by the complex structure, which we call the {\em canonical spin$^c$-structure}. If $V$ is a complex equivariant vector bundle equipped with its canonical spin$^c$-structure, then its equivariant $K$-theoretic Euler class is given by
\[
e_G^K(V) = \sum_{j=0}^{r} (-1)^r \wedge^j V^*
\]
where $r = rank(V)$.

Consider a $G$-equivariant monopole map $f : S^{V,U} \to S^{V',U'}$. To define $K$-theoretic invariants we will require that $V,V',U,U'$ are $K$-oriented. Since $V,V'$ are complex vector bundles we can equip them with their canonical spin$^c$-structures. By suspending if necessary, we can assume $U$ admits a $G$-equivariant complex structure (indeed we can replace $f$ by $id_U \wedge f  : S^{V , U \oplus U} \to S^{V' , U \oplus U'}$. Clearly $U \oplus U$ has a $G$-invariant complex structure). Then since $U' \cong U \oplus H^+$, to $K$-orient $U'$, it suffices to $K$-orient $H^+$. Let $\phi : B \to S(H^+)$ be a chamber for $f$. Then we have a decomposition $H^+ \cong \mathbb{R}\phi \oplus H_0$, where $H_0 \cong H^+/\langle \phi \rangle$. The trivial line bundle $\mathbb{R}\phi$ may be given the trivial spin$^c$-structure, and thus in the presence of a chamber a $K$-orientation on $H^+$ is equivalent to choosing a $K$-orientation on $H_0$. 

Assume that a $K$-orientation for $H^+$ has been chosen. Then similar to the cohomological case, the chamber $\phi$ defines a refined $K$-theoretic Thom class
\[
\tau^{K,\phi}_{V',U'} \in K^{2a'+b'}_{G_{\mathfrak{s}}}( S^{V',U'} , S^{U} ).
\]
The pullback of $\tau^{K,\phi}_{V',U'}$ by $f$ has the form
\[
f^*( \tau^{K,\phi}_{V',U'} ) = \delta \tau^K_U \eta^{K,\phi}
\]
where
\[
\eta^{K,\phi} \in K^{2a'+b^+-1}_{G_{\mathfrak{s}}}( S(V) ) \cong K^{2a'+b^+-1}_{G}( \mathbb{P}(V)).
\]
Let $\pi : \mathbb{P}(V) \to B$ denote the projection map. This induces a pushforward
\[
\pi_* : K^{*}_{G}( \mathbb{P}(V) ) \to K^{*-2a+2}_{G}( B ).
\]

\begin{definition}
Let $f : S^{V,U} \to S^{V',U'}$ be a $G$-monopole map over $B$, $\phi : B \to S(H^+)$ a chamber for $f$. Let $\mathfrak{o}$ denote an equivariant spin$^c$-structure on $H^+$. The {\em $K$-theoretic $G$-equivariant Seiberg--Witten invariant of $f$ with respect to the chamber $\phi$ and spin$^c$-structure $\mathfrak{o}$} is the map
\[
SW_{G,f}^{K,\phi , \mathfrak{o}} : K^*_{G_{\mathfrak{s}}}(B) \to K^{*-\delta}_G(B)
\]
where $\delta = 2d-b_+-1$, given by
\[
SW^{K,\phi , \mathfrak{o}}_{G,f}( \theta ) = \pi_*( \eta^{K,\phi} \pi^*(\theta))
\]
where by $\pi^*(\theta)$ we mean the pullback of $\theta$ to $K^*_{G_\mathfrak{s}}( S(V)) \cong K^*_G( \mathbb{P}(V))$.
\end{definition}

The case of most interest to us will be when $B = pt$ is a point and $b_+$ is odd. Recall that $K^0_G(pt) \cong R(G)$, where $R(G)$ denotes the representation ring of $G$. Then the $K$-theoretic equivariant Seiberg--Witten invariant of $f$ takes the form
\[
SW^{K,\phi,\mathfrak{o}}_{G,f} : R(G_\mathfrak{s}) \to R(G).
\]
In the split case, a choice of splitting $G_\mathfrak{s} \cong S^1 \times G$ yields an isomorphism $R(G_\mathfrak{s}) \cong R(G)[\xi,\xi^{-1}]$, where $\xi$ denotes the standard $1$-dimensional complex representation of $S^1$ and $SW^{K,\phi,\mathfrak{o}}_{G,f}$ is an $R(G)$-module homomorphism $R(G)[\xi,\xi^{-1}] \to R(G)$. It is completely determined by its values on powers of $\xi$:
\[
SW^{K , \phi , \mathfrak{o}}_{G,f}(\xi^m) \in R(G),  \; m \in \mathbb{Z}.
\]

Given any lift to $Spin^c(n)$, any other lift differs by tensoring by a complex equivariant line bundle, hence by an element of $Hom(G , S^1) = H^1(G ; S^1) \cong H^2(G ; \mathbb{Z})$. This will alter $SW_{G,f}^{K , \phi , \mathfrak{o}}$ as follows. Lett $\mathfrak{o}$ be a spin$^c$-structure on $H_0$ and $\chi$ a $1$-dimensional representation of $G$, then
\[
SW_{G,f}^{K, \phi , \chi \otimes \mathfrak{o}}( \lambda ) = \chi \otimes SW_{G,f}^{K , \phi , \mathfrak{o}}(\lambda).
\]

Generally, we will omit the choice of $\mathfrak{o}$ from the notation (just as we omit the choice of orientation on $H^+(X)$ from the notation for the cohomological invariants) and write $SW^{K,\phi}_{G,f}$. In the case that $f$ is the $G$-equivariant Bauer--Furuta invariant of $(X,\mathfrak{s})$ with respect to a given $G$-action, where $X$ is a compact, oriented, smooth $4$-manifold with $b_1(X) = 0$, $b_+(X)$ odd and $\mathfrak{s}$ is a $G$-invariant spin$^c$-structure on $X$, we write the $K$-theoretic Seiberg--Witten invariants of $(X,\mathfrak{s})$ as
\[
SW^{K,\phi}_{G,X,\mathfrak{s}} : R(G_\mathfrak{s}) \to R(G).
\]

\section{Properties of the invariants}\label{sec:prop}

In this section we establish various fundamental properties of the cohomological and $K$-theoretic invariants.

\begin{theorem}\label{thm:propb}
The equivariant Seiberg--Witten invariants satisfy the following properties:

\begin{itemize}
\item[(1)]{(Diffeomorphism invariance) Let $X,Y$ be compact oriented smooth $4$-manifolds with $b_1 = 0$ and let $G$ be a finite group which acts on $X$ and $Y$ by orientation preserving diffeomorphisms. Let $f : X \to Y$ be a $G$-equivariant diffeomorphism. Then
\[
SW_{G , X , f^*(\mathfrak{s}_Y)}^{f^*(\phi_Y)} = SW_{G , Y , \mathfrak{s}_Y}^{\phi_Y}
\]
where $\mathfrak{s}_Y$ is a $G$-invariant spin$^c$-structure on $Y$ and $\phi_Y$ is a chamber on $Y$.}
\item[(2)]{(Change of group) Suppose $G$ acts on $X$. Let $\psi : K \to G$ be a group homomorphism. Then $K$ acts on $X$ through $\psi$ and
\[
SW_{K , X , \mathfrak{s} }^\phi( \psi^*(\theta)) = \psi^*(SW_{G , X , \mathfrak{s}}^\phi(\theta)),
\]
where $\mathfrak{s}$ is a $G$-invariant spin$^c$-structure and $\psi^* : H^*_G(pt) \to H^*_K(pt)$ is the pullback map in equivariant cohomology.} 
\item[(3)]{(Trivial group). If $G$ is the trivial group then
\[
SW_{G,X,\mathfrak{s}}^\phi(x^m) = \begin{cases} SW(X,\mathfrak{s},\phi) & \text{if } m = d(X,\mathfrak{s})/2 \\ 0 & \text{otherwise}, \end{cases}
\]
where $SW(X,\mathfrak{s},\phi)$ denotes the ordinary Seiberg--Witten invariant of $(X , \mathfrak{s})$ with respect to the chamber $\phi$.
}
\item[(4)]{(Finiteness) Assume $b_+(X)^G > 1$. Then there exists only finitely many $G$-invariant spin$^c$-structures $\mathfrak{s}$ for which $SW_{G,X,\mathfrak{s}}$ is non-zero.}
\end{itemize}

Similar results hold for the $K$-theoretic equivariant Seiberg--Witten invariants.

\end{theorem}
\begin{proof}
(1)-(3) are straighforward. (4) follows by the same argument as for the ordinary Seiberg--Witten invariants. Namely, that for a fixed metric $g$ and $2$-form perturbation $\eta$, there are only finitely many $\mathfrak{s}$ for which the Seiberg--Witten equations for $(X , \mathfrak{s} , g , \eta)$ have a solution.
\end{proof}

Recall \cite[Page 51]{nic} that to each spin$^c$-structure on a $4$-manifold $X$, there is another spin$^c$-structure $-\mathfrak{s}$ called the {\em charge conjugate} spin$^c$-structure, which satisfies $c(-\mathfrak{s}) = -c(\mathfrak{s})$. There is also a conjugation isomorphism $\psi : G_{\mathfrak{s}} \to G_{-\mathfrak{s}}$ which satisfies $\psi(u) = \overline{u}$ for all $u \in S^1$. This induces isomorphisms $H^*_{G_{\mathfrak{s}}}(pt ; A) \to H^*_{G_{-\mathfrak{s}}}(pt ; A)$ and $R(G_{\mathfrak{s}}) \to R(G_{-\mathfrak{s}})$ which we shall denote by $\theta \mapsto \overline{\theta}$. In the case of a split extension $G_{\mathfrak{s}} \cong S^1 \times G$, this isomorphism satisfies $\overline{x} = -x$.

\begin{proposition}[Charge conjugation]
Let $X$ be a compact, oriented, smooth $4$-manifold with $b_1(X) = 0$ and let $G$ be a finite group acting on $X$ by orientation preserving diffeomorphism. Let $\mathfrak{s}$ be a $G$-invariant spin$^c$-structure. Then we have the following equality (known as charge conjugation symmetry):
\[
SW_{G,X,-\mathfrak{s}}^\phi( \theta ) = (-1)^{d+b_+(X)+1} SW_{G,X,\mathfrak{s}}^{-\phi}( \overline{\theta} ).
\]
If $b_+(X)$ is odd and $\mathfrak{o}$ is a $G$-invariant spin$^c$-structure on $H^+(X)$ then we also have
\[
SW_{G,X,-\mathfrak{s}}^{K,\phi , \mathfrak{o}}( \theta ) = (-1)^{d} SW_{G,X,\mathfrak{s}}^{K,-\phi , \mathfrak{o}}(\overline{\theta}).
\]
\end{proposition}
\begin{proof}
Using Theorem \ref{thm:eqfam2}, the result follows from the corresponding result for the families Seiberg--Wittten invariant, which itself follows by a straighforward generalisation of \cite[Proposition 2.2.26]{nic} to the families setting.
\end{proof}

\subsection{Wall crossing formula}\label{sec:wcf}

\begin{theorem}
Let $X$ be a compact, oriented smooth manifold and let $G$ be a finite group that acts smoothly and orientation preservingly on $X$ and that $dim( H^+(X)^G) = 1$. Suppose $\mathfrak{s}$ is a spin$^c$-structure preserved by $G$. Suppose that $G_{\mathfrak{s}}$ is split and choose a splitting. Let $\phi$ be a chamber and set $H_0 = H^+(X)/\langle \phi \rangle$. Assume that $G$ acts orientation preservingly on $H^+(X)$ and orient $H_0$ so that $H^+(X) \cong \mathbb{R}\phi \oplus H_0$. Then
\[
SW_{G,X,\mathfrak{s}}^\phi( x^m) - SW_{G,X,\mathfrak{s}}^{-\phi}(x^m) = e( H_0 ) s_{m-(d-1)}(D).
\]
\end{theorem}
\begin{proof}
We can use the same proof as in the families case \cite[Corollary 5.5]{bk} to deduce that $SW^\phi_G( x^m ) - SW^{-\phi}_G(x^m) = Obs(\phi , -\phi) s_{m-(d-1)}(D)$, where $Obs(\phi,-\phi)$ is the $G$-equivariant relative obstruction class. Furthermore, by factoring $\phi : pt \to S(H^+)$ as $pt \to S( \mathbb{R}\phi ) \to S( H^+)$ and similarly for $-\phi$, it follows easy that $Obs(\phi , -\phi) = e( H_0 )$, where $H_0 = H^+(X)/\langle \phi \rangle$ and where the orientation on $H_0$ is induced by the orientation on $H^+(X)$ by $H^+(X) = \mathbb{R}\phi \oplus H_0$ as a direct sum of oriented vector spaces ($\mathbb{R} \mathbb{\phi}$ is oriented positively in the direction of $\phi$).
\end{proof}

\begin{remark}
In the case that the action of $G$ on $H^+(X)$ is not orientation preserving, we still have a wall-crossing formula. We can either work with local coefficients or we can work mod $2$, in which case the Euler class $e(H_0)$ should be replaced by the top Stiefel--Whitney class.
\end{remark}

\subsection{Positive scalar curvature}\label{sec:psc}

We will show that the equivariant Seiberg--Witten invariants are obstructions to the existence of $G$-invariant positive scalar curvature metrics.

\begin{theorem}
Let $X$ be a compact, oriented, smooth $4$-manifold with $b_1(X) = 0$ and $G$ a finite group acting smoothly and orientation preservingly on $X$. Let $\mathfrak{s}$ be a spin$^c$-structure whose isomorphism class is preserved by $G$. 
\begin{itemize}
\item[(1)]{Suppose that $b_+(X)^G > 1$ and that $X$ admits a $G$-invariant metric of positive scalar curvature. Then $SW_{G,X,\mathfrak{s}} = 0$.}
\item[(2)]{Suppose that $b_+(X)^G = 1$ and $X$ admits a $G$-invariant metric $g$ of positive scalar curvature. Let $\phi \in S(H^+(X)_g)^G$ be a chamber. If $\langle \phi , c(\mathfrak{s}) \rangle \ge 0$, then $SW^\phi_{G,X,\mathfrak{s}} = 0$.}
\end{itemize}
\end{theorem}
\begin{proof}
(1) Let $g$ be a $G$-invariant metric of positive scalar curvature. Let $\phi \in S(H^+(X)_g)^G$ which we regard as a $G$-invariant self-dual harmonic $2$-form. Consider the Seiberg--Witten equations for $(X , \mathfrak{s} , g , \eta)$ where $\eta$ is a $2$-form perturbation of the form $\eta = i \lambda \phi$, where $\lambda > 0$. To be precise, this means we consider the equations $D_A \psi = 0$, $F^+_A + \eta = \sigma(\psi)$. Since $g$ has positive scalar curvature, there are no solutions to the $\eta$-perturbed Seiberg--Witten equations provided $\lambda$ is sufficiently small. This means that the Seiberg--Witten moduli space is empty. Of course this trivially implies that the moduli space is cut out transversally and hence the equivariant Seiberg--Witten invariants are zero, as they are given by evaluating cohomology classes on the moduli space.

(2) The argument is similar except that we have to keep track of chambers. We get vanishing of the equivariant Seiberg--Witten invariants for the chamber which $\eta$ lies in. The chamber which $\eta$ lies in is determined by the sign of 
\[
\rho = \langle -c(\mathfrak{s}) - (i/2\pi) \eta , \phi \rangle = -\langle c(\mathfrak{s}) , \phi \rangle + \frac{\lambda}{2\pi} \langle \phi , \phi \rangle.
\]
Assume that $\langle \phi , c(\mathfrak{s}) \rangle \le 0$. Then for all $\lambda > 0$ we have $\rho > 0$, which means that $\eta$ lies in the $\phi$ chamber. Hence $SW^\phi_{G,X,\mathfrak{s}} = 0$.
\end{proof}

\subsection{Mod 2 invariants for spin actions}\label{sec:mod2}

Let $X$ be a compact, oriented, smooth $4$-manifold with $b_1(X) = 0$ and $G$ a finite group acting smoothly and orientation preservingly on $X$. If $\mathfrak{s}$ is a spin structure on $X$ and the action of $G$ lifts to the principal $Spin(4)$-bundle of $\mathfrak{s}$ we can use the results of \cite{bar2} to evaluate the mod $2$ equivariant Seiberg--Witten invariants.

\begin{theorem}\label{thm:mod2}
Let $X$ be a compact, oriented, smooth $4$-manifold with $b_1(X) = 0$ and $G$ a finite group acting smoothly and orientation preservingly on $X$. Let $\mathfrak{s}$ be a spin structure on $X$ and suppose the action of $G$ lifts to the principal $Spin(4)$-bundle of $\mathfrak{s}$. If $m\ge 0$ is even and $\phi$ is any chamber, then 
\[
SW^\phi_{X,\mathfrak{s}}(x^m) = w_{b_+(X)-3}( H^+(X) ) s_{m-(d-2)}(D) \in H^*_G(pt ; \mathbb{Z}_2)
\]
where $d = -\sigma(X)/8$. In particular, the mod $2$ value of $SW^\phi_{X,\mathfrak{s}}(x^m)$ is independent of the chamber.
\end{theorem}
\begin{proof}
By Theorem \ref{thm:eqfam2}, it suffices to show that
\[
SW^{\phi_E}_{E,\mathfrak{s}_E}(x^m) = w_{b_+(X)-3}( H^+(E/B) ) s_{m-(d-2)}(D_E) \in H^*(B ; \mathbb{Z}_2)
\]
for any principal $G$- bundle $P \to B$ over a smooth compact base $B$, where $(E , \mathfrak{s}_E , \phi_E)$ is the associated family, as described in Section \ref{sec:eqfam}. But this follows immediately from \cite[Theorem 6.7]{bar2}.
\end{proof}

\begin{remark}
Note that Theorem \ref{thm:mod2} is only of interest if $|G|$ is even. Indeed $H^j_G(pt ; A)$ is $|G|$-torsion for $j > 0$ and any coefficient group $A$. Hence $H^j_G(pt ; \mathbb{Z}_2) = 0$ if $|G|$ is odd and $j > 0$. 
\end{remark}

\section{Relations to families Seiberg--Witten invariants}\label{sec:eqfam}

In this section we will establish a precise relationship between the equivariant Seiberg--Witten invariants and families Seiberg--Witten invariants. Let $X$ be a compact, oriented, smooth $4$-manifold with $b_1(X) = 0$. Let $B$ be a compact smooth manifold. A {\em smooth family of $4$-manifolds over $B$ with fibres diffeomorphic to $X$} is a smooth fibre bundle $\pi : E \to B$ with fibrewise orientation (i.e., an orientation on the vertical tangent bundle $T(E/B) = Ker(\pi_*)$) such that the fibres of $E$ are orientation preservingly diffeomorphic to $X$. A {\em families spin$^c$-structure} $\mathfrak{s}_E$ is a spin$^c$-structure on $T(E/B)$. A {\em chamber} for $(E , \mathfrak{s}_E)$ is a section $\phi_E : B \to S(H^+(E/B))$, where $H^+(E/B)$ denotes the bundle over $B$ whose fibre over $b \in B$ is the space of harmonic self-dual $2$-forms on the $X_b = \pi^{-1}(b)$ of $E$ over $b$ (this definition requires a choice of metric on $T(E/B)$, but the isomorphism class of $H^+(E/B)$ does not depend on the choice). Associated to the triple $(E , \mathfrak{s}_E , \phi_E)$ is the families Seiberg--Witten invariant
\[
SW^{\phi_E}_{E , \mathfrak{s}_E} : H^*_{S^1}(B ; \mathbb{Z}) \to H^{*-d(X,\mathfrak{s})}(B ; \mathbb{Z}_w)
\]
where $w = w_1( H^+(E/B))$ and $\mathfrak{s}$ is the spin$^c$-structure on $X$ given by restricting $\mathfrak{s}_E$ to a fibre (this only determines $\mathfrak{s}$ up to the action of the diffeomorphism group of $X$, but $d(X,\mathfrak{s})$ does not depend on the choice). Recall that $d(X,\mathfrak{s}) = 2d-b_+-1$, where $b_+ = b_+(X)$ and $d = (c(\mathfrak{s})^2 - \sigma(X))/8$.

The families Seiberg--Witten invariant $SW^\phi_{E,\mathfrak{s}_E}$ is traditionally defined by evaluating cohomology classes over the families Seiberg--Witten moduli space for $(E , \mathfrak{s}_E , \phi_E)$. However it can also be evaluated using finite-dimensional approximation. Taking a finite dimensional approximation of the Seiberg--Witten monopole map for the family $E$ yields a monopole map, the families Bauer--Furuta invariant of $(E , \mathfrak{s}_E)$:
\[
f_E : S^{V_E,U_E} \to S^{V'_E,U'_E}
\]
where $V_E,V'_E \to B$ are complex vector bundles, $U_E,U'_E \to B$ are real vector bundles and $f$ is $S^1$-equivariant. $f_E$ is a monopole map in the sense of Definition \ref{def:mm} for the trivial group $G = \{1\}$. We have that $V_E - V'_E = D_E$ is the index of the family of spin$^c$ Dirac operators determined by $(E , \mathfrak{s}_E)$ and $U'_E - U_E = H^+$ is the bundle $H^+(E/B)$. According to \cite[Theorem 2.24 and Theorem 3.6]{bk}, the families Seiberg--Witten invariant of $(E , \mathfrak{s}_E , \phi)$ coincides with the abstract Seiberg--Witten invariant of $f_E$. That is, we have an equality
\[
SW^{\phi_E}_{E , \mathfrak{s}_E} = SW^{\phi_E}_{f_E}.
\]

Now let $G$ be a finite group which acts smoothly and orientation preservingly on $X$ and let $\mathfrak{s}$ be a spin$^c$-structure whose isomorphism class is $G$-invariant. Suppose also that $G$ lifts to the spinor bundles of $\mathfrak{s}$. This is equivalent to giving a splitting $G_{\mathfrak{s}} \cong S^1 \times G$. Let $f : S^{V,U} \to S^{V',U'}$ denote the $G$-equivariant Bauer--Furuta invariant of $(X , \mathfrak{s})$ with respect to the $G$-action. As usual, set $D = V - V'$, $H^+ = U' - U$. Let $\phi \in S(H^+)^G$ be a chamber for $f$. The $G$-equivariant Seiberg--Witten invariant of $(X , \mathfrak{s} , \phi)$ takes form 
\[
SW^\phi_{G,X,\mathfrak{s}} = SW^\phi_{G,f} : H^*_{S^1 \times G}(pt ; \mathbb{Z}) \to H^{*-d(X,\mathfrak{s})}_G(pt ; \mathbb{Z}_w),
\]
where $w = w_1(H^+)$. Let $B$ be a compact smooth manifold and $P \to B$ a principal $G$-bundle over $B$. To $P$ we can associate a smooth family $E_P = P \times_G X$ with fibres diffeomorphic to $X$. The action of $G$ on the spinor bundles of $\mathfrak{s}$ determines a spin$^c$-structure on $T(E_P/B)$. To see this, note that $E_P$ is the quotient of the trivial family $P \times X$ by the $G$-action $g( p , x ) = (pg^{-1} , gx)$. The vertical tangent bundle of the trivial family is canonically isomorphic to the pullback of $TX$ to $P \times X$, hence $\mathfrak{s}$ determines a families spin$^c$-structure $\widetilde{\mathfrak{s}}$ for $P \times X$. Since $G$ acts on the spinor bundles of $\mathfrak{s}$, it also acts on the spinor bundles for $\widetilde{\mathfrak{s}}$ and hence $\widetilde{\mathfrak{s}}$ descends to a families spin$^c$-structure $\mathfrak{s}_{E_P}$ on the quotient family $E_P$. In a similar manner the chamber $\phi$ defines a chamber $\phi_{E_P}$ for the family $E_P$.

To any space $Y$ on which $G$ acts we obtain a corresponding fibre bundle $Y_P = P \times_G Y$ over $B$. To any $G$-equivariant map $\psi : Y \to Y'$ between two $G$-spaces, we get an associated bundle map $\psi_P : Y_P \to Y'_P$. The Seiberg--Witten monopole equations for the family $E_P$ are obtained from the Seiberg--Witten equations for $X$ by exactly this construction. It follows that The Bauer--Furuta map for $(E_P , \mathfrak{s}_{E_P})$ is the map $f_P : S^{V_P , U_P} \to S^{V'_P , U'_P}$ assocated to $f : S^{V,U} \to S^{V',U'}$. In the case that $P = EG$ is the universal $G$-bundle over $BG$, the assignment $Y \mapsto Y_P$ is the Borel construction. The $G$-equivariant Seiberg--Witten invariants of $f$ are then seen to correspond to the non-equivariant Seiberg--Witten invariants of $f_{EG}$. Any principal bundle $P \to B$ is isomorphic to the pullback of $EG \to BG$ by its classifying map $\varphi_P : B \to BG$, in which case $Y_P \cong \varphi_P^*( Y_{EG} )$. It follows that the Seiberg--Witten invariants for $f_P$ are the pullback under $\varphi_P$ of the $G$-equivariant Seiberg--Witten invariants for $f$. More precisely, we have a commutative diagram:
\[
\xymatrix{
H^*_{S^1 \times G}(pt ; \mathbb{Z}) \ar[d]^-{\varphi_P^*} \ar[r]^-{SW^\phi_{G,f}} & H^{*-d(X,\mathfrak{s})}_G(pt ; \mathbb{Z}_w) \ar[d]^-{\varphi_P^*} \\
H^*_{S^1}(B ; \mathbb{Z}) \ar[r]^-{SW^{\phi_P}_{f_P}} & H^{*-d(X,\mathfrak{s})}(B ; \mathbb{Z}_w).
}
\]
In particular, we have
\[
SW^{\phi_P}_{f_P}( x^m ) = \varphi_P^*( SW^{\phi}_f(x^m))
\]
for all $m \ge 0$. Expressing this in terms of the equivariant and families Seiberg--Witten invariants of $(X , \mathfrak{s})$, we have shown:

\begin{theorem}\label{thm:eqfam}
Let $X$ be a compact, oriented, smooth $4$-manifold with $b_1(X)=0$. Let $G$ act smoothly and orientation preservingly on $X$ and suppose that the $G$-action lifts to the spinor bundles of a spin$^c$-structure $\mathfrak{s}$. Let $\phi \in H^+(X)^G$ be a chamber. Let $P \to B$ be a principal $G$-bundle over a compact smooth manifold $B$ and $E = P \times_G X$ the associated family. Then
\[
SW^{\phi_{E}}_{E , \mathfrak{s}_{E}}( x^m ) = \varphi_P^*( SW^{\phi}_{G,X,\mathfrak{s}}( x^m))
\]
for all $m \ge 0$, where $\varphi_P : B \to BG$ is the classifying map of $P$. A similar result also holds for the $K$-theoretic Seiberg--Witten invariants.
\end{theorem}

The following result shows that the equivariant Seiberg--Witten invariants are completely determined by the associated families invariants. Thus, it is possible to compute equivariant Seiberg--Witten invariants in terms of families invariants.
\begin{theorem}\label{thm:eqfam2}
Let $X$ be a compact, oriented, smooth $4$-manifold with $b_1(X)=0$. Let $G$ act smoothly and orientation preservingly on $X$ and suppose that the $G$-action lifts to the spinor bundles of a spin$^c$-structure $\mathfrak{s}$. Let $\phi \in H^+(X)^G$ be a chamber. Let $m \ge 0$ be given. Suppose $\lambda \in H^{2m-d(X,\mathfrak{s})}_G( pt ; \mathbb{Z}_w)$ satisfies
\[
SW^{\phi_E}_{E , \mathfrak{s}_E}(x^m) = \varphi_P^*(\lambda)
\]
for every principal bundle $P \to B$ over a compact smooth manifold $B$, where $E = P \times_G X$. Then $\lambda = SW^\phi_{G,X,\mathfrak{s}}(x^m)$.
\end{theorem}
\begin{proof}
In light of Theorem \ref{thm:eqfam}, it suffices to show that if $\lambda_1, \lambda_2 \in H^*_G(pt ; \mathbb{Z}_w)$ satisfy $\varphi^*(\lambda_1) = \varphi^*(\lambda_2)$ for every map $\varphi : B \to BG$ from a compact smooth manifold $B$ into $BG$, then $\lambda_1 = \lambda_2$. This will follow by constructing finite dimensional approximations of $BG$. Since $G$ is a finite group, it embeds in $O(k)$ for some $k$. Finite dimensional approximations for $EO(k)$ are given by the Stiefel manifolds $E_{n,k} = V_k( \mathbb{R}^{n+k} ) = O(n+k)/O(n)$ and finite dimensional approximations for $BO(k)$ are given by the Grassmannians $B_{n,k} = E_{n,k}/O(k)$. By restriction $G$ acts freely on $E_{n,k}$ and we can take $B_n = E_{n,k}/G$ as finite dimensional approximations of $BG = EO(K)/G$.

For any $n$ we have that $H^j( EO(k) , E_{n,k} ) = 0$ for all $j \le n$. This is because $EO(k)$ is contractible and the first non-trivial homotopy group of $E_{n,k}$ is $\pi_n(E_{n,k})$. The Leray--Serre spectral sequence for the Borel fibration immediately implies that $H^j( EO(k)/G , E_{n,k}/G) = H^j_G( EO(k) , E_{n,k}) = 0$ for all $j \le n$. Hence for any $j$, the restriction $H^j( BG ) \to H^j( B_n)$ is an isomorphism for large enough $n$. Now we take $B = B_n$ for large enough $n$ and take $\varphi : B \to BG$ to be inclusion map to deduce that $\lambda_1 = \lambda_2$.
\end{proof}

\begin{remark}
The equivalent of Theorem \ref{thm:eqfam2} for the $K$-theoretic invariants is not true in general. Given a virtual representation $W \in R(G)$ and a principal $G$-bundle $P \to B$ over a manifold $B$ we obtain an associated virtual vector bundle $W_B \to B$ and corresponding $K$-theory class $[W_B] \in K^0(B)$. Let $I(G) = Ker( R(G) \to \mathbb{Z})$ denote ideal of rank $0$ virtual representations. If $W \in I_\infty(G) = \bigcap_{n = 1}^{\infty} I(G)^n$ then $[W_B] = 0$ because every class in $\widetilde{K}^0(B)$ is nilpotent. Hence any class in $I_\infty(G)$ can not be detected by finite dimensional manifolds. In fact, by \cite[Theorem 7.2]{at}, it follows that a class $W \in R(G)$ satisfies $[W_B] = 0$ for every $P \to B$ if and only if $W \in I_\infty(G)$. If $I_\infty(G) = 0$, then the equivalent of Theorem \ref{thm:eqfam2} holds for the $K$-theoretic invariants. Note that according to \cite[Proposition 6.10]{at}, we have $I_\infty(G) = \{ \chi \in R(G) \; | \; \chi(g) = 0 \text{ for all } g \in G \text{ of prime power order} \}$.
\end{remark}

\section{Equivariant Seiberg--Witten moduli spaces}\label{sec:eswm}

Consider a $G$-monopole map $f : S^{V,U} \to S^{V',U'}$ over a point and let $y \in (S^{U'})^{G_{\mathfrak{s}}} \setminus (S^{U})^{G_{\mathfrak{s}}}$ define a chamber for $f$. In this section we investigate sufficient conditions under which $f$ can be equivariantly homotopied (relative $S^U$) to a map $f'$ which is transverse to $y$. If this is possible then we obtain a smooth ``moduli space'' $\mathcal{M} = (f')^{-1}(y)/S^1$ which plays the role of the Seiberg--Witten moduli space for $f$. Since $f$ and $f'$ are equivariantly homotopic their abstract Seiberg--Witten invariants agree. However, the Seiberg--Witten invariants of $f'$ can be computed by evaluating cohomology classes on the moduli space $\mathcal{M}$ exactly as one does in the traditional approach to Seiberg--Witten theory. Thus, when equivariant transversality can be achieved, we obtain a moduli space interpretation of the equivariant Seiberg--Witten invariants. Under further assumptions it can be shown that the equivariant cobordism class of $\mathcal{M}$ is independent of the choice of homotopy and so defines an invariant of the stable equivariant homotopy class of $f$. In the zero-dimensional case this amounts to an oriented finite $G$-set $\mathcal{M}$ (i.e., an element of the Burnside ring $A(G)$) together with a $G_{\mathfrak{s}}$-equviariant line bundle $\widetilde{\mathcal{M}} \to \mathcal{M}$.

Note that we can not hope to achieve equivariant transversality in all cases. In fact the equivariant Seiberg--Witten invariants are particularly interesting when the expected dimension is negative. However in such a case we can not obtain equivariant transversality (except possibly if the invariants all vanish).

\subsection{Equivariant transversality}\label{sec:et}

Consider a $G$-equivariant monopole map $f : S^{V,U} \to S^{V',U'}$ over $B = pt$ equivariant with respect to $G_{\mathfrak{s}}$. We assume $(H^+)^G \neq 0$. Hence there exists chambers. In terms of monopole maps, there is a point $y \in (S^{V',U'})^{G_\mathfrak{s}}$ not in the image of $f|_{S^U}$. Let $Y = \{y\}$. If $y$ were a regular value of $f$, then $f^{-1}(y)$ would be a smooth $G_{\mathfrak{s}}$-invariant moduli space. Our goal in this section is to modify $f$ by a $G_{\mathfrak{s}}$-equivariant homotopy relative $S^U$ so that the resulting map $f'$ is transverse to $Y$. If this can be achieved then $\widetilde{\mathcal{M}} = (f')^{-1}(y)$ is a smooth $G_{\mathfrak{s}}$-manifold on which $S^1$ acts freely. The Seiberg--Witten invariants of $f$ can be obtained by integrating over the ``moduli space'' $\mathcal{M} = \widetilde{\mathcal{M}}/S^1$ (this follows by straightforward extension of \cite[Theorem 3.6]{bk} to the $G$-equivariant case).

We will follow's Petire's approach to the $G$-transversality problem \cite{pet}. Let $\mathcal{L}$ denote the set of conjugacy classes of subgroups of $G_{\mathfrak{s}}$ which occur as stabilisers of points in $S^{V,U}$. There are two classes of such groups: those which contain $S^1$ and those which intersect trivially with $S^1$. The subgroups which contain $S^1$ have the form $H_{\mathfrak{s}}$, where $H$ is a subgroup of $G$ and $H_{\mathfrak{s}}$ is the restriction of $G_\mathfrak{s}$ to $H$. The subgroups that intersect trivially with $S^1$ have the form $\psi(H)$ for some subgroup $H$ of $G$ and some splitting $\psi : H \to H_{\mathfrak{s}}$. $\mathcal{L}$ is partially ordered by inclusion, i.e. $H_1 \le H_2$ if $H_1$ is conjugate to a subgroup of $H_2$. Choose a bijection $\alpha : \mathcal{L} \to \{1, \dots , T\}$ such that $H_1 \le H_2$ implies $\alpha(H_1) \ge \alpha(H_2)$.

The process of modifying $f$ is done in steps. Define $Z_{k-1} \subseteq S^{V,U}$ by setting
\[
Z_{k-1} = \bigcup_{H \; | \; \alpha( H ) < k} (S^{V,U})^H
\]
where the union is over closed subgroups $H \subseteq G_{\mathfrak{s}}$ such that $\alpha(H) < k$. So $Z_0$ is empty, $Z_1 = (S^{V,U})^{G_{\mathfrak{s}}}$, $Z_T = (S^{V,U})$. Since $Z_0$ is empty, we trivially have that $f$ is transverse to $Y$ along $Z_0$. Suppose inductively we have contructed a homotopy $f_{k-1}$ of $f$ which is transverse to $Y$ along $Z_{k-1}$. Note that $f_{k-1}$ will also be transverse to $Y$ in some neighbourhood of $Z_{k-1}$. Next we wish to homotopy $f_{k-1}$ to a new map $f_k$ such that $f_{k} = f_{k-1}$ in a neighbourhood of $Z_{k-1}$ and so that $f_k$ is transverse to $Y$ along $Z_k$. If we can do this for all $k$, then we will have found the desired homotopy. Since $f^{-1}(Y)$ is disjoint from $S^U$, we can assume the homotopies are relative $S^U$. As explained in \cite[Page 188]{pet}, we can also assume that whenever $\alpha(K) = k$, the restricted map $f_{k-1}^K : (S^{V,U})^K \to (S^{V',U'})^K$ is transverse to $Y$ (note $Y^K = Y$ since $Y = \{y\}$ is a point fixed by $G_{\mathfrak{s}}$). Therefore $X^K = (f_{k-1}^K)^{-1}(Y)$ is a smooth manifold. Obstruction theory yields a sequence of obstructions for obtaining $f_k$. These obstructions take the form of classes
\[
O_*(f_{k-1} , K) \in H^*( X^K/N(K) , X_K/N(K) ; \pi_{*-1}(V(K)) ),
\]
where $\alpha(K) = k$, $X_K = \bigcup_{H | K \subset H } X^H$ and $V(K)$ is a certain fibre bundle (we are abusing notation and using $V(K)$ to denote both the fibre bundle and also its fibres). Also $\pi_{*-1}(V(K))$ should be understood as the local system whose fibres are the corresponding homotopy groups of the fibres of $V(K)$.

The obstruction $O_j( f_{k-1} , K)$ is defined provided all the obstructions $O_i( f_{k-1} , K)$ vanish for $i < j$. Note also that the fibre bundle $V(K)$ can be empty. One defines $O_0( f_{k-1} , K)$ to be zero if $V(K)$ is non-empty and to be non-zero if $V(K)$ is empty. Thus for the $0$-th obstruction to vanish it is necessary that $V(K)$ is non-empty (except when $X^K = X_K$, in which case there are no obstructions for trivial reasons). Strictly speaking $V(K)$ is not a fibre bundle since its fibres can be different over different components of $X^K$. However this does not cause any difficulties since we can apply obstruction theory componentwise.

For a representation $W$ of $K$ write $W^K$ for the invariant part and $W_K$ for the orthogonal complement of $W^K$ in $W$. We can regard $W_K$ either as a subspace of $W$ or as the quotient $W/W^K$. Set $A = V \oplus U$, $A' = V' \oplus U'$. The fibres of $V(K)$ are $Hom^s_K( A_K , A'_K )$, the space of $K$-equivariant surjective maps $A_K \to A'_K$. Using Schur's lemma, one sees that $V(K)$ is a product of Stiefel manifolds.

Let $\widehat{K}$ denote the space of isomorphism classes of irreducible real representations of $K$. If $\chi \in \widehat{K}$, let $d_{\chi}$ denote the dimension of $Hom_K(\chi , \chi)$, which is $1$, $2$ or $4$ according to whether $\chi$ is real, complex, or quaternionic. If $W$ is a finite dimensional representation of $K$, let $W_\chi$ denote the multiplicity of $\chi$ in $W$. We also define $W_\chi$ in the case $W$ is a virtual representation. We have \cite{pet} that $V(K)$ is empty if and only if $d(K,A_K-A'_K) < 0$, where 
\[
d(K,W) = \min_{\chi \in \widehat{K} \setminus \{1\}}( d_{\chi}( W_\chi + 1) - 1).
\]
(Note that $A_K, A'_K$ contain no copies of the trivial representation which is why we can restrict to $\chi \in \widehat{K} \setminus \{1\}$ in the above definition). Furthermore, $\pi_i( V(K) ) = 0$ for $i < d(K , A_K-A'_K)$ \cite[Lemma 4.6]{pet}. Noting that $A - A' = D - H^+$, this leads to the following result:

\begin{proposition}
If the dimension of $X^K/N(K)$ less than or equal to $d(K , D_K - H^+_K)$, then the obstructions $O_*(f_{k-1} , K)$ all vanish.
\end{proposition}

In our case if $X^K \setminus X_K$ is non-empty then it has dimension $dim(D^K) - dim((H^+)^K)$. The normaliser $N(K)$ will always contain the circle $S^1 \subseteq G_{\mathfrak{s}}$ and (by our assumption that $y$ is not in the image of $f|_{S^U}$), the circle will always act freely on $X^K$. So $dim((X^K \setminus X_K )/N(K)) = dim(D^K) - dim((H^+)^K) - 1$. So the above can be re-written as:

\begin{corollary}
If $dim(D^K - (H^+)^K) \le d(K , D_K - H^+_K) + 1$, then the obstructions $O_*(f_{k-1},K)$ all vanish.
\end{corollary}

Although the above theorem gives sufficient conditions for equivariant transversality, it is more restrictive than necessary. We can obtain a better result using \cite[Theorem 5.1]{pet}, which reduces the problem to vanishing of certain obstructions $O_*(\alpha , K) \in H^*(Y^K/N(K) , Y_K/N(K) ; \pi_{*-1}(V'(K)))$ (\cite[Theorem 5.1]{pet} is stated for groups of the form $T \times F$ where $T$ is a torus and $F$ is finite, but the result works just as well for a central extension of a finite group by a torus). In our case $Y = \{y\}$ is just a point and $Y^K = Y_K$ except for $K = G_{\mathfrak{s}}$. So we only have one obstruction $O_0(\alpha , G_{\mathfrak{s}})$. This gives:

\begin{theorem}
If $H^+$ is a trivial representation of $G$ and if $D_\chi \ge 0$ for each non-trivial character $\chi$ of $G_{\mathfrak{s}}$, then $O_0(\alpha , G_{\mathfrak{s}})$ vanishes and hence equivariant transversality can be achieved.
\end{theorem}

Unfortunately, the obstruction $O_0(\alpha , G_{\mathfrak{s}})$ will often be non-zero even in cases where equivariant transversality can be achieved. Fortunately a small modification of Petrie's theorem gives a much better result. To explain the modification we first explain the idea behind \cite[Theorem 5.1]{pet}. The idea is that under the given assumptions, $V(K)$ is a pullback of a corresponding bundle $V'(K)$ on $Y^K$. Given a section $s_{k-1}^0$ of $V'(K)$ on $Y^K \cap Z'_{k-1}$, where $Z'_{k-1} = \bigcup_{H | \alpha(H) <k } (S^{A'})^H$, we get a series of obstructions $O_*(\alpha , K)$ to extending the section to $Z'_k$. The obstruction to extending $f_{k-1}^*(s_{k-1}^0)$ to $X^K$ is precisely $f^*( O_*(\alpha , K))$ and this means that $O_*( f_{k-1} , K) = f^*( O_*(\alpha , K))$ {\em provided} that the section we wish to extend is the pullback section $f_{k-1}^*(s_{k-1}^0)$. If $O_0(\alpha , G_{\mathfrak{s}}) \neq 0$, then this approach doesn't work. However we can modify the argument slightly, as follows.

Suppose there is a $k_0$ such that $(f^{-1}(y))^{K}$ is empty for all $k < k_0$, where $\alpha(K) = k$. If this happens then we can take $f_{{k_0}-1} = f$. In our case $(f^{-1}(y))^K$ is empty whenever $K$ contains $S^1$. So we order the elements of $\mathcal{L}$ such that all the groups containing $S^1$ go first in the list and then take $k_0$ to the the index of the first group not containing $S^1$. Now let $K = H_k$ be a subgroup of $G_{\mathfrak{s}}$ such that $\alpha(H) < k_0$ for any group that properly contains $K$ (i.e., $K$ is maximal with respect to those $K$ in $\mathcal{L}$ with $\alpha(K) \ge k_0)$. Then $X_K$ is empty and hence the relative obstructions $O_*( f_{k-1} , K)$ are absolute obstructions: $O_*(f_{k-1} , K) \in H^*( X^K/N(K) ; \pi_{*-1}(V(K)))$. Then since $V(K))$ is a pullback from $Y^K$, the obstruction is the pullback of some obstruction class $O_*(\alpha , K) \in H^*( Y^K ; \pi_{*-1}(V'(K)) )$. Now if these obstructions vanish, then we can proceed along the list of groups and achieve transversality. For $K$ which are non-maximal amongst $\{ [H] \in \mathcal{L} \; | \; \alpha(H) \ge k_0 \}$, we will again get relative obstructions which are pullbacks of classes in $H^*( Y^K/N(K) , Y_K/N(K) ; \pi_{*-1}(V'(K)))$. But in such cases $Y^K = Y_K$ and so there is no obstruction for trivial reasons. So we have proven that equivariant transversality can be achieved provided that $d( K , D_K - (H^+)_K ) \ge 0$ for each maximal element of $\{ [H] \in \mathcal{L} \; | \; \alpha(H) \ge k_0 \}$.  Since $d(K , D_K - (H^+)_K) = \min_{\chi \neq 1} d_{\chi}( D_\chi - H^+_\chi +1) -1$,  the condition is equivalent to saying $D_\chi \ge H^+_\chi$ for every $\chi \in \widehat{K} \setminus \{1\}$. Thus we have proven:

\begin{theorem}\label{thm:eqtr}
Suppose that for each maximal element $K$ of $\{ [H] \in \mathcal{L} \; | \; \alpha(H) \ge k_0 \}$ and for each $\chi \in \widehat{K} \setminus \{1\}$ we have $D_\chi \ge H^+_\chi$, then equivariant transversality can be achieved.
\end{theorem}

\begin{example}\label{ex:eqtr}
Suppose that $G = \mathbb{Z}_p = \langle g \rangle$ for a prime $p$. Let $\mathbb{C}_j$ denote the $1$-dimensional complex representation of $G$ where $g$ acts as multiplication by $e^{2\pi ij/p}$. Let $d_j$ be the multiplicity of $\mathbb{C}_j$ in $D$ and let $b_j$ denote the multiplicity of $\mathbb{C}_j$ in $H^+ \otimes_{\mathbb{R}} \mathbb{C}$. The maximal subgroups of $G_{\mathfrak{s}} \cong S^1 \times G$ not containing $S^1$ are exactly the splittings $\langle ( \xi^{-j} , g ) \rangle$ of $G_{\mathfrak{s}} \to G$. The hypotheses of Theorem \ref{thm:eqtr} is that for each $j \in \mathbb{Z}_p$ and each $i \in \mathbb{Z}_p^*$, we have $d_{j-i}+d_{j+i} \ge b_i$ for $p$ odd, or $2d_j \ge b_+ - b_0$ for $p=2$.
\end{example}

\begin{remark}
If equivariant transversality can be achieved and $d(X,\mathfrak{s}) = 2d - b_+ - 1 < 0$, then the moduli space is empy and the equivariant Seiberg--Witten invariants all vanish.
\end{remark}

\begin{remark}
Theorem \ref{thm:eqtr} can be improved slightly. If the expected dimension $dim(D^K-(H^+)^K)-1$ of $X^K/S^1$ is negative, then $X^K$ is empty and there are no obstructions associated to $K$.
\end{remark}

Suppose that the condition in Theorem \ref{thm:eqtr} for equivariant transversality is satisfied, that is, $D_\chi > H^+_\chi$ for every non-trivial $\chi \in \widehat{K}$ and every $K$ which is maximal in $\{ [H] \in \mathcal{L} \; | \; \alpha(H) \ge k_0 \}$. Then there exists a $G_{\mathfrak{s}}$-equivariant homotopy $f'$ of $f$ relative $S^U$ such that $\widetilde{\mathcal{M}} = (f')^{-1}(y)$ is smooth. Hence we get a smooth moduli space $\mathcal{M} = \widetilde{\mathcal{M}}/S^1$ on which $G$ acts and a $G_{\mathfrak{s}}$-equivariant line bundle $L = \widetilde{\mathcal{M}} \times_{S^1} \mathbb{C}$. The question we now address is to what extent $(\mathcal{M} , L)$ depends the choice of homotopy $f'$ (and also the choice of regular value $y$).

\begin{theorem}\label{thm:eqtr2}
Suppose that for each maximal element $K$ of $\{ [H] \in \mathcal{L} \; | \; \alpha(H) \ge k_0 \}$ and for each $\chi \in \widehat{K} \setminus \{1\}$ we have $D_\chi \ge H^+_\chi$, then equivariant transversality can be achieved. Let $y' \in (S^{U'} \setminus S^U)^{G_\mathfrak{s}}$ and let $f'$ be a $G_\mathfrak{s}$-equivariant homotopy of $f$ for which $y'$ is a regular value of $f'$. Let $\widetilde{\mathcal{M}}' = (f')^{-1}(y)$, $\mathcal{M}' = \widetilde{\mathcal{M}}'/S^1$ and $L' = \widetilde{\mathcal{M}}' \times_{S^1} \mathbb{C}$.

Suppose that we also have $D_\chi > H^+_\chi$ whenever $\chi \in \widehat{K} \setminus \{1\}$ and $d_\chi = 1$. Then the equivariant cobordism class of $(\mathcal{M}' , L')$ does not depend on the choice of $y'$ and $f'$. More precisely, suppose that $(\mathcal{M}'' , L'')$ is the pair associated to some other choice of $y''$ and $f''$ and where $y', y''$ lie in the same connected component of $(S^{U'} \setminus S^U)^{G_\mathfrak{s}}$. Then there exists a $G$-equivariant cobordism $Z$ from $\mathcal{M}$ to $\mathcal{M}'$ and a $G_{\mathfrak{s}}$-equivariant line bundle $L_Z \to Z$ which restricts to $L'$ on $\mathcal{M}'$ and $L''$ on $\mathcal{M}''$.
\end{theorem}
\begin{proof}
Since $f'$ and $f''$ are both $G_\mathfrak{s}$-equivariantly homotopic to $f$ relative $S^U$, there exists a $G_{\mathfrak{s}}$-equivariant homotopy $f_t$ relative $S^U$ from $f'$ to $f''$. Consider the map $F : [0,1] \times S^{V,U} \to [0,1] \times S^{V',U'}$ given by $F(t,x) = (t , f_t(x))$. Since $y',y''$ lie in the same connected component of $(S^{U'} \setminus S^U)^{G_\mathfrak{s}}$, we can find a smooth path $y_t$ from $y_0 = y'$ to $y_1 = y''$ taking values in $(S^{U'} \setminus S^U)^{G_\mathfrak{s}}$. Let $Y = \{ (t , y_t) \} \subset [0,1] \times S^{V',U'}$. Our goal will be to find a $G_{\mathfrak{s}}$-equivariant homotopy $F'$ of $F$ relative $A = ([0,1] \times S^U) \cup (\{0,1\} \times S^{V,U})$ such that $F'$ is transverse to $Y$. Then $Z = (F')^{-1}(Y)/S^1$ will be the desired homotopy. For the existence of such a homotopy we use exactly the same argument as for the proof of Theorem \ref{thm:eqtr}. The only difference now is that $Y$ is $1$-dimensional, so we need that $d_{\chi}( D_\chi - H^+_\chi +1) -1 \ge 1$ for each maximal element $K$ of $\{ [H] \in \mathcal{L} \; | \; \alpha(H) \ge k_0 \}$ and for each $\chi \in \widehat{K} \setminus \{1\}$. If $d_\chi > 1$ then this condition reduces to $D_\chi \ge H^+_\chi$. If $d_\chi = 1$, then the condition instead reduces to $D_\chi > H^+_\chi$.
\end{proof}

\begin{remark}\label{rem:odd}
If $G$ has odd order and $K \subseteq G$, then every $\chi \in \widehat{K} \setminus \{1\}$ has $d_\chi \neq 2$, as follows easily by considering Frobenius--Schur indicators (the map $g \mapsto g^2$ is a bijection). Hence if $G$ has odd order, then the hypotheses of Theorem \ref{thm:eqtr2} are no stronger than the hypotheses of Theorem \ref{thm:eqtr}.
\end{remark}

\subsection{Zero-dimensional case}\label{sec:zero}

If $f : S^{V,U} \to S^{V',U'}$ is a $G$-monopole map such that $D,H^+$ satisfy Theorem \ref{thm:eqtr}, then to any chamber $\phi$ we obtain a moduli space $\mathcal{M}$, which is a smooth, compact $G$-manifold of dimension $2d-b_+-1$ and a $G_{\mathfrak{s}}$-equivariant line bundle $L = \widetilde{\mathcal{M}} \times_{S^1} \mathbb{C}$ where the $S^1$-subgroup of $G_{\mathfrak{s}}$ acts on $L$ by scalar multiplication. In this section we consider the case that $\mathcal{M}$ is zero-dimensional, hence a finite $G$-set. Assume further that $G$ acts orientation preservingly on $H^+$ and fix a choice of orientation. Then $\mathcal{M}$ is an oriented finite $G$-set. Its equivariant cobordism class is an element of the Burnside ring $A(G)$ \cite{die}. The Burnside ring $A(G)$ is the Grothendieck group of the monoid whose elements are isomorphism classes of finite $G$-sets. Addition is given by disjoint union and multiplication is given by the Cartesian product. This makes $A(G)$ into a commutative ring with identity. Since every $G$-set can be decomposed into orbits, each element of $A(G)$ admits a canonical decomposition $\sum_H a_H [G/H]$ for some $a_H \in \mathbb{Z}$, where the sum is over conjugacy classes of subgroups of $G$ and $[G/H]$ denotes the isomorphism class of the homogeneous space $G/H$. Thus as an abelian group $A(G)$ is free of finite rank and has a basis $[G/H]$ where $H$ runs over the conjugacy classes of subgroups of $G$.

In addition to $\mathcal{M}$ we have also the line bundle $L \to \mathcal{M}$. This extra data defines a module over $A(G)$ which we will refer to as a Burnside module:

\begin{definition}
Let $G_{\mathfrak{s}}$ denote an $S^1$-central extension of $G$. Define the {\em ($\mathfrak{s}$-twisted) Burnside module} $\widehat{A}_{\mathfrak{s}}(G)$ to the the Grothendieck group of the monoid consisting of isomorphism classes of pairs $(M , L)$, where $M$ is a finite $G$-set and $L$ is a $G_{\mathfrak{s}}$-equivariant line bundle where the central $S^1$ subgroup acts on $L$ by scalar multiplication and the $G_{\mathfrak{s}}$-action covers the $G$-action on $M$. Addition is given by disjoint union.
\end{definition}

In the case that $G_{\mathfrak{s}} = S^1 \times G$ is the trivial extension, we can instead regard $L$ as a $G$-equivariant line bundle. The module structure $A(G) \times \widehat{A}_{\mathfrak{s}}(G) \to \widehat{A}_{\mathfrak{s}}(G)$ is defined as follows. Given a finite $G$-set $M_1$ and a pair $(M_2,L_2)$ where $M_2$ is a finite $G$-set and $L_2$ is a $G_{\mathfrak{s}}$-equivariant line bundle over $M_2$, we set $M_1 (M_2 , L_2) = (M_1 \times M_2 , L )$, where $L$ is the pullback of $L_2$ to $M_1 \times M_2$.

Given two central extensions $p : G_{\mathfrak{s}} \to G, p' : G_{\mathfrak{s}'} \to G$, we obtain a third central extension $G_{\mathfrak{s}+\mathfrak{s}'}$ by setting 
\[
G_{\mathfrak{s}+\mathfrak{s}'} = \{ (g,g') \in G_{\mathfrak{s}} \times G_{\mathfrak{s}'} \; pg = p'g' \}/\{ (u,u^{-1}) \; | \; u \in S^1\}.
\]
This operation corresponds to the addition operation on $H^2(G ; S^1) \cong H^3(G ; \mathbb{Z})$. Then we have a product $\widehat{A}_{\mathfrak{s}}(G) \times \widehat{A}_{\mathfrak{s}'}(G) \to \widehat{A}_{\mathfrak{s} + \mathfrak{s}'}(G)$ by setting $(M_1,L_1)(M_2,L_2) = (M_1 \times M_2 , L_1 \boxtimes L_2)$, where $L_1 \boxtimes L_2$ is the external tensor product. This makes $\widehat{A}(G) = \bigoplus_{\mathfrak{s}} \widehat{A}_{\mathfrak{s}}(G)$ into a ring which is graded by extension classes $\mathfrak{s} \in H^2(G ; S^1)$.

Given $(M,L)$, we can decompose $M$ into orbits $M_1, \dots , M_n$ giving a decomposition $(M,L) = (M_1,L_1) \cup \cdots \cup (M_n , L_n)$ where $L_i = L|_{M_i}$. This allows us to restrict attention to the case that $M = G/H$ for some subgroup $H \subseteq G$. Let $H_{\mathfrak{s}}$ denote the restriction of $G_{\mathfrak{s}}$ to $H$. Then $H_{\mathfrak{s}}$ acts on the fibre of $L$ over the identity coset. This defines a homomorphism $\chi : H_{\mathfrak{s}} \to S^1$ which is the identity map on the $S^1$ subgroup of $H_{\mathfrak{s}}$. Thus $\chi$ defines a splitting of the sequence $1 \to S^1 \to H_{\mathfrak{s}} \to H \to 1$. Let $H' = Ker(\chi)$. Then $H'$ is a subgroup of $G_{\mathfrak{s}}$ such that $H' \cap S^1 = \{1\}$. We can recover the pair $(M,L)$ from $H'$, namely $M = G/H$, where $H$ is the image of $H'$ in $G$ and $L$ is the line bundle associated to the circle bundle $G_{\mathfrak{s}}/H' \to G/H$. In this way, we see that the isomorphism class of $(M , L)$ is equivalent to specifying the conjugacy class of $H'$. From this it follows that $\widehat{A}_{\mathfrak{s}}(G)$ is a free abelian group with generators corresponding to conjugacy classes of subgroups of $G_{\mathfrak{s}}$ whose intersection with $S^1$ equals the identity.

\begin{example}
Let $G = \mathbb{Z}_p = \langle \sigma \rangle$ for a prime $p$. Then the only $S^1$ central extension of $G$ is the trivial one $S^1 \times \mathbb{Z}_p$. Then $\widehat{A}(G)$ is a free abelian group of rank $p+1$ with generators $x_0 , \dots , x_{p-1}, y$ defined as follows. For each $j$, $x_j = (pt , \mathbb{C}_j)$ is a point together with the $1$-dimensional representation $\mathbb{C}_j$ on which $\sigma$ acts as $\omega^j$, $\omega = e^{2\pi i/p}$ and $y$ is a free orbit. The ring structure of $\widehat{A}(G)$ is given by $x_i x_j = x_{i+j}$, $x_i y = y$, $y^2 = py$. So $\widehat{A}(G) \cong \mathbb{Z}[x,y]/( x^p-1 , xy-y , y^2-py)$, where $x = x_1$.
\end{example}

Suppose for simplicity that $G_{\mathfrak{s}} \cong S^1 \times G$ is split. Given $[(M,L)] \in \widehat{A}_{\mathfrak{s}}(G)$, we wish to describe the associated equivariant Seiberg--Witten invariants. Denote these by $SW_{(M,L)} : H^*_{G_{\mathfrak{s}}}(pt ; \mathbb{Z}) \to H^*_G(pt ; \mathbb{Z})$. By linearity, it suffices to consider the homogeneous case $M = G/H$ and $L$ is induced by a $1$-dimensional representation $\chi : H \to S^1$ of $H$. Let $c_1(\chi) \in H^2_H(pt ; \mathbb{Z})$ denote the first Chern class of this representation.

\begin{proposition}
We have
\[
SW_{ (M , L) }(x^m) = tr^G_H( c_1(\chi)^m)
\]
where $tr^G_H : H^*_H( pt ; \mathbb{Z}) \to H^*_G(pt ; \mathbb{Z})$ is the transfer map.
\end{proposition}
\begin{proof}
From the definition of the Seiberg--Witten invariants we have $SW_{(M,L)}(x^m) = p_*( c_1(L)^m )$, where $c_1(L) \in H^2_G( M ; \mathbb{Z})$ is the equivariant Chern class of $L$ and $p_* : H^*_G(M ; \mathbb{Z}) \to H^*_G(pt ; \mathbb{Z})$ is the push-forward map. Since $M = G/H$, we have $H^2_G( M ; \mathbb{Z}) \cong H^2_H( pt ; \mathbb{Z})$ and under this isomorphism $c_1(L)$ corresponds to $c_1(\chi) \in H^2_H(pt ; \mathbb{Z})$. Under the isomorphism $H^*_G( M ; \mathbb{Z}) \cong H^*_H(pt ; \mathbb{Z})$, the push-forward $p_* : H^*_H(pt ; \mathbb{Z}) \to H^*_G( pt : \mathbb{Z})$ is the transfer map $tr^G_H$. Hence we have $SW_{ (M , L) }(x^m) = tr^G_H( c_1(\chi)^m)$.
\end{proof}

\subsection{Free actions}\label{sec:trfree}

In this section we specialise to the case of free actions. We will see that the only obstruction to equivariant transversality is that the dimension of the moduli space be non-negative.

\begin{proposition}\label{prop:free}
Let $X$ be a compact, oriented, smooth $4$-manifold with $b_1(X) = 0$. Let $G$ be a finite group which acts smoothly, orientation preservingly and freely on $X$. Let $\mathfrak{s}$ be a spin$^c$-structure whose isomorphism class is fixed by $G$. If $d(X,\mathfrak{s}) \ge 0$ then equivariant transversality of the Seiberg--Witten moduli space can be achieved for $(X , \mathfrak{s})$ with respect to any chamber. More precisely, the conditions of Theorem \ref{thm:eqtr} are satisfied.
\end{proposition}
\begin{proof}
Consider $W = H^0(X ; \mathbb{R}) \oplus H^2(X ; \mathbb{R}) \oplus H^4(X \; \mathbb{R})$ as a representation of $G$. The Lefschetz index theorem gives $\chi_{\mathbb{R}}( W , g ) = 0$ for every $g \in G \setminus \{1\}$. Since the character of a real representation determines its isomorphism class \cite[Theorem 9.22]{isa}, it follows that $W$ is isomorphic to several copies of the regular representation. Thus for any finite subgroup $K \subseteq G$, and any $\chi \in \widehat{K}$, the multiplicity of $\chi$ in $W$ equals $m_\chi dim(W)/|K|$, where $m_\chi$ is the multiplicity of $\chi$ in the regular representation. Similarly if $S = H^+(X) - H^-(X)$ is the signature virtual representation of $G$, then the $G$-signature theorem gives $\chi_{\mathbb{R}}(S , g) = 0$ for any $g \in G \setminus \{1\}$ and hence for any $\chi \in \widehat{K}$, the multiplicity of $\chi$ in $S$ is $m_\chi \sigma(X)/|K|$. Since $2(H^+(X) + \mathbb{R}) = W + S$, it follows that for any $\chi \in \widehat{K} \setminus \{1\}$, the multiplicity of $\chi$ in $H^+(X)$ is $m_\chi (b_+(X)+1)/|K|$. Similarly, the $G_{\mathfrak{s}}$-equivariant Lefschetz index theorem for the spin$^c$ Dirac operator implies that $\chi_{\mathbb{C}}( D , g ) = 0$ for every $g \in G_{\mathfrak{s}} \setminus S^1$. Hence also $\chi_{\mathbb{R}}( D , g ) = 0$. If $s : K \to K_{\mathfrak{s}}$ is a splitting and we set $K' = s(K)$, then it follows that $D$ restricted to $K'$ is several copies of the regular representation. Hence for any $\chi \in \widehat{K}$, the multiplicity of $\chi$ in $D$ equals $m_\chi 2d/|K|$.

If $\chi \in \widehat{K} \setminus \{1\}$ then from what we have shown above, it follows that
\[
D_\chi - H^+_\chi = m_{\chi}( 2d - b_+- 1)/|K|  = m_{\chi} d(X,\mathfrak{s})/|K| \ge 0.
\]
Hence the conditions of Theorem \ref{thm:eqtr} are satisfied.
\end{proof}

Let $X$ be as in Proposition \ref{prop:free}. Let $H \subseteq G$ be a subgroup of $G$. The set of splittings $H \to H_{\mathfrak{s}}$ is in bijection with the set of isomorphism classes of spin$^c$-structures $\mathfrak{s}'$ on $X/H$ such that $q_H^*(\mathfrak{s}') \cong \mathfrak{s}$, where $q_H : X \to X/H$ is the projection. Indeed, given a splitting $H \to H_{\mathfrak{s}}$, we get an action of $H$ on the spinor bundles of $\mathfrak{s}$ which we can use to descend $\mathfrak{s}$ to a spin$^c$-structure on $X/H$ and conversely every spin$^c$-structure $\mathfrak{s}'$ on $X/H$ with $q_H^*(\mathfrak{s}') \cong \mathfrak{s}$ arises in this manner. Given a splitting $s : H \to H_{\mathfrak{s}}$, we let $\mathfrak{s}_s$ denote the corresponding spin$^c$-structure on $X/H$.

Assume now that $d(X,\mathfrak{s}) = 0$ and $b_+(X)^G > 0$. If $b_+(X)^G = b_+(X/G) = 1$, then fix a chamber $\phi$. Then we get a smooth moduli space $\mathcal{M}$, which is a finite, $G$-set. The action of $G$ on $H^+(X)$ is automatically orientation preserving. This is because for any non-trivial $g \in G$, the restiction of $G_{\mathfrak{s}}$ to the cyclic group $K = \langle g \rangle$ splits. Then for any $\chi \in \widehat{K} \setminus \{1\}$, the proof of Proposition \ref{prop:free} gives $H^+_\chi = D_\chi = 2 m_\chi (d/|K|)$, which is even because $|K|$ divides $d$ (since $D$ is isomorphic to several copies of the regular representation as a complex representation of $K$). Therefore, $g$ acts orientation preservingly on $H^+(X)$. A choice of orientation on $H^+(X)$ induces an orientation on $\mathcal{M}$, and $SW(X , \mathfrak{s})$ is equal to the number of points of $\mathcal{M}$, counted with orientation.

Let $\widetilde{\mathcal{M}} \to \mathcal{M}$ be the $S^1$-bundle over $\mathcal{M}$ corresponding to the line bundle $L$. Let $H$ be a subgroup of $G$ and $s : H \to H_{\mathfrak{s}}$ a splitting. Then it is clear that solutions to the Seiberg--Witten equations for $(X/H , \mathfrak{s}_s)$ (suitably perturbed) correspond to $sH$-invariant solutions on $X$, hence we can identify the moduli space of the Seiberg--Witten equations for $(X/H , \mathfrak{s})$ with $\mathcal{M}_s = \widetilde{\mathcal{M}}^{sH}/S^1$. In particular, it follows that $\mathcal{M}^H = \bigcup_s \mathcal{M}_s$, where the union is over splittings $s : H \to H_{\mathfrak{s}}$.

For a finite oriented set $M$ we write $|M| \in \mathbb{Z}$ for the number of points of $M$ counted with sign. Then it follows that
\[
| \mathcal{M}^H | = \sum_{s} | \mathcal{M}_s |,
\]
where $\mathcal{M}^H$ and $\mathcal{M}_s$ are oriented as subsets of $\mathcal{M}$ with the induced orientation. We would like to equate $|\mathcal{M}_s|$ with the Seiberg--Witten invariant $SW(X/H , \mathfrak{s}_s)$, however there are two subtleties that prevent us from immediately doing this. First, in order for $SW(X/H , \mathfrak{s}_s)$ to be well-defined we need to choose an orientation on $H^+(X/H) = H^+(X)^H$. Second, it is not immediately clear whether the two orientations on $\mathcal{M}_s$ (one as a subset of $\mathcal{M}$ and the other as a Seiberg--Witten moduli space on $X/H$) agree. To better understand this, we consider the problem in a more general context.

If $V_1, V_2$ are oriented vector spaces, then $V_1 \oplus V_2$ is oriented in the obvious way. Note that the order of the summands is important. Let $A,B$ be compact, oriented, smooth manifolds on which $G_{\mathfrak{s}}$ acts smoothly and orientation preservingly. Let $f : A \to B$ be a $G_{\mathfrak{s}}$-equivariant map. Let $y \in B^{G_{\mathfrak{s}}}$ be a regular value, assume that $S^1$ acts freely on $\widetilde{M} = f^{-1}(y)$ and set $M = \widetilde{M}/S^1$. We orient $\widetilde{M}$ according to the convention that
\[
TA|_{\widetilde{M}} \cong T_y B \oplus T\widetilde{M}
\]
and this induces an orientation on $M$ according to
\[
T\widetilde{M} = Ker(q_*) \oplus q^*(TM),
\]
where $q : \widetilde{M} \to M$ is the projection map, and the vertical tangent bundle $Ker(q^*) \cong Lie(S^1) \cong \mathbb{R}$ is oriented by choosing a generator of $Lie(S^1)$. Now let $H$ be a subgroup of $G$ and $s : H \to H_{\mathfrak{s}}$ a splitting. Set $H' = sH$. Restricting to the $H'$ fixed points we have an $S^1$-equivariant map
\[
f^{H'} : A^{H'} \to B^{H'}.
\]
Then $y$ is a regular value of $f^{H'}$ and $\widetilde{M}^{H'} = (f^{H'})^{-1}(y)$. Set $M^H = \widetilde{M}^{H'}/S^1$. Choose orientations on $B^{H'}$ and $A^{H'}$ (in the case of interest, this will correspond to choosing an orientation for $H^+(X/H)$). Then we get orientations on $\widetilde{M}^{H'}$ and $M^H$ according to 
\[
TA^{H'}|_{\widetilde{M}^{H'}} \cong T_y B^{H'} \oplus T\widetilde{M}^{H'}
\]
and
\[
T\widetilde{M}^{H'} = Ker(q^{H'}_*) \oplus (q^{H'})^*(TM^{H'})
\]
where $q^{H'} : \widetilde{M}^{H'} \to M^{H'}$ is the projection. Now suppose that $dim(M^{H'}) = dim(M)$. Then $M^{H'} \to M$ has codimension zero and hence inherits an orientation. So now we have two ways of orienting $M^{H'}$ and we need to compare these orientations. It suffices to compare the two corresponding orientations on $\widetilde{M}^{H'}$. Let us write $(T\widetilde{M}^{H'} , \mathfrak{o}_1)$ and $(T\widetilde{M}^{H'} , \mathfrak{o}_2)$ for $T\widetilde{M}^{H'}$ equipped with these two choices of orientation. Let $\nu(A , A^{H'})$ denote the normal bundle of $A^{H'}$ in $A$. The chosen orientations on $A$ and $A^{H'}$ induce an orientation on $\nu(A , A^{H'})$. Similarly, the orientations on $B$ and $B^{H'}$ induce an orientation on $(T_y B)_{H'}$. We have the following equalities:
\begin{align*}
TA^{H'}|_{\widetilde{M}^{H'}} &\cong T_y B^{H'} \oplus (T\widetilde{M}^{H'} , \mathfrak{o}_1) \\
TA|_{\widetilde{M}^{H'}} &\cong T_y B \oplus (T\widetilde{M}^{H'} , \mathfrak{o}_2) \\
TA|_{\widetilde{M}^{H'}} &\cong \nu(A , A^{H'})|_{\widetilde{M}^{H'}} \oplus TA^{H'}|_{\widetilde{M}^{H'}} \\
T_y B &\cong (T_y B)_{H'} \oplus T_y B^{H'}.
\end{align*}
Combined, these give
\begin{equation}\label{equ:tgt}
(T_y B)_{H'} \oplus T_y B^{H'} \oplus (T\widetilde{M}^{H'} , \mathfrak{o}_2) \cong \nu(A,A^{H'}) \oplus T_y B^{H'} \oplus (T\widetilde{M}^{H'} , \mathfrak{o}_1).
\end{equation}

Since $y$ is a regular value of $f$, it follows that the derivative of $f$ yields an isomorphism $\varphi : \nu( A , A^{H'} )|_{\widetilde{M}^{H'}} \to (T_y B)_{H'}$. Since $\nu(A,A^{H'})$ and $(T_y B)_{H'}$ are oriented, for each $m \in \widetilde{M}^{H'}$ we can define $sign(det(\varphi))(m)$ to be $+1$ or $-1$ according to whether $\varphi$ is orientation preserving or reversing at $m$. Note that $sign(det(\varphi))(m)$ only depends on the image of $m$ in $M^{H'}$. Then it follows from (\ref{equ:tgt}) that the orientations $\mathfrak{o}_1$ and $\mathfrak{o}_2$ differ at $m$ by $sign(det(\varphi))(m)$.

Note that if we reverse the chosen orientation on $A^{H'}$ or $B^{H'}$ then $sign(det(\varphi))$ changes by an overall sign. If $sign(det(\varphi))$ is independent of $m$, then we can choose the orientations on $A^{H'}$ and $B^{H'}$ so that the two orientation on $M^{H'}$ agree. Assume now that $A = S^{V,U}$ where as usual $V$ is a complex representaiton of $G_{\mathfrak{s}}$ and $U$ is a real representation of $G$. Then $A^{H'} = S^{V^{H'} , U^{H'}}$ and hence $\nu(A , A^{H'}) \cong V_{H'} \oplus U_{H'}$. This means that for each $m \in M^{H'}$, $\varphi_m$ can be interpreted as a map $\varphi_m : R \to S$, where $R = V_{H'} \oplus U_{H'}$, $S = (T_y B)_{H'}$. Observe that $R,S$ are representations of $H' \cong H$ and that $\varphi$ is $H$-equivariant. Furthermore, $R,S$ contain no copies of the trivial representation.

\begin{lemma}\label{lem:or}
Let $G$ be a finite group with the property that every non-trivial real irreducible representation $\chi$ has $d_\chi \neq 1$. Then
\begin{itemize}
\item[(1)]{Every real representation $R$ of $G$ with $R^G = 0$ admits a $G$-invariant complex structure.}
\item[(2)]{Let $V$ be a complex representation of $G$ with $V^G = 0$. Then any isomorphism $\varphi : V \to V$ of $V$ as a real representation (forgetting the complex structure) has $det(\varphi)>0$.}
\end{itemize}
\end{lemma}
\begin{proof}
(1) decompose $R$ into a sum of irreducible representations. Each summand $R_i$ is non-trivial since $R^G = 0$. Hence by the assumption on $G$ we have $End_{G,\mathbb{R}}(R_i) \cong \mathbb{C}$ or $\mathbb{H}$. In either case, $R_i$ admits a $G$-invariant complex structure.
(2) decomposing $V$ into isotypical components, it suffices to consider the case $V \cong W^{\oplus n}$ for some non-trivial irreducible complex representation $W$. By the assumption on $G$, the underlying real representation of $W$ is irreducible and $End_{G,\mathbb{R}}(V) \cong M_n(\mathbb{C})$, $M_n(\mathbb{H})$. Hence the space of real isomorphisms is isomorphic to $GL(n,\mathbb{C})$ or $GL(n,\mathbb{H})$. In either case, the group is connected and hence $det(\varphi) > 0$.
\end{proof}

Assume now that $M^{H'}$ is non-empty and that $H$ satisfies the hypothesis of Lemma \ref{lem:or}. Then $R$ and $S$ are isomorhpic as real representations of $H$ since $\varphi_m: R \to S$ for any $m \in M^{H'}$ is an isomorphism. By Lemma \ref{lem:or} (1), we can choose $H$-invariant complex structures on $R$ and $S$ so that $R$ and $S$ are isomorphic as complex representations (choose a complex structure on $S$ and then define the complex structure on $R$ to be the pullback under $\varphi_m$ for some $m$). The complex structures on $R,S$ define orientations and then by Lemma \ref{lem:or} (2), we have that $sign(det(\varphi))(m) = 1$ for every $m \in M^{H'}$.

Translating back to the case that $f : S^{V,U} \to S^{V',U'}$ is a $G$-equivariant monopole map. We have proven the following:

\begin{proposition}\label{prop:compareorn}
Let $f : S^{V,U} \to S^{V',U'}$ be a $G$-equivariant monopole map. Let $y \in (S^{U'} \setminus S^U)^{G_\mathfrak{s}}$ be a regular value and set $\mathcal{M} = f^{-1}(y)/S^1$. Let $H$ be a subgroup of $G$ and $s : H \to G_{\mathfrak{s}}$ a splitting. Assume that $\mathcal{M}^{H'}$ is non-empty and that $dim(\mathcal{M}) = dim( \mathcal{M}^{H'})$.

Assume that $G$ acts orientation preservingly on $H^+$ and fix a choice of orientation. Suppose that $d_\chi \neq 1$ for every non-trivial real irreducible representation $\chi$ of $H$. Then there is an orientation on $H^+_H$ such that the two orientations on $\mathcal{M}^{H'}$ coincide. The orientation on $H^+_H$ is characterised by the property that it is induced by a $H$-invariant complex structure for which $H^+_H \cong D_H$ as complex representations (here we use the splitting $s : H \to H_{\mathfrak{s}}$ to regard $D$ as a representation of $H$).
\end{proposition}

Applying Proposition \ref{prop:compareorn} in the case of a free action, we get the following:

\begin{theorem}\label{thm:fixH}
Let $X$ be a compact, oriented, smooth $4$-manifold with $b_1(X) = 0$. Let $G$ be a finite group which acts smoothly, orientation preservingly and freely on $X$. Let $\mathfrak{s}$ be a spin$^c$-structure whose isomorphism class is fixed by $G$ and suppose that $d(X,\mathfrak{s}) \ge 0$. Assume that $b_+(X)^G > 0$ and if $b_+(X)^G = 1$ then fix a chamber $\phi$. 

Fix an orientation on $H^+(X)$. Then the moduli space $\mathcal{M}$ for $(X , \mathfrak{s} , \phi)$ is a finite oriented $G$-set. Let $H \subseteq G$ be a subgroup of $G$. If $\mathcal{M}^H$ is non-empty, then there is a distinguished choice of orientation on $H^+(X/H)$ such that
\[
| \mathcal{M}^H | = \sum_{ \mathfrak{s}' } SW(X/H , \mathfrak{s}' )
\]
where the sum is over the spin$^c$-structures $\mathfrak{s}'$ on $X/H$ whose pullback to $X$ is isomorphic to $\mathfrak{s}$ and $SW(X/H , \mathfrak{s}')$ is defined using the chamber $\phi$ in the case $b_+(X/H) = 1$.

The orientation on $H^+(X/H)$ is defined as follows. If $\mathcal{M}^H$ is non-empty, then there exists a splitting $s : H \to H_{\mathfrak{s}}$. Then we require $H^+(X) \cong D_H \oplus H^+(X/H)$ as oriented representations of $H$, where we use $s$ to regard $D$ as a representation of $H$.
\end{theorem}

Let $M$ be a finite oriented $G$-set. Then we have
\begin{equation}\label{equ:cong1}
\sum_{(H) \text{ cyclic}} \phi(|H|) |G/N_G H| | M^H| = 0 \; ({\rm mod} \; |G|)
\end{equation}
where the sum is over conjugacy classes of cyclic subgroups of $G$ and $\phi$ is the Euler $\phi$ function \cite[Chapter IV, \textsection 5]{die}. Let $(C , \le )$ denote the set of subgroups of $G$ partially ordered by inclusion. Let $\mu_G( H_1 , H_2 )$ denote the M\"obius function of $(C,\le)$ and set $\mu_G(H) = \mu_G(1,H)$. Then we have \cite[Chapter IV, Proposition 5.10]{die}:
\begin{equation}\label{equ:cong2}
\sum_{H \in C} \mu_G(H) |M^H| = 0 \; ({\rm mod} \; |G|).
\end{equation}

Now we apply the above two congruences to the case that $M$ is the $G$-equivariant moduli space for a free action on $X$. This gives a congruence relation between the Seiberg--Witten invariants of $X$ and the Seiberg--Witten invariants of quotients of $X$.
\begin{theorem}\label{thm:swcong}
Let $X$ be a compact, oriented, smooth $4$-manifold with $b_1(X) = 0$. Let $G$ be a finite group of odd order which acts smoothly and freely on $X$. Let $\mathfrak{s}$ be a spin$^c$-structure whose isomorphism class is fixed by $G$ and suppose that $d(X,\mathfrak{s}) \ge 0$. Assume that $b_+(X)^G > 0$ and if $b_+(X)^G = 1$ then fix a chamber $\phi$. 

Fix an orientation on $H^+(X)$. Then the moduli space $\mathcal{M}$ for $(X , \mathfrak{s} , \phi)$ is a finite oriented $G$-set. Let $H \subseteq G$ be a subgroup of $G$. If $\mathcal{M}^H$ is non-empty, then let $H^+(X/H)$ be oriented as in Theorem \ref{thm:fixH}. Then we have
\[
\sum_{(H) \text{ cyclic}} \phi(|H|) |G/N_G H| \sum_{\mathfrak{s}' | q_H^*(\mathfrak{s}') \cong \mathfrak{s}} SW(X/H , \mathfrak{s}') = 0 \; ({\rm mod} \; |G|)
\]
and
\[
\sum_{H \subseteq G} \mu_G(H)  \sum_{\mathfrak{s}' | q_H^*(\mathfrak{s}') \cong \mathfrak{s}} SW(X/H , \mathfrak{s}') = 0 \; ({\rm mod} \; |G|)
\]
where $SW(X/H , \mathfrak{s}')$ is defined using the chamber $\phi$ in the case $b_+(X/H) = 1$.
\end{theorem}
\begin{proof}
Since $G$ has odd order, then by Remark \ref{rem:odd} every non-trivial real irreducible representation $\chi$ of any subgroup of $G$ has $d_\chi \neq 1$, hence Theorem \ref{thm:fixH} applies to every subgroup of $G$. Applying this to Equations (\ref{equ:cong1}), (\ref{equ:cong2}) gives the result.
\end{proof}

\section{Reduced Seiberg--Witten invariants}\label{sec:red}

Let $f : S^{V,U} \to S^{V',U'}$ be a $G_{\mathfrak{s}}$-equivariant monopole map. Given a subgroup $H \subseteq G$ we let $H_{\mathfrak{s}}$ denote the preimage of $H$ under $G_{\mathfrak{s}} \to G$. Then $H_{\mathfrak{s}}$ is a central extension of $H$ by $S^1$. Suppose that $H_{\mathfrak{s}}$ is a split extension and let $s : H \to H_{\mathfrak{s}}$ denote a splitting. Let $N_G(H)$ and $N_{G_{\mathfrak{s}}}(H_{\mathfrak{s}})$ denote the normalisers of $H$ and $H_{\mathfrak{s}}$ in $G$ and $G_{\mathfrak{s}}$. Then $N_{G_{\mathfrak{s}}}(H_{\mathfrak{s}})$ is the preimage of $N_G(H)$ under $G_{\mathfrak{s}} \to G$. Let $W = N_G(H)/H$ and $W_{\mathfrak{s}} = N_{G_{\mathfrak{s}}}( H_{\mathfrak{s}} )/ s(H)$. Then $W_{\mathfrak{s}}$ is a central extension of $W$ by $S^1$. Let $f^{s} : S^{V^{sH},U^{sH}} \to S^{(V')^{sH},(U')^{sH}}$ be the restriction of $f$ to the fixed point set of $sH$. Then $f^{s}$ is $N_{G_\mathfrak{s}}(H_{\mathfrak{s}})$-equivariant. Furthermore, $sH$ acts trivially, so this action descends to $W_{\mathfrak{s}}$. Thus $f^{s}$ has the structure of a $W_{\mathfrak{s}}$-equivariant monopole map and we can consider the $W$-equivariant Seiberg--Witten invariants of $f^{s}$.

If $f$ is the Bauer--Furuta map associated to a $G$-action on a pair $(X , \mathfrak{s})$, then $f^{s}$ is a finite dimensional approximation of the $sH$-invariant part of the Seiberg--Witten monopole map. We will call $f^{s}$ the {\em $H$-reduced Bauer--Furuta invariant of $(X,\mathfrak{s})$ with respect to the splitting $s$}. Similarly the Seiberg--Witten invariants of $f^{s}$ will be called the {\em $H$-reduced Seiberg--Witten invariants of $(X,\mathfrak{s})$.}

Suppose $H = G$. If we forget the action of $W$ then $f^{s}$ can simply be viewed as an $S^1$-equivariant map. It is a finite-dimensional approximation of the Seiberg--Witten equations restricted to $sG$-invariant configurations. Let $b_+^G = dim_{\mathbb{R}}((H^+)^G)$. If $b_+^G > 0$, then it is known (eg, \cite[Theorem 4.1]{cho}) that for generic $G$-invariant perturbations, the moduli space $\mathcal{M}$ of $sG$-invariant solutions to the Seiberg--Witten equations for $(X,\mathfrak{s})$ is a compact, smooth manifold of dimension $2d^{sG} - b_+^G - 1$, where $d^{sG} = dim_{\mathbb{C}}( V^{sG}) - dim_{\mathbb{C}}( (V')^{sG})$. Over this moduli space is a line bundle $L \to \mathcal{M}$ which is constructed in the same way as for the usual Seiberg--Witten moduli space. If $2d^{sG} - b_+^G - 1 = 2m$ is even and non-negative, then one can evaluate $x^m$ over $\mathcal{M}$, where $x = c_1(L)$, to obtain an integer invariant of the $G$-action. In terms of the reduced monopole map $f^s$, the invariant is given by $SW_{f^{s}}( x^m )$ (if $b_+^G = 1$ we need to choose a chamber). We will denote $SW_{f^s}(x^m)$ by $\overline{SW}^{s}_{G,X,\mathfrak{s}}$ and call it the {\em reduced Seiberg--Witten invariant of $(X , \mathfrak{s})$ (with respect to the $G$-action and splitting $s$)}. The reduced invariant $\overline{SW}^{s}_{G,X,\mathfrak{s}}$ was studied in \cite{sung0} where it was called the {\em $G$-monopole invariant} and denoted $SW^G_{X,\mathfrak{s}}$.

In the case of a free action, the reduced invariant corresponds to the usual Seiberg--Witten invariant of the quotient space:

\begin{proposition}\label{prop:freeinv}
Let $X$ be a compact, oriented, smooth $4$-manifold with $b_1(X) = 0$. Let $G$ act on $X$ smoothly and orientation preservingly and let $\mathfrak{s}$ be a $G$-invariant spin$^c$-structure on $X$. Let $H$ be a subgroup of $G$ such that $H$ acts freely on $X$ and let $s  : H \to H_{\mathfrak{s}}$ be a splitting. Let $f$ be the $G$-equivariant Bauer--Furuta monopole map of $(X, \mathfrak{s})$. Let $\mathfrak{s}_s$ be the spin$^c$-structure on $X/H$ determined by $s$ (as in Section \ref{sec:trfree}). Suppose that $b_+(X/H) > 0$ and that $d(X/H , \mathfrak{s}_s) = 2m$ is even and non-negative. Then
\[
SW_{f^s}(x^m) = SW(X/H , \mathfrak{s}_s),
\]
where if $b_+(X/H) = 1$ both invariants are defined with respect to a chamber of $H^+(X)^H \cong H^+(X/H)$.
\end{proposition}
\begin{proof}
Since $H$ acts freely on $X$, then the reduction $f^s$ is precisely the Bauer--Furuta monopole map of $(X/H , \mathfrak{s}_s)$. Hence the result follows.
\end{proof}

The equivariant and reduced Seiberg--Witten invariants are related to one another. To understand the relation we will use localisation in cohomology and $K$-theory. It turns that localisation behaves much better in $K$-theory than cohomology, which is one motivation for considering $K$-theoretic invariants. 
 
\subsection{Localisation in cohomology}\label{sec:locc}

In order to apply localisation to the cohomological invariants we will restrict attention to the case that $G = \mathbb{Z}_n = \langle g \rangle$, is a cyclic group of order $n = p^k$ for some prime $p$. We have $H^*_G( pt ; \mathbb{Z}) \cong \mathbb{Z}[v]/(nv)$ where $deg(v) = 2$. In this case $G_{\mathfrak{s}}$ is split. Choose a splitting $G_{\mathfrak{s}} \cong S^1 \times G$. Let $f : S^{V,U} \to S^{V',U'}$ be the monopole map. Let $\phi$ be a chamber.

Let $\sigma = g^{n/p}$, $H = \langle \sigma \rangle \cong \mathbb{Z}_p$. Let $w = w_1(H^+) \in H^1(G ; \mathbb{Z}_2)$. Note that $w=0$ if $n \neq 2$. Suspending if necessary we can assume that $G$ acts orientation preservingly on $U$. Let $\tau^\phi_{V',U'} \in H^{2a'+b'}_{G_{\mathfrak{s}}}( S^{V',U'} , S^{U} ; \mathbb{Z}_w )$. Then $f^*( \tau^\phi_{V' , U'} ) = (-1)^b \delta \tau_U \eta^\phi$ for some $\eta^\phi \in H^*_G( \mathbb{P}(V) ; \mathbb{Z}_w)$ and $SW^\phi_{G,f}( \theta ) = (\pi_{\mathbb{P}(V)})_*( \eta^\phi \theta )$, where $\pi_{\mathbb{P}(V)} : \mathbb{P}(V) \to pt$ is the projection map.

We have a decomposition $V = \bigoplus_{i=0}^{p-1} V_i$ where $\sigma$ acts on $V_i$ as multiplication by $\omega^i$, $\omega = e^{2\pi i/p}$. Then it follows that
\[
\mathbb{P}(V)^H = \bigcup_{i=0}^{n-1} \mathbb{P}(V_i).
\]
Consider the multiplicative subset of $H^*_G(pt ; \mathbb{Z})$ generated by $v$. The localisation theorem \cite[III (3.8)]{die} implies that the restriction map $\iota^* : H^*_G( \mathbb{P}(V) ; \mathbb{Z}_w ) \to H^*_G( \mathbb{P}(V)^H ; \mathbb{Z}_w )$ becomes an isomorphism upon localising with respect to $v$. Let $N_j$ denote the normal bundle of $\mathbb{P}(V_j)$ in $\mathbb{P}(V)$. Then by the Euler sequence, we have $N_j \cong (V/V_j)(1)$, where for a $G$-equivariant vector bundle $W$ on $\mathbb{P}(V)$, $W(m)$ denotes $W \otimes \mathcal{O}_V(m)$.

\begin{lemma}\label{lem:einv}
For each $j$, $e_G(N_j)$ is invertible in the localised ring $v^{-1}H^*_G( \mathbb{P}(V_j) ; \mathbb{Z} )$.
\end{lemma}
\begin{proof}
For $u \in \mathbb{Z}_n$, let $a_u$ denote the dimension of the $e^{2\pi i u/n}$-eigenspace of $g$ acting on $V$. Then 
\[
e_G(V_j(1)) = e_{G_{\mathfrak{s}}}(V_j) = \! \! \prod_{u \in \mathbb{Z}_n | u = j \; ( {\rm mod} \; p) } \! \! \! \! \! \! \! \! \! \! \! \! \! \! (x + uv)^{a_u}, \quad e_G(N_j) = \! \! \prod_{u \in \mathbb{Z}_n | u \neq j \; ( {\rm mod} \; p) } \! \! \! \! \! \! \! \! \! \! \! \! \! \! (x + uv)^{a_u}.
\]
Hence to show that $e(N_j)$ is invertible in $v^{-1}H^*_G( \mathbb{P}(V_j) ; \mathbb{Z})$, it suffices to show that $(x+uv)$ is invertible whenever $u \neq j \; ({\rm mod} \; p)$. We have that $v^{-1} H^*_G( \mathbb{P}(V_j) ; \mathbb{Z} ) \cong v^{-1} \mathbb{Z}_n[x,v]/( e_G(V_j(1)) )$. We first claim that $x+jv$ is nilpotent in $H^*_G(\mathbb{P}(V_j) ; \mathbb{Z} )$. To see this, rewrite the above product for $_G(V_j(1))$ as
\[
e_G(V_j(1)) =  \! \! \prod_{u \in \mathbb{Z}_{n/p} } \! \! (x+jv + puv)^{a_{pu+j}}.
\]
Expanding the product in terms of powers of $(x+jv)$, we get that $e_G(V_j(1)) = (x+jv)^{r_j} + b_1 (x+jv )^{r_j - 1} + \cdots + b_{r_j}$ for some $r_j$. Moreover, $b_1, \dots , b_{r_j}$ are multiples of $p$. Hence $e(V_j) = (x+jv)^{r_j} - p q(x,v)$ for some polynomial $q(x,v)$ in $x$ and $v$. Now in $H^*_G( \mathbb{P}(V_j) ; \mathbb{Z})$, we have $e_G(V_j(1)) = 0$ and hence $(x+jv)^{r_j} = p q(x,v)$. Therefore, $(x+jv)^{k r_j} = p^k q(x,v)^{k} = 0$.

Now if $u \neq j \; ({\rm mod} \; p)$, then we can write $(x+uv) = x+jv + (u-j)v = y+cv$ where $y = x+jv$ and $c=u-j$. Then since $y=x+jv$ is nilpotent and $c$ is invertible, the following expression for $(x+uv)^{-1}$ is a well-defined element of  $v^{-1} H^*_G( \mathbb{P}(V_j) ; \mathbb{Z})$:
\[
(x+uv)^{-1} = (y+cv)^{-1} = c^{-1} v^{-1} - c^{-2}v^{-2}y + c^{-3} v^{-3} y^2 - \cdots
\]
\end{proof}

\begin{remark}
The proof of Lemma \ref{lem:einv} shows that $e_G(N_j)^{-1}$ can be expressed as a formal power series in $v^{-1}$ with coefficients given by polynomials in $x$ and $v$. Modulo $e_G(V_j(1))$, all but finitely terms in the series vanish, giving a well-defined element of $v^{-1} H^*_G( \mathbb{P}(V_j) \; \mathbb{Z}) \cong v^{-1} \mathbb{Z}_n[x,v]/(e_G(V_j(1)))$.
\end{remark}

Let $\iota_j : \mathbb{P}(V_j) \to \mathbb{P}(V)$ denote the inclusion of $\mathbb{P}(V_j)$ in $\mathbb{P}(V)$. Let $\theta \in H^*_G( pt ; \mathbb{Z})$. Consider the following expression:
\[
\mu = \sum_{j=0}^{p-1} (\iota_j)_*( e_G(N_j)^{-1} \iota_j^*(\eta^\phi) \theta ) \in v^{-1} H^*_G( \mathbb{P}(V) ; \mathbb{Z}_w).
\]
Then
\begin{align*}
\iota^*(\mu) &= \sum_{j=0}^{p-1} \iota_j^* (\iota_j)_*( e_G(N_j)^{-1} \iota_j^*(\eta^\phi) \theta) \\
&= \sum_{j=0}^{p-1} e_G(N_j) e_G(N_j)^{-1} \iota^*_j( \eta^\phi) \theta \\
&= \iota^*( \eta^\phi \theta). 
\end{align*}
The localisation theorem implies that $\mu = \eta^\phi \theta$ in $v^{-1}H^*_G( \mathbb{P}(V) ; \mathbb{Z}_w)$. Therefore we have
\begin{align*}
SW^\phi_{G,f}(\theta ) &= (\pi_{\mathbb{P}(V)})_* ( \eta^\phi \theta ) \\
&= (\pi_{\mathbb{P}(V))})_* (\mu ) \\
&= (\pi_{\mathbb{P}}(V))_* \left( \sum_{j=0}^{p-1} (\iota_j)_* ( e_G(N_j)^{-1} \iota_j^*(\eta^\phi) \theta ) \right) \\
&= \sum_{j=0}^{p-1} (\pi_{\mathbb{P}(V_j)})_*\left( e_G(N_j)^{-1} \iota_j^*(\eta^\phi) \theta \right)
\end{align*}
where $\pi_{\mathbb{P}(V_j)} : \mathbb{P}(V_j) \to pt$ is the projection map.

Recall that $H = \langle \sigma \rangle$. Let $s_j : H \to H_{\mathfrak{s}} \subset S^1 \times G$ denote the splitting given by $s_j( \sigma ) = ( \omega^{-j} , \sigma)$, where $\omega = e^{2\pi i/p}$. Let $f^{s_j} : S^{V_j,U^H} \to S^{(V'_j) , (U')^H}$ be the restriction of $f$ to the fixed points of $s_j(H)$. Define $\eta^\phi_j \in H^*_G( \mathbb{P}(V_j) ; \mathbb{Z})$ by $(f^{s_j})^* ( \tau^\phi_{V'_j , (U')^H} ) = (-1)^{b^H} \delta \tau_{U^H} \eta^\phi_j$ where $b^H$ denotes the rank of $U^H$. Then the $G$-equivariant Seiberg--Witten invariants of $f^{s_j}$ are given by $SW_{G,f^{s_j}}^\phi( \theta ) = (\pi_{\mathbb{P}(V_j)})_*( \eta^\phi_j \theta )$.

\begin{lemma}\label{lem:etaj}
The following equality holds in $H^*_G( \mathbb{P}(V_j) ; \mathbb{Z}_w)$:
\[
\iota^*_j(\eta^\phi) =  \eta^\phi_j e_{G_{\mathfrak{s}}}( V'/V'_j ) e_G( H^+/ (H^+)^H).
\]
\end{lemma}
\begin{proof}
By the definition of $f^{s_j}$ we have a commutative diagram:
\[
\xymatrix{
S^{V,U} \ar[r]^-{f} & S^{V' , U'} \\
S^{V_j,U^H} \ar[r]^-{f^{s_j}} \ar[u]^-{i} & S^{(V'_j) , (U')^H} \ar[u]^-{i}
}
\]
where the vertical arrows are inclusions and $V'_j$ denotes the $\omega^j$-eigenspace of $\sigma$ acting on $V$. Note that all spaces in this diagram are acted upon by $G_{\mathfrak{s}}$ and the maps are $G_{\mathfrak{s}}$-equivariant. Recall that $f^*( \tau^\phi_{V',U'}) = (-1)^b \delta \tau_U \eta^\phi$. Therefore
\begin{align*}
i^* f^*(\tau^\phi_{V',U'}) &= (-1)^b i^*( \delta \tau_U \eta^\phi ) \\
&= (-1)^b \delta \tau_{U^H} e_G(U/U^H) \iota^*_j(\eta^\phi) \\
\end{align*}
where we choose orientations on $U^H$ and $U/U^H$ so that $U = U^H \oplus (U/U^H)$ as oriented vector spaces. Next, we have
\begin{align*}
(f^{s_j})^* i^*( \tau^\phi_{V',U'}) &= (-1)^{b^H_+(b-b^H)} (f^{s_j})^* ( \tau^\phi_{V'_j , (U')^H} e_{G_{\mathfrak{s}}}( V'/V'_j ) e_G(U/U^H) e_G( H^+/(H^+)^H)) \\
&= (-1)^{b^H_+ (b-b^H)+b^H} \delta_{U^H} \eta^\phi_j e_{G_{\mathfrak{s}}}( V'/V'_j ) e_G(U/U^H) e_G( H^+/(H^+)^H).
\end{align*}
The sign factor $(-1)^{b^H_+(b-b^H)}$ in the first line is because $U' = U^H \oplus (U/U^H) \oplus (H^+)^H \oplus (H^+/(H^+)^H)$. To re-write this as $(U')^H \oplus (U/U^H) \oplus (H^+/(H^+)^H)$ we need to swap the $U/U^H$ and $(H^+)^H$ factors.

Since $f \circ i = i \circ f^{s_jH}$, we have $i^* f^*( \tau^\phi_{V',U'}) = (f^{s_jH})^* i^*( \tau^\phi_{V',U'})$ and hence
\begin{align*}
(-1)^{b} e_{G}( U/U^H ) \iota^*_j(\eta^\phi) &= (-1)^{b^H_+ (b-b^H)+b^H}  \eta^\phi_j e_{G_{\mathfrak{s}}}( V'/V'_j ) e_G(U/U^H) e_G( H^+/(H^+)^H) \\
&= (-1)^{b} e_G(U/U^H)  e_{G_{\mathfrak{s}}}( V'/V'_j ) \eta^\phi_j e_G( H^+/(H^+)^H).
\end{align*}
Upon cancelling factors of $e_G(U/U^H)$ (which is possible since $e_G(U/U^H)$ not a zero divisor), we obtain
\[
\iota^*_j(\eta^\phi) =  \eta^\phi_j e_{G_{\mathfrak{s}}}( V'/V'_j ) e_G( H^+/ (H^+)^H).
\]
\end{proof}

Let $D_j$ denote the $\omega^j$ virtual eigenspace of $\sigma$ acting on $D$.

\begin{theorem}\label{thm:loc}
Let $f$ be a $G = \mathbb{Z}_n$-monopole map, where $n = p^k$ is a prime power. Then
\[
SW^\phi_{G,f}(\theta) = \sum_{j=0}^{p-1} SW^\phi_{G,f^{s_j}}(  e_G(H^+/(H^+)^\sigma)  e_{G_{\mathfrak{s}}}(D/D_j)^{-1} \theta ).
\]
\end{theorem}
\begin{proof}
We have already shown that
\[
SW^\phi_{G,f}(\theta ) = \sum_{j=0}^{p-1} (\pi_{\mathbb{P}(V_j)})_*\left( e_G(N_j)^{-1} \iota_j^*(\eta^\phi) \theta \right).
\]
Then applying Lemma \ref{lem:etaj} and using $e_G(N_j) = e_{G_\mathfrak{s}}( V/V_j)$, we get 
\begin{align*}
SW^\phi_{G,f}(\theta ) &= \sum_{j=0}^{p-1} (\pi_{\mathbb{P}(V_j)})_*\left( \eta^\phi_j  e_{G_\mathfrak{s}}(V/V_j)^{-1} e_G( H^+/ (H^+)^\sigma) e_{G_\mathfrak{s}}( V'/V'_j ) \theta \right) \\ 
&= \sum_{j=0}^{p-1} ( \pi_{\mathbb{P}(V_j)} )_* \left( \eta^\phi_j e_G( H^+/ (H^+)^\sigma)  e_{G_{\mathfrak{s}}}^{-1}  \theta \right) \\
&=  \sum_{j=0}^{p-1} SW^\phi_{G,f^{s_j}} \left( e_G( H^+/ (H^+)^\sigma) e_{G_{\mathfrak{s}}}(D/D_j)^{-1} \theta \right).
\end{align*}

\end{proof}

\subsection{Localisation in $K$-theory}\label{sec:lock}

We now consider localisation of the $K$-theoretic invariants. For a finite group $G$, any class in $\mu \in R(G)$ is completely determined by its character. But for any $g \in G$, $\chi(\mu , g)$ is determined by the restriction of $\mu$ to $R(H)$, where $H = \langle g \rangle$ is the cyclic subgroup of $G$ generated by $g$. For this reason it suffices to consider localisation only for the case that $G$ is a cyclic group.

Let $G = \mathbb{Z}_n = \langle g \rangle$ be a cyclic group of order $n$, for any positive integer $n$. We will localise to the fixed points of $g$. For each $j \in \mathbb{Z}_n$ let $V_j$ denote the $\omega^j$-eigenspace of $g$ acting on $V$, where $\omega = e^{2\pi i j/n}$. Fix a splitting $G_{\mathfrak{s}} \cong S^1 \times G$. For each $j \in \mathbb{Z}_n$, we define a splitting $s_j : G \to S^1 \times G$ by $s_j(g) = ( \omega^{-j} , g)$. Let $S \subset R(G)$ be the multiplicative subset of elements $W \in R(G)$ such that $\chi(W,g) \neq 0$. Evaluating characters at $g$ defines an injective map $\chi( \; , g) : S^{-1}R(G) \to \mathbb{C}$ whose image is $\mathbb{Q}(\omega)$, the $n$-cyclotomic field. If $W$ is a complex representation of $G$ which contains no copies of the trivial representation then $\chi( e^K_G(W) , g ) = det( 1 - g^{-1} : W \to W) \neq 0$, so $e^K_G(W)$ is invertible in $S^{-1}R(G)$. This means that the proof of Lemma \ref{lem:etaj} carries over to the $K$-theoretic setting and gives:
\[
\iota^*_j(\eta^{K,\phi}) = e^K_G( H^+/ (H^+)^H) e^K_{G_{\mathfrak{s}}}( V'/V'_j ) \eta^{K,\phi}_j.
\]

Let $N_j$ denote the normal bundle of $\mathbb{P}(V_j)$ in $\mathbb{P}(V)$, so $N_j = (V/V_j)(1)$. The $K$-theoretic equivalent of Lemma \ref{lem:einv} is the following:
\begin{lemma}\label{lem:einvk}
For each $j$, $e_G^K(N_j)$ is invertible in the localised ring $S^{-1}K^*_G( \mathbb{P}(V_j) )$.
\end{lemma}
\begin{proof}
Let $a_j$ denote the rank of $V_j$. Then $e^K_G(V_j(1)) = (1-\tau^{-j} \xi^{-1})^{a_j}$, where $\xi$ denote the standard $1$-dimensional representation of $S^1$. Hence $K^*_G( \mathbb{P}(V_j) ) \cong R(G)[\xi,\xi^{-1}]/( e^K_G( V_j(1)))$. Hence 
\[
S^{-1}K^*_G( \mathbb{P}(V_j) ) \cong \frac{\mathbb{Q}(\omega)[\xi,\xi^{-1}]}{ (1 - \tau^{-j} \xi^{-1} )^{a_j} ) }.
\]
Let $\mathbb{C}_j$ denote the $1$-dimensional representation of $G$ where $g$ acts as multiplication by $\omega^j$. Then $N_j$ is a direct sum of line bundles of the form $\mathbb{C}_i(1)$ where $i \neq j$. So it suffices to show that $e^K_G( \mathbb{C}_i(1)) = 1 - \omega^{-i} \xi^{-1}$ is invertible. We have
\begin{align*}
1 - \omega^{-i}\xi^{-1} &= 1 - \omega^{-(i-j)} + \omega^{-(i-j)}(1 - \omega^{-j}\xi^{-1}) \\
&= (1-\omega^{-(i-j)})( 1 - u )
\end{align*}
where $u = (1-\omega^{i-j})^{-1}(1 - \omega^{-j} \xi^{-1})$. Then $u^{a_j} = 0$ in $S^{-1}K^*_G( \mathbb{P}(V_j) )$ because $0 = e^K_G( V_j(1) ) = (1-\omega^{-j} \xi^{-1})^{a_j}$. Therefore
\[
(1-\omega^{-i}\xi^{-1})^{-1} = (1-\omega^{-(i-j)})^{-1}( 1 + u + u^2 + \cdots ),
\]
where there are only finitely many terms on the right since $u$ is nilpotent.
\end{proof}

Equipped now with the $K$-theoretic equivalents of Lemmas \ref{lem:einv} and \ref{lem:etaj} and using the localisation theorem in $K$-theory, we obtain the $K$-theoretic equivalent of Theorem \ref{thm:loc}:

\begin{theorem}\label{thm:lock}
Let $f$ be a $G$-monopole map. Assume that $\mathbb{R} \oplus H^+$ can be given a $G$-invariant complex structure which we use to $K$-orient $H^+$. Then
\[
\chi( SW^{K,\phi}_{G,f}(\theta) , g ) = \chi( e^K_G(H^+/(H^+)^g) , g) \sum_{j=0}^{n-1} \chi( SW^{K,\phi}_{G,f^{s_j}}( e^K_{G_{\mathfrak{s}}}(D/D_j)^{-1} \theta , g).
\]
\end{theorem}

\subsection{Free actions revisited}\label{sec:far}

In this section we will apply Theorem \ref{thm:lock} to the case of free actions. Let $X$ be a compact, oriented, smooth $4$-manifold with $b_1(X) = 0$. Let $G$ be a group which acts smoothly, orientation preservingly and freely on $X$. Assume that $b_+(X)^G = b_+(X/G) > 0$. Let $\mathfrak{s}$ be a $G$-invariant spin$^c$-structure and assume that $d(X,\mathfrak{s}) = 0$. Since $d(X,\mathfrak{s}) = 2d-b_+(X)-1 = 0$, this implies that $b_+(X)$ is odd.

In the case of free actions with $d(X,\mathfrak{s}) = 0$, we have a method of $K$-orienting $\mathbb{R} \oplus H^+(X/H)$ for every subgroup $H$. First, from the proof of Proposition \ref{prop:free} we see that $\mathbb{R} \oplus H^+(X)$ is isomorphic to several copies of the real regular representation. In fact, since $d(X , \mathfrak{s}) = 0$, there are $2d_0$ copies, where $d_0 = d/|G|$. Therefore, $\mathbb{R} \oplus H^+(X)$ can be (non-canonically) equipped with a $G$-invariant complex structure $I$ for which $\mathbb{R} \oplus H^+(X)$ is isomorphic as a complex representation to $d_0$ copies of the complex regular representation. Fix a choice of such a complex structure. We use this to define a $K$-orientation on $\mathbb{R}\oplus H^+(X)$. Then for any $H \subseteq G$, we have an isomorphism $\mathbb{R} \oplus H^+(X/H) \cong (\mathbb{R} \oplus H^+(X))^H$. Moreover, $(\mathbb{R} \oplus H^+(X))^H$ is a complex subspace of $\mathbb{R} \oplus H^+(X)$ and hence it has an induced complex structure and a $K$-orientation.

\begin{theorem}
Let $X$ be a compact, oriented, smooth $4$-manifold with $b_1(X) = 0$. Let $G$ be a finite group which acts smoothly and freely on $X$. Let $\mathfrak{s}$ be a spin$^c$-structure whose isomorphism class is fixed by $G$ and suppose that $d(X,\mathfrak{s}) \ge 0$. Assume that $b_+(X)^G > 0$ and if $b_+(X)^G = 1$ then fix a chamber $\phi$. 

Fix a choice of $G$-invariant complex structure on $\mathbb{R} \oplus H^+(X)$ for which $\mathbb{R} \oplus H^+(X)$ is isomorphic to a sum of copies of the complex regular represention and use this to $K$-orient $\mathbb{R} \oplus H^+(X/H)$ for every $H \subseteq G$, as described above. Then we have
\[
\sum_{g \in G} \sum_{s'} SW(X/\langle g \rangle , \mathfrak{s}' ) = 0 \; ({\rm mod} \; |G|)
\]
where the second sum is over spin$^c$-structures on $X/\langle g \rangle$ whose pullback to $X$ is isomorphic to $\mathfrak{s}$. Equivalently, we have
\[
\sum_{(H) \text{ cyclic}} \phi(|H|) |G/N_G H| \sum_{\mathfrak{s}' | q_H^*(\mathfrak{s}') \cong \mathfrak{s}} SW(X/H , \mathfrak{s}') = 0 \; ({\rm mod} \; |G|)
\]
where the first sum is over conjugacy classes of cyclic subgroups of $G$ and the second sum is over spin$^c$-structures on $X/H$ whose pullback to $X$ is isomorphic to $\mathfrak{s}$. If $b_+(X/H) = 1$ then $SW(X/H , \mathfrak{s}')$ is defined using the chamber $\phi$.
\end{theorem}
\begin{proof}
Let $f$ be the $G$-equivariant Bauer--Furuta monopole map of $(X,\mathfrak{s})$. Let $M = SW^{K}_{G , f}(1) \in R(G)$. The restriction of $M$ to $R(1) \cong \mathbb{Z}$, which is the rank of $M$, is just the ordinary ($K$-theoretic) Seiberg--Witten invariant of $(X,\mathfrak{s})$. Since $d(X,\mathfrak{s}) = 0$, it follows that this equals the usual Seiberg--Witten invariant $SW(X,\mathfrak{s})$ (see \cite[\textsection 6]{bk}). Thus $\chi( M , 1) = SW(X , \mathfrak{s})$.

Let $g \in G$ have order $n$. Set $H = \langle g \rangle$. Since $H$ is cyclic, we can choose a splitting $H_{\mathfrak{s}} \cong S^1 \times H$. Then as in Section \ref{sec:lock}, let $s_j$ denote the splitting given by $s_j(g) = (\omega^{-j} , g)$, $\omega = e^{2\pi i /n}$. Let $\mathbb{C}_j$ be the $1$-dimensional representation of $H$ where $g$ acts as multiplication by $\omega^j$. Use the splitting $s_0$ to regard $D$ as a representation of $H$. Since the action is free, we have 
\begin{equation}\label{equ:dsum}
D \cong \bigoplus_{j=0}^{n-1} \mathbb{C}_j^{d/n}.
\end{equation}
Recall that we have fixed a complex structure on $\mathbb{R} \oplus H^+(X)$ for which $\mathbb{R} \oplus H^+(X)$ is isomorphic to a sum of copies of the regular representation of $G$. Restricting to $H$, we therefore have that $\mathbb{R} \oplus H^+(X)$ as a representation of $H$ is isomorphic to a sum of copies of the regular representation. Hence
\begin{equation}\label{equ:hplus}
H^+/(H^+)^H \cong \bigoplus_{j=1}^{n-1} \mathbb{C}_j^{d/n}.
\end{equation}

Consider the reduced monopole map $f^{s_j}$. Its $H$-equivariant $K$-theoretic Seiberg--Witten invariant takes the form of a map of $R(H)$-modules:
\[
SW^K_{H , f^{s_j} } : R(H)[\xi , \xi^{-1}] \to R(H).
\]
We have $R(H) \cong \mathbb{Z}[t]/(t^n-1)$, where $t = [\mathbb{C}_1]$. By Proposition \ref{prop:freeinv}, it follows that $SW^K_{H , f^{s_j}}(1) = SW(X/H , \mathfrak{s}_{s_j})$. Under the splitting $s_j$, $s_j^*(\xi) = t^{-j}$, hence $s_j^*( t^j \xi ) = 1$. This implies that $SW^K_{H,f^{s_j}}( t^j \xi ) = SW(X/H , \mathfrak{s}_{s_j})$, hence $SW^K_{H,f^{s_j}}( \xi ) = t^{-j} SW(X/H , \mathfrak{s}_{s_j})$ and more generally $SW^K_{H,f^{s_j}}( \xi^m ) = t^{-mj} SW(X/H , \mathfrak{s}_{s_j})$ for any $m$. Therefore,
\[
\chi( SW^K_{H , f^{s_j}}( \xi^m ) , g ) = \omega^{-mj} SW(X/H , \mathfrak{s}_{s_j}).
\]

Theorem \ref{thm:lock} gives
\begin{equation}\label{equ:chimg}
\chi(M , g) = \chi( e^K_H( H^+/ (H^+)^H ) , g) \sum_{j=0}^{n-1} \chi( SW^K_{H , f^{s_j}}( e^K_{H_\mathfrak{s}}( D/D_j )^{-1} ) , g ).
\end{equation}
From Equation (\ref{equ:dsum}) we have
\[
e^K_{H_{\mathfrak{s}}}( D/D_j) = \prod_{i \neq j} ( 1 - \omega^{-i} \xi^{-1})^{d/n}
\]
and hence
\begin{align*}
\chi( SW^K_{H , f^{s_j}}( e^K_{H_{\mathfrak{s}}}(D/D_j)^{-1} , g ) &= \left( \prod_{i \neq j} ( 1 - \omega^{j-i})^{-d/n} \right) SW(X/H , \mathfrak{s}_{s_j}) \\
&= \left( \prod_{i=1}^{n-1} (1- \omega^j) \right)^{-d/n} SW(X/H , \mathfrak{s}_{s_j}).
\end{align*}

From Equation (\ref{equ:hplus}) we have
\[
e^K_H(  H^+/ (H^+)^H ) = \prod_{i=1}^{n-1} (1 - t^{-i})^{d/n},
\]
hence
\[
\chi( e^K_H( H^+/(H^+)^H) , g ) = \left( \prod_{i=1}^{n-1} (1- \omega^{-i} ) \right)^{d/n}.
\]
Substituting into Equation (\ref{equ:chimg}) gives
\[
\chi(M,g) = \sum_{j=0}^{n-1} SW(X/H , \mathfrak{s}_{s_j}).
\]

Now since $\sum_{g \in G} \chi(M , g) = 0 \; ({\rm mod} \; |G|)$, we have shown that
\[
\sum_{g \in G} \sum_{s'} SW(X/\langle g \rangle , \mathfrak{s}' ) = 0 \; ({\rm mod} \; |G|).
\]
\end{proof}

\subsection{A constraint on smooth group actions}\label{sec:constraint}

\begin{theorem}\label{thm:fang}
Let $X$ be a compact, oriented, smooth $4$-manifold with $b_1(X) = 0$. Let $G$ be a finite group acting on $X$ by orientation preserving diffeomorphisms. Let $\mathfrak{s}$ be a $G$-invariant spin$^c$-structure. Assume that $b_+(X)^G > 0$ and that $d(X,\mathfrak{s}) = 0$, hence $b_+(X)$ is odd. Suppose also that $\mathbb{R} \oplus H^+(X)$ can be equipped with a $G$-invariant complex structure. Suppose that $SW(X,\mathfrak{s}) \neq 0 \; ({\rm mod} \; |G|)$. Then for some non-trivial cyclic subgroup $\{1 \} \neq H \subseteq G$ and some splitting $s : H \to G_{\mathfrak{s}}$, we have $2 \, dim_{\mathbb{C}} ( D^{sH} ) > dim_{\mathbb{R}}( H^+(X)^H)$.
\end{theorem}
\begin{proof}
Let $f$ be the $G$-equivariant Bauer--Furuta monopole map of $(X,\mathfrak{s})$. Let $M = SW^{K}_{G , f}(1) \in R(G)$. Then $\chi(M,1) = SW(X,\mathfrak{s})$. The complex structure on $\mathbb{R} \oplus H^+(X)$ yields a complex structure on $\mathbb{R} \oplus H^+(X)^H$ for every $H \subseteq G$ and hence a $K$-orientation.

Let $g \in G$ have order $n$. Theorem \ref{thm:lock} gives
\[
\chi(M,g) = \chi( e^K_H(H^+/(H^+)^H) , g) \sum_{s} \chi( SW^{K,\phi}_{H,f^{s}}( e^K_{G_{\mathfrak{s}}}(D/D^{sH})^{-1} \theta , g)
\]
where the sum is over splittings $s : H \to H_{\mathfrak{s}}$. Now since $s(H)$ acts trivially on the domain and codomain of $f^{s}$, there are no obstructions to achieving equivariant transversality. The expected dimension of the moduli space of $sH$-invariant solutions to the Seiberg--Witten equations is $2 \, dim_{\mathbb{C}}( D^{sH} ) - dim_{\mathbb{R}}( H^+(X)^H )  - 1$. If this is negative then $SW^K_{H,f^{s}} = 0$.

Now we re-write the congruence $\sum_{g \in G} \chi( M , g ) = 0 \; ({\rm mod} \; |G| )$ as
\[
SW(X,\mathfrak{s}) = -\sum_{g \neq 1} \chi(M,g) \; ({\rm mod} \; |G|).
\]
Hence if $SW(X,\mathfrak{s}) \neq 0 \; ({\rm mod} \; |G|)$ then $\chi(M,g) \neq 0$ for some $g \in G \setminus \{1\}$ and thus $2 \, dim_{\mathbb{C}} ( D^{sH} ) > dim_{\mathbb{R}}( H^+(X)^H)$ for some splitting $s$ of $H$.
\end{proof}

\subsection{Divisibility conditions}\label{sec:divc}

Let $f : S^{V,U} \to S^{V',U'}$ be an ordinary monopole map over a point, $V = \mathbb{C}^a, V' = \mathbb{C}^{a'}$, $U = \mathbb{R}^b, U' = \mathbb{R}^{b'}$, $d = a-a'$, $b_+ = b'-b$. Assume that $b_+ \ge 1$ and that $2d-b_+ - 1 = 2m$ is even and non-negative. Choose an orientation on $H^+$ and choose a chamber if $b_+ = 1$. Then the abstract Seiberg--Witten invariant of $f$ is defined and given by $SW_f(x^m) \in H^0(pt ; \mathbb{Z}) = \mathbb{Z}$. Let us denote it by $SW(f)$ for simplicity. If $f$ is the Bauer--Furuta monopole map of $(X,\mathfrak{s})$, then $SW(f) = SW(X,\mathfrak{s})$ is the usual Seiberg--Witten invariant.

\begin{theorem}\label{thm:div}
Let $f : S^{V,U} \to S^{V',U'}$ be an ordinary monopole map over a point and let $p$ be a prime. If $SW(f) \neq 0 \; ({\rm mod} \; p)$, then $(b_+ -1)/2$ is divisible by $p^{e+1}$ whenever $p^e \le \left \lfloor \frac{m}{p-1} \right\rfloor$, where $2d-b_+-1 = 2m$. In particular, if $SW(f) \neq 0 \; ({\rm mod} \; p)$ and $b_+ \neq 1 \; ({\rm mod} \; 2p)$ then $m \le p-2$.
\end{theorem}
\begin{proof}
We give the proof in the case $p$ is odd. The case $p=2$ is similar. We will make use of the Steenrod powers $P^i$. Using the Borel model, the Steenrod powers can be defined on $S^1$-equivariant cohomology with $\mathbb{Z}_p$-coefficients. Let $W$ be an $S^1$-equivariant vector bundle on $B$ with Thom class $\tau_B$. Then
\begin{equation}\label{equ:pjthom}
P^j(\tau_W) = q_j(W) \tau_W
\end{equation}
for some characteristic class $q_j(W) \in H^{2(p-1)j}_{S^1}(B ; \mathbb{Z}_p)$. Setting $q_t = q_0 + tq_1 + \cdots $, one finds \cite[Chapter 19]{ms} that if $W$ is complex with Chern roots $\{ a_i \}$ then:
\[
q_t(W) = \prod_i ( 1 + ta_i^{p-1}).
\]
Setting $P_t = P^0 + tP^1 + \cdots $, we have $P_t(\tau_W) = q_t(W) \tau_W$.

Considering $\mathbb{C}$ with the standard $S^1$-action and $\mathbb{R}$ with the trivial $S^1$-action as $S^1$-equivariant vector bundles over $B = pt$, we have
\[
q_t(\mathbb{C}) = 1+tx^{p-1}, \quad q_t(\mathbb{R}) = 1.
\]

Since $deg(x) = 2$, we have $P^1(x) = x^p$ and $P^j(x) = 0$ for $j > 1$. Hence $P_t(x) = (1+tx^{p-1})x$. Consider $f^* : H^*_{S^1}( S^{V',U'} ; S^{U} ; \mathbb{Z}_p) \to H^*_{S^1}( S^{V,U} ; S^{U} ; \mathbb{Z}_p)$. Suspending if necessary, we can assume $U$ is even-dimensional. Then $f^*( \tau^\phi_{V',U'} ) = \delta \tau_U \eta^\phi$, where $\eta^\phi = SW(f) x^{a-1 - m}$. Applying $P_t$ to both sides of
\[
SW(f) \delta \tau_U x^{a-1-m} = f^*( \tau^\phi_{V',U'})
\]
gives
\begin{align*}
(1+tx^{p-1})^{a-1-m} SW(f) \delta \tau_U x^{a-1-m} &= (1+tx^{p-1})^{a'} f^*( \tau^\phi_{V',U'}) \\
&= (1+tx^{p-1})^{a'} SW(f) \delta \tau_U x^{a-1-m}.
\end{align*}
(To prove this we used that Equation (\ref{equ:pjthom}) also holds for the refined Thom class $\tau^\phi_{V',U'}$. This follows by a straightforward extension of the usual proof. We also used that $P_t$ commutes with the coboundary operator $\delta$).

If $SW(f) \neq 0 \; ({\rm mod} \; p)$, then the above equation reduces to
\[
(1+tx^{p-1})^{a-1-m} x^{a-1-m} = (1+tx^{p-1})^{a'} x^{a-1-m}.
\]
This is an equality in $H^*( \mathbb{P}(V) ; \mathbb{Z}_p)[[t]] \cong \mathbb{Z}_p[x][[t]]/( x^a )$. Multiplying both sides by $(1+tx^{p-1})^{-a'} = 1 -a't x^{p-1} + \cdots $ gives
\[
(1+tx^{p-1})^{d-1-m} x^{a-1-m} = x^{a-1-m} \text{ in } \mathbb{Z}_p[x][[t]]/(x^a).
\]
Expanding on the left, we get
\[
\binom{d-1-m}{j} = 0 \; ({\rm mod} \; p) \text{ whenever } 1 \le j \le m/(p-1).
\]

Let $p^e \le \lfloor m/(p-1) \rfloor$. Setting $j = p^u$ with $u = 0,1, \dots , e$ we get
\[
\binom{d-1-m}{p^u} = 0 \; ({\rm mod} \; p) \text{ for } 0 \le u \le e.
\]
This implies that $d-1-m$ is divisible by $p^{e+1}$. Noting that $2m = 2d-b_+ - 1$, we have $d-m-1 = (b_+ - 1)/2$, and hence $(b_+-1)/2$ is divisible by $p^{e+1}$.
\end{proof}

\begin{remark}
The conclusion that $m \le p-2$ when $SW(f) \neq 0 \; ({\rm mod} \; p)$ and $b_+ \neq 1 \; ({\rm mod} \; 2p)$ was also shown in \cite[Corollary 1.5]{kkny}.
\end{remark}

\subsection{$\mathbb{Z}_p$-actions}\label{sec:zpact}

Let $X$ be a compact, oriented, smooth $4$-manifold with $b_1(X) = 0$. Let $G = \mathbb{Z}_p = \langle g \rangle$ act on $X$ by orientation preserving diffeomorphism, where $p$ is a prime. For convenience we will work with $\mathbb{Z}_p$-coefficients. For odd $p$ we have $H^*_{\mathbb{Z}_p}(pt ; \mathbb{Z}_p) \cong \mathbb{Z}_p[u,v]/(u^2)$ where $deg(u) = 1$, $deg(v) = 2$. For $p=2$ we have $H^*_{\mathbb{Z}_2}(pt ; \mathbb{Z}_2) \cong \mathbb{Z}_2[u]$ where $deg(u) = 2$. In this case we set $v = u^2$ and sometimes write $u = v^{1/2}$.

Suppose that $H^+(X)^G \neq 0$ and let $\phi$ be a chamber. Let $\mathfrak{s}$ be a $G$-invariant spin$^c$-structure. Fix a trivialisation $G_{\mathfrak{s}} \cong S^1 \times G$. This amounts to choosing a lift of $\mathbb{Z}_p$ to the spinor bundles associated to $\mathfrak{s}$. Then the equivariant Seiberg--Witten invariants $SW_{\mathbb{Z}_p , X , \mathfrak{s}}^\phi : H^*_{\mathbb{Z}_p}(pt ; \mathbb{Z}_p)[x] \to H^{* - d(X,\mathfrak{s})}_{\mathbb{Z}_p}( pt ; \mathbb{Z}_p)$ are defined. Let $\mathbb{C}_j$ be the $1$-dimensional complex representation of $G$ where $g$ acts as multiplication by $\omega^j$, $\omega = e^{2\pi i/p}$. Then $D = \bigoplus_{j=0}^{p-1} \mathbb{C}_j^{d_j}$ for some $d_0, \dots , d_{p-1}$. The weights $d_j$ can be computed using the $G$-spin theorem. Let $b_0$ be the dimension of $H^+(X)^G$. To each $j \in \mathbb{Z}_p$ we obtain a splitting $s_j : \mathbb{Z}_p \to S^1 \times \mathbb{Z}_p$ given by $s_j(g) = (\omega^{-j} , g)$. Let $f$ denote the $G$-monopole map associated to $(X , \mathfrak{s})$ and $f^{s_j}$ the reduced monopole map. The expected dimension of the moduli space for $f^{s_j}$ is $2d_j - b_0 - 1$. Let $\delta_j = d_j - (b_0+1)/2$. If this is integral and non-negative then we obtain a reduced Seiberg--Witten invariant $\overline{SW}^{s_j,\phi}_{G,X,\mathfrak{s}} = SW_{f^{s_j}}^\phi( x^{\delta_j} ) \in \mathbb{Z}$. If $\delta_j$ is non-integral or negative, then we set $\overline{SW}^{s_j,\phi}_{G,X,\mathfrak{s}} = 0$. Recall that each splitting $s_j$ defines an isomorphism $\psi_{s_j} : H^*_{G_{\mathfrak{s}}}(pt ; \mathbb{Z}_p) \to H^*_{\mathbb{Z}_p}(pt ; \mathbb{Z})[x]$. If we use $s_0$ to identify $H^*_{G_{\mathfrak{s}}}(pt ; \mathbb{Z}_p)$ with $H^*_{\mathbb{Z}_p}(pt ; \mathbb{Z})[x]$ then $\psi_{s_j}$ is given by $\psi_{s_j}(x) = x-jv$, $\psi_{s_j}(v) = v$. Since $s_j(G)$ acts trivially on the domain and codomain of $f_{s_j}$, we have that $SW^\phi_{G , f^{s_j}}( \psi_{s_j}^{-1}( x^m) ) = \overline{SW}^{s_j,\phi}_{G,X,\mathfrak{s}}$ if $m=\delta_j$ and is zero otherwise. Then since $\psi_{s_j}^{-1}(x) = x+jv$, it follows that for any $\theta \in \mathbb{Z}_p[x,v]$, $SW_{G , f^{s_j}}^\phi( \theta ) = c_j \overline{SW}^{s_j,\phi}_{G,X,\mathfrak{s}}$, where $c_j$ is the coefficient of $(x+jv)^{\delta_j}$ when $\theta$ is written as a polynomial in $x+jv$ with coefficients in $\mathbb{Z}_p[v]$.

Let $k = \mathbb{Z}_p(v)$ be the ring of rational functions in $v$ with coefficients in $\mathbb{Z}_p$. Let $k(x)$ be the ring of rational functions in $x$ with coefficients in $k$. For any $a \in k$ we have a natural homomorphism $k(x) \to k((x-a))$ from $k(x)$ into the ring $k((x-a))$ of formal Laurent series in $x-a$ with coefficients in $k$ and with finite polar part. If $f \in k(x)$ then the image of $f$ in $k((x-a))$ can be uniquely written as $f = \sum_{j = -n}^{\infty} c_j (x-a)^j$, $c_j \in k$. We refer to $c_j$ as the coefficient of $(x-a)^j$ in the Laurent expansion of $f$ at $a$. For any $j \in \mathbb{Z}_p$ and integers $n , n_0, n_1, \dots , n_{p-1}$, define $c_j( n ; n_0 , \dots , n_{p-1})$ to be the coefficient of $(x+jv)^n$ in the Laurent expansion of $\prod_{i=0}^{p-1} (x+iv)^{n_i}$. It follows that
\[
c_j( n ; n_0 , \dots , n_{p-1}) = \left( \sum_{k_i} \prod_{i | i\neq j} \binom{n_i}{k_i} (i-j)^{n_i - k_i} \right) v^{\sum_i n_i - n}
\]
where the sum is over non-negative integers $k_0, \dots , \hat{k}_j , \dots , k_{p-1}$ such that $k_0 + \cdots + \hat{k}_j + \cdots + k_{p-1} = n-n_j$. We will extend the definition of $c_j( n ; n_0 , \dots , n_{p-1})$ to non-integer values of $n$ by setting $c_j( n ; n_0 , \dots , n_{p-1}) = 0$ if $n$ is non-integral.

\begin{theorem}\label{thm:zp}
For any non-negative integers $m_0, \dots , m_{p-1}$, we have the following equality in $H^*(\mathbb{Z}_p ; \mathbb{Z}_p)$:
\begin{align*}
& SW^\phi_{\mathbb{Z}_p , X , \mathfrak{s}}( x^{m_0} (x+v)^{m_1} \cdots (x+(p-1)v)^{m_{p-1}}) \\
& \quad \quad = e_{\mathbb{Z}_p}( H^+/(H^+)^g ) \sum_{j=0}^{p-1} c_j\left(  -\dfrac{(b_0+1)}{2} ; m_0-d_0 , \dots , m_{p-1} - d_{p-1} \right) \overline{SW}^{s_j,\phi}_{G,X,\mathfrak{s}}.
\end{align*}
\end{theorem}
\begin{proof}
By Theorem \ref{thm:loc}, we have
\[
SW^\phi_{G,f}(\theta) = e_G(H^+/(H^+)^g) \sum_{j=0}^{p-1} SW^\phi_{G,f^{s_j}}(    e_{G_{\mathfrak{s}}}(D/D_j)^{-1} \theta ).
\]
If $\theta = \prod_{i=0}^{p-1} (x+iv)^{m_i}$, then
\[
e_{G_{\mathfrak{s}}}(D/D_j)^{-1} \theta = \prod_{i=0}^{p-1}(x+iv)^{m_i} \prod_{i  |  i \neq j} (x+jv)^{-d_i}.
\]
Hence $SW^\phi_{G,f^{s_j}}(    e_{G_{\mathfrak{s}}}(D/D_j)^{-1} \theta )$ equals $SW^{G,s_j,\phi}_{X,\mathfrak{s}}$ times the coefficient of $(x+jv)^{\delta_j}$ in $\prod_{i=0}^{p-1}(x+iv)^{m_i} \prod_{i  |  i \neq j} (x+jv)^{-d_i}$. But $\delta_j = d_j - (b_0+1)/2$, so this is $\alpha_j$ times the coefficient of $(x+jv)^{-(b_0+1)/2}$ in $\prod_{i=0}^{p-1}(x+iv)^{m_i - d_i}$, which is $c_j( -(b_0+1)/2 ; m_0-d_0 , \dots , m_{p-1} - d_{p-1})$.
\end{proof}

For instance, when $p=2$, we get:
\[
SW_{\mathbb{Z}_2 , X , \mathfrak{s}}^\phi( x^{m_0}(x+v)^{m_1} ) = \left( \binom{m_1-d_1}{\delta_0-m_0} \overline{SW}^{s_0,\phi}_{G,X,\mathfrak{s}} + \binom{m_0-d_0}{\delta_1-m_1} \overline{SW}^{s_1,\phi}_{G,X,\mathfrak{s}} \right) v^{m_0+m_1 - \delta}
\]
where we set $\binom{a}{b} = 0$ if $b$ is non-integral or negative.

\section{K\"ahler actions}\label{sec:kahler}

Suppose $(X , I )$ is a compact complex surface with $b_1(X) = 0$ and suppose that $G$ is a finite group which acts on $X$ by biholomorphism. Since $b_1(X) = 0$, $X$ admits a K\"ahler metric. Furthermore, the set of K\"ahler metrics is convex so by averaging we can find a $G$-invariant K\"ahler metric $g$ with associated K\"ahler form $\omega$. The Hodge decomposition yields an equality $H^+(X) = \mathbb{R}[\omega] \oplus Re( H^{2,0}(X) )$ and hence $H^+(X) \cong \mathbb{R} \oplus H^0(X,K_X)$ where $K_X$ denotes the canonical bundle of $X$. Since $H^0(X,K_X)$ is a complex vector space, this provides a distinguished orientation on $H^+(X)$ given by $\{ \omega , e_1 , Ie_1 , \dots , e_n , Ie_n\}$, where $e_1, \dots , e_n$ is a complex basis for $H^0(X,K_X)$. We also have a distinguished chamber given by setting $\phi = [\omega]$. We call this the {\em K\"ahler chamber} and we call $-\phi$ the {\em anti-K\"ahler chamber}. Since $H^+(X)^G \cong \mathbb{R}[\omega] \oplus H^0(X,K_X)^G$, we see that $dim( H^+(X)^G ) \ge 1$ and equals $1$ only if $H^0(X , K_X)^G = 0$.

The complex structure on $X$ defines a spin$^c$-structure $\mathfrak{s}_{can}$, the {\em canonical spin$^c$-structure} with spinor bundles $S^{\pm} = \wedge^{0,ev/odd} T^*X$. Any other spin$^c$-structure is of the form $\mathfrak{s}_L = L \otimes \mathfrak{s}_{can}$ for some complex line bundle $L$ and we have $c(\mathfrak{s}_L) = 2c_1(L) - c_1(K_X)$. The spinor bundles for $\mathfrak{s}_L$ are given by $S^{\pm}_L = L \otimes S^{\pm}$. The charge conjugate of $\mathfrak{s}_L$ is $\mathfrak{s}_{K_X^* L}$. A spin$^c$-structure $\mathfrak{s}_L$ is preserved by $G$ if and only $G$ preserves the isomorphism class of $L$. Furthermore, since $G$ has a canonical lift to $S^{\pm}$ we see that lifts of $G$ to $S^{\pm}_L$ correspond to lifts of $G$ to $L$.

Let $L$ be a complex line bundle whose isomorphism class is preserved by $G$. To compute the equivariant Seiberg--Witten invariants of $(X , \mathfrak{s}_L)$, we first consider the Seiberg--Witten equations with respect to a $2$-form perturbation of the form $\eta = i \lambda \omega$, where $\lambda$ is a sufficiently large positive real number. Note that it suffices to consider only the K\"ahler chamber since charge conjugation exchanges the K\"ahler and anti-K\"ahler chambers. We follow the approach of \cite[Chapter 12]{sal}. Let $\mathcal{M} = \mathcal{M}(X , L , g , \eta)$ denote the moduli space of solutions to the $\eta$-perturbed Seiberg--Witten equations on $X$ with respect to the metric $g$ and spin$^c$-structure $\mathfrak{s}_L$. By \cite[Proposition 12.23]{sal}, we have a natural bijection between $\mathcal{M}$ and the space of effective divisors on $X$ representing $c_1(L)$. If the image of $c_1(L)$ in $H^2(X ; \mathbb{R})$ is not of type $(1,1)$, then there are no divisors representing $c_1(L)$ and hence $\mathcal{M}$ is empty. On the other hand if $c_1(L)$ is of type $(1,1)$, then $L$ admits a holomorphic structure. Since $b_1(X) = 0$, the holomorphic structure is unique up to isomorphism and so we can regard $L$ as a fixed holomorphic line bundle. Let $V^i = H^i(X , L)$ denote the cohomology groups of $L$ and let $h^i(L)$ denote the dimension of $V^i$. The $V^i$ are representations of $G_{\mathfrak{s}_L}$ where the $S^1$-subgroup acts by scalar multiplication. Effective divisors representing $c_1(L)$ are in bijection with non-zero holomorphic sections of $L$ up to scale, so we have an identification $\mathcal{M} \cong \mathbb{P}(V^0)$. Under this identification the action of $G$ on $\mathcal{M}$ corresponds to the action on $\mathbb{P}(V^0)$ induced by the action of $G_{\mathfrak{s}_L}$ on $V^0$.

Although the moduli space $\mathcal{M}$ is a smooth manifold, the perturbation $\eta = i \lambda \omega$ is typically not a regular perturbation. That is, the moduli space $\mathcal{M}$ is typically not cut out transversally and its dimension, $2( h^0(L) - 1)$, is typically larger than the expected dimension which is $2( h^0(L) - 1) - 2( h^1(L) - h^2(L) + h^{2,0}(X))$. To compute the equivariant Seiberg--Witten equations in this setting we will make use of the technique of obstruction bundles \cite{fm}, \cite[\textsection 12.9]{sal}. The failure of $\mathcal{M}$ to be cut out transversally is measured by a bundle $Obs \to \mathcal{M}$ called the {\em obstruction bundle}. The fibres of $Obs$ are the cokernels of the linearisation of the Seiberg--Witten equations. As shown in \cite[\textsection 12.9]{sal}, we have an exact sequence of bundles on $\mathcal{M}$:
\begin{equation}\label{equ:obs}
0 \to \widetilde{V}^1 \to Obs \to H^2( X , \mathcal{O}) \to \widetilde{V}^2  \to 0,
\end{equation}
where $\widetilde{V}^i$ denotes the vector bundle over $\mathbb{P}(V^0) = S(V^0)/S^1$ given by $\widetilde{V}^i = V^i \times_{S^1} S(V^0)$ and $H^2(X , \mathcal{O})$ is to be regarded as a trivial vector bundle on $\mathbb{P}(V^0)$. In the presence of a $G$-action, the obstruction bundle $Obs$ is a $G$-equivariant vector bundle and the above sequence is an exact sequence of $G$-equivariant vector bundles on $\mathcal{M}$.

\begin{lemma}
Assume that $G_{\mathfrak{s}}$ is split. Then for any $\theta \in H^*_{G_{\mathfrak{s}_{\scalebox{.7}{$\scriptscriptstyle L$}}}} \! (pt ; \mathbb{Z})$, we have
\[
SW_G^\omega( \theta ) = ( \pi_{\mathbb{P}(V^0)} )_*( e_G( Obs ) \theta )
\]
where $\theta$ is regarded as an element of $H^*_G( \mathbb{P}(V^0) ; \mathbb{Z} )$ by pulling it back to $S(V^0)$ and using $H^*_{G_{\mathfrak{s}_{\scalebox{.7}{$\scriptscriptstyle L$}}
}} \! ( S(V^0) ; \mathbb{Z}) \cong H^*_G( \mathbb{P}(V^0) ; \mathbb{Z})$.
\end{lemma}
\begin{proof}
In the non-equivariant setting, this follows from the technique of obstruction bundles. We need to extend the result to the equivariant setting. If $Obs$ admits a $G$-invariant section whhich is tranverse to the zero section, then the usual argument can be carried out equivariantly. However in general an equivariant vector bundle need not admit invariant sections which are transverse to the zero section. To get around this problem we will work with families instead of equivariantly.

For any compact smooth manifold $B$ and any principal $G$-bundle $P \to B$, we have an associated family $E = P \times_G X$. Since the K\"ahler structure of $X$ is preserved by $G$, the family $E$ is a K\"ahler family in the sense that the fibres are equipped with a smoothly varying K\"ahler structure. We will use this structure to evaluate the families Seiberg--Witten invariants $SW_{G,E, \mathfrak{s}_L}^\omega(\theta)$. The families moduli space is $\mathcal{M}_E = P \times_G \mathcal{M}$. It is not cut out transversally, but the obstruction bundle technique can be applied. The obstruction bundle $Obs_E$ for this family is just the associated vector bundle $Obs_E = P \times_G Obs$ and therefore we have
\[
SW_{G,E,\mathfrak{s}_L}^{\omega}(\theta) = (\pi_*)( e( Obs_E) \theta )
\]
where $\pi$ is the projection map $\pi : \mathcal{M}_E \to B$. It follows immediately that
\[
SW_{G,E,\mathfrak{s}_L}^{\omega}(\theta) = (f_P)^*(   ( \pi_{\mathbb{P}(V^0)} )_*( e_G( Obs ) \theta ) )
\]
where $f_P : B \to BG$ is the classifying map for $P$. Then by Theorem \ref{thm:eqfam2}, it follows that $SW_G^\omega( \theta ) = ( \pi_{\mathbb{P}(V^0)} )_*( e_G( Obs ) \theta )$.
\end{proof}

We turn to the computation of $e_G(Obs)$. Since $Obs$ is a complex vector bundle of rank $r = h^1(L) - h^2(L) + h^{2,0}(X)$, we have $e_G(Obs) = c_{r,G}(Obs)$, where $c_{r,G}$ denotes the $r$-th equivariant Chern class. Let $c_G = 1 + c_{1,G} + \cdots$ denote the total equivariant Chern class and $s_G = 1 + s_{1,G} + \cdots$ the total equivariant Segre class. From (\ref{equ:obs}) we have
\[
c_G(Obs) = c_G(\widetilde{V}^1) s_G( \widetilde{V}^2 ) c_G( H^2(X , \mathcal{O}))
\]
and hence
\[
e_G(Obs) = \sum_{\substack{i+j+k=r \\ i,j,k \ge 0}} c_{i,G}( \widetilde{V}^1) s_{j,G}( \widetilde{V}^2) c_{k,G}( H^2(X , \mathcal{O})).
\]

Assume now that $G_{\mathfrak{s}_{\scalebox{.7}{$\scriptscriptstyle L$}}} \!$ is a split extension. Fix a splitting $G_{\mathfrak{s}_{\scalebox{.7}{$\scriptscriptstyle L$}}} \! \cong S^1 \times G$. A choice of splitting amounts to a choice of lift of the $G$-action to $L$. This makes $V^i$ into representations of $G$ and then $\widetilde{V}^i \cong V^i \otimes \mathcal{O}_{V^0}(1)$. The equivariant Chern and Segre classes of $\widetilde{V}^i$ can now be expressed in terms of the Chern and Segre classes of $V^i$ and $x = c_{1,G}( \mathcal{O}_{V^0}(1))$ using the following identities for a vector bundle $E$ of rank $r$ and a line bundle $N$:
\[
c_{j}(E \otimes N) = \sum_{l=0}^j c_{l}(E) c_{1}(N)^{j-l} \binom{r-l}{j-l}, \quad s_{j}(E \otimes N) = \sum_{l=0}^{j} s_{l}(E) c_{1}(N)^{j-l} \binom{-r-l}{j-l}.
\]
The same result also applies more generally when $E$ is a virtual vector bundle. We can simplify the expression for $e_G(Obs)$ by writing it in terms of the virtual bundle $\widetilde{V}^1 - \widetilde{V}^2$, namely
\[
e_G(Obs) = \sum_{\substack{i+j=r \\ i,j \ge 0}} c_{i,G}( \widetilde{V}^1 - \widetilde{V}^2) c_{j,G}( H^2(X , \mathcal{O})).
\]
We also have $H^*_{G_{\mathfrak{s}_{\scalebox{.7}{$\scriptscriptstyle L$}}}} \! (pt ; \mathbb{Z}) \cong H^*_G(pt ; \mathbb{Z})[x]$ and so it suffices to compute $SW_G^\omega( x^m )$ for each $m \ge 0$. We also have that $(\pi_{\mathbb{P}(V^0)})_*( x^j ) = s_{j-(d-1) , G}(V^0)$ where $d = h^0(L)-h^1(L) + h^2(L)$. Putting it all together, we have
\begin{align*}
& SW_{G}^\omega( x^m ) = (\pi_{\mathbb{P}(V^0)})_* \left( \sum_{\substack{i+j=r \\ i,j \ge 0}} c_{i,G}( \widetilde{V}^1 - \widetilde{V}^2) c_{j,G}( H^2(X , \mathcal{O})) x^m \right) \\
& = \sum_{\substack{i+j=r \\ i,j \ge 0}} \sum_{l=0}^{i} \binom{h^1(L)-h^2(L)-l}{i-l}  c_{l,G}(V^1-V^2) c_{j,G}( H^2(X , \mathcal{O})) (\pi_{\mathbb{P}(V^0)})_*( x^{m+i-l} ) \\
& = \sum_{\substack{i+j=r \\ i,j \ge 0}} \sum_{l=0}^{i} \binom{h^1(L)-h^2(L)-l}{i-l} s_{i-l+m - (h^0(L)-1),G}(V^0) c_{l,G}(V^1-V^2) c_{j,G}( H^2(X , \mathcal{O})).
\end{align*}

\begin{theorem}\label{thm:swk}
Let $X$ be a compact complex surface with $b_1(X) = 0$. Let $G$ be a finite group that acts on $X$ by biholomorphisms. Let $L$ be a $G$-equivariant line bundle. If $L$ is not holomorphic, or if $L$ is holomorphic and $h^0(L) = 0$, then $SW^\omega_{G , X , \mathfrak{s}_L} = 0$. If $L$ is holomorphic and $h^0(L) > 0$, then
\begin{align*}
& SW_{G , X , \mathfrak{s}_L}^\omega( x^m ) \\
& = \sum_{\substack{i+j=r \\ i,j \ge 0}} \sum_{l=0}^{i} \binom{h^1(L)-h^2(L)-l}{i-l} s_{i-l+m - (h^0(L)-1),G}(V^0) c_{l,G}(V^1-V^2) c_{j,G}( H^2(X , \mathcal{O}))
\end{align*}
where $V^i = H^i(X , L)$, $d = h^0(L) - h^1(L) + h^2(L)$, $r = h^1(L) - h^2(L) + h^{2,0}(X)$. 
\end{theorem}

The expression for $SW_{G,X,\mathfrak{s}_L}^\omega$ given by Theorem \ref{thm:swk} gives a formula for $SW_{G,X,\mathfrak{s}_L}^\omega$ purely in terms of the representations $V^0,V^1,V^2$ and $H^2(X , \mathcal{O})$ and hence purely in terms of the complex geometry of $X$. Next, we restrict to the case $b_+(X) = 1$.

\begin{theorem}
Let $X$ be a compact complex surface with $b_1(X) = 0$ and $b_+(X)=1$. Let $G$ be a finite group that acts on $X$ by biholomorphisms. Let $L$ be a $G$-equivariant line bundle.  Let $D = V^0 - V^1 + V^2$ and $d = h^0(L) - h^1(L) + h^2(L)$.
\begin{itemize}
\item[(1)]{If $h^0(L) > 0$, then
\[
SW_{G,X,\mathfrak{s}_L}^\omega(x^m) = s_{m-(d-1)}(D), \quad SW_{G,X,\mathfrak{s}_L}^{-\omega}(x^m) = 0.
\]
}
\item[(2)]{If $h^2(L) > 0$, then
\[
SW_{G,X,\mathfrak{s}_L}^\omega(x^m) = 0, \quad SW_{G,X,\mathfrak{s}_L}^{-\omega}(x^m) = -s_{m-(d-1)}(D).
\]
\item[(3)]{If $h^0(L) = h^2(L) = 0$, then $SW_{G,X,\mathfrak{s}_L}^{\pm \omega} = 0$.}
}
\end{itemize}
\end{theorem}
\begin{proof}
Part (3) is clear, so it remains to prove (1) and (2). Since $b_+(X) = 1$, we have $H^2(X , \mathcal{O}) \cong H^0(X , K_X)^* = 0$. If $h^0(L), h^2(L) > 0$, then there are non-zero holomorphic sections $\alpha,\beta$ of $L$ and $K^*_X L$. But then $\alpha \beta$ is a non-zero section of $K_X$, which is impossible. This means that at most one of $h^0(L), h^2(L)$ is non-zero. Consider the case that $h^0(L) > 0$, $h^2(L) = 0$. The case $h^0(L) = 0$, $h^2(L) > 0$ will follow from charge conjugation symmetry. If $h^0(L) > 0$, then $SW_{G,X,\mathfrak{s}_L}^{-\omega} = 0$ and hence the wall-crossing formula gives $SW_{G,X, \mathfrak{s}_L}^{\omega}( x^m ) = s_{m-(d-1)}(D)$.
\end{proof}

Next, we consider the case that $X$ is a K3 surface:

\begin{theorem}
Let $X$ be a complex K3 surface and let $G$ be a finite group which acts on $X$ by biholomorphisms. Let $L$ be a $G$-equivariant line bundle.
\begin{itemize}
\item[(1)]{If $c_1(L)$ is not $(1,1)$ then $SW_{G,X,\mathfrak{s}_L}^{\pm \omega} = 0$.}
\item[(2)]{If $c_1(L)$ is $(1,1)$, $c_1(L) \neq 0$ and $G$ acts trivially on $K_X$, then there is only one chamber and $SW_{G,X,\mathfrak{s}_L} = 0$.}
\item[(3)]{If $c_1(L) = 0$ and $G$ acts trivially on $K_X$, then there is only one chamber and $SW_{G,X,\mathfrak{s}_L}(x^m) = s_{1,G}(L)^m$.}
\item[(4)]{If $c_1(L) = 0$ and $G$ acts non-trivially on $K_X$, then $SW_{G,X,\mathfrak{s}_L}^\omega( x^m) = s_{1,G}(L)^m$, $SW_{G,X,\mathfrak{s}_L}^{-\omega} = (s_{1,G}(L) - s_{1,G}(K_X) )^m$.}
\item[(5)]{If $c_1(L)$ is $(1,1)$, $c_1(L) \neq 0$ and $G$ acts non-trivially on $K_X$, then there are two chambers. If $h^0(L) > 0$, then
\[
SW_{G,X,\mathfrak{s}_L}^\omega(x^m) = e_G( H^0(X,K_X) ) s_{m-(d-1)}(D), \quad SW_{G,X,\mathfrak{s}_L}^{-\omega}(x^m) = 0.
\]
If $h^0(L) = 0$, then
\[
SW_{G,X,\mathfrak{s}_L}^\omega(x^m) = 0, \quad SW_{G,X,\mathfrak{s}_L}^{-\omega}(x^m) = -e_G( H^0(X,K_X) ) s_{m-(d-1)}(D).
\]
}
\end{itemize}

\end{theorem}
\begin{proof}
Case (1) is clear, so we can assume that $c_1(L)$ is $(1,1)$.

Suppose that $h^0(L)$ and $h^2(L)$ are non-zero. If $\alpha \in H^0(X,L), \beta \in H^0(X , L^*)$ are non-zero, then $\alpha\beta \in H^0(X,\mathcal{O})$ is non-zero and hence non-vanishing. This means that $\alpha, \beta$ are non-vanishing and hence $c_1(L) = 0$. This means that if $c_1(L) \neq 0$ then $h^0(L)$ or $h^2(L)$ is zero. Now if the action of $G$ on $K_X$ is trivial, then $H^+(X)^G = H^+(X)$ and there is only one chamber. Then it follows that $SW_{G,X,\mathfrak{s}_L} = 0$ since $h^0(L) = 0$ or $h^2(L) = 0$. This proves (2). If $c_1(L) \neq 0$ and $G$ acts non-trivially on $K_X$ then there are two chambers, but we have that either $h^0(L) = 0$, in which case $SW_{G,X,\mathfrak{s}_L}^\omega = 0$, or $h^2(L) = 0$, in which case $SW_{G,X,\mathfrak{s}_L}^{-\omega} = 0$. Then the invariants for the other chamber are given by the wall-crossing formula. This proves (5).

It remains to prove (3) and (4). Hence we assume that $c_1(L) = 0$. This means that $h^0(L) = h^2(L) = 1$, $h^1(L) = 0$, $d = 2$, $r=0$. So Theorem \ref{thm:swk} simplifies to
\[
SW_{G , X , \mathfrak{s}_L}^\omega( x^m ) = \binom{-1}{0} s_{m,G}(V^0) = s_{m,G}(V^0).
\]
But $V^0 = H^0(X , L) \cong L$ is $1$-dimensional, hence $s_{m,G}(V^0) = s_{1,G}(L)^m$. This proves (3). Finally in case (4) the formula for $SW_{X,G,\mathfrak{s}_L}^{-\omega}$ follows from the formula for $SW_{G,X,\mathfrak{s}_L}^\omega$ and charge conjugation, or alternatively from the wall-crossing formula.
\end{proof}

\section{Gluing formulas}\label{sec:glue}

In this section we prove gluing formulas for the equivariant Seiberg--Witten invariants of equivariant connected sums.

Let $f : S^{V,U} \to S^{V',U'}$ be a $G$-monopole map. Define the {\em $G_{\mathfrak{s}}$-equivariant degree} $deg_{G_\mathfrak{s}}(f) \in H^{b_+ - 2d}_{G_{\mathfrak{s}}}(pt ; \mathbb{Z}_w)$ by $f^*( \tau_{V',U'} ) = deg_{G_{\mathfrak{s}}}(f) \tau_{V,U}$. The following result is an extension of \cite[Theorem 6.1]{bar} to the case where $G_{\mathfrak{s}}$ is not necessarily split.

\begin{proposition}\label{prop:donaldson}
Let $f : S^{V,U} \to S^{V',U'}$ be a $G$-monopole map. If $e(H^+) = 0$, then $deg_{G_\mathfrak{s}}(f) = 0$. If $e(H^+) \neq 0$, then $d \le 0$ and $deg_{G_\mathfrak{s}}(f) = e(H^+) s_{G_{\mathfrak{s}} , -d}(D)$. Furthermore in this case we have that $e_G(H^+) s_{G_{\mathfrak{s}} , j}(D) = 0$ for $j > -d$.
\end{proposition}
\begin{proof}
Consider the commutative square
\[
\xymatrix{
S^{V , U } \ar[r]^-{f} & S^{V' , U'} \\
S^{U} \ar[u]^-{i} \ar[r]^-{\iota} & S^{U'} \ar[u]^-{i}
}
\]
where $\iota : S^{U} \to S^{U'}$ is given by inclusion. We have
\[
i^* f^*( \tau_{V' , U'} ) = i^*( \tau_{V , U} deg_{G_\mathfrak{s}}(f)) = \tau_{U} e_{G_{\mathfrak{s}}}(V) deg_{G_\mathfrak{s}}(f).
\]
On the other hand,
\[
\iota^* i^*( \tau_{V' , U'}) = \iota^*( \tau_{U'} e_{G_{\mathfrak{s}}}(V') ) = \tau_{U} e(H^+) e_{G_{\mathfrak{s}}}(V').
\]
Since $f \circ i = i \circ \iota$, we may equate these two expressions, giving
\[
e_{G_{\mathfrak{s}}}(V) deg_{G_\mathfrak{s}}(f) = e(H^+) e_{G_{\mathfrak{s}}}(V').
\]
Using a spectral sequence argument it can be shown that $e_{G_{\mathfrak{s}}}(V)$ is not a zero divisor and so the above equation uniquely determines $deg_{G_\mathfrak{s}}(f)$. However, it seems difficult to obtain an explicit formula for $deg_{G_\mathfrak{s}}(f)$ from this equation. We can remedy this by considering the larger group $\widehat{G}_{\mathfrak{s}} = S^1 \times G_{\mathfrak{s}}$. We let $\widehat{G}_{\mathfrak{s}}$ act on $S^{V,U}$ and $S^{V',U'}$ through the homomorphism $S^1 \times G_{\mathfrak{s}} \to G_{\mathfrak{s}}$ given by $(u,g) \mapsto ug$. Then $f$ can be regarded as a $\widehat{G}_{\mathfrak{s}}$-equivariant map. Repeating the above computation for this larger group gives
\begin{equation}\label{equ:degghat}
e_{\widehat{G}_{\mathfrak{s}}}(V) deg_{\widehat{G}_{\mathfrak{s}}}(f) = e(H^+) e_{\widehat{G}_{\mathfrak{s}}}(V').
\end{equation}
Let $\mathbb{C}_1$ be the $1$-dimensional representation of $S^1 \times G_{\mathfrak{s}}$ where the first factor acts by scalar multiplication and the second factor acts trivially. Then $H^*_{\widehat{G}_{\mathfrak{s}}}(pt ; \mathbb{Z}_w) \cong H^*_{G_{\mathfrak{s}}}(pt ; \mathbb{Z}_w)[y]$ where $y = c_1(\mathbb{C}_1)$. Equation (\ref{equ:degghat}) can be expanded as
\begin{equation*}
( y^a + c_{G,1}(V) y^{a-1} + \cdots + c_{G,a}(V)) deg_{\widehat{G}_{\mathfrak{s}}}(f) = e(H^+)( y^{a'} + c_{G,1}(V') y^{a-1} + \cdots + c_{G,a'}(V') ).
\end{equation*}
If $e(H^+) = 0$, then it follows that $deg_{\widehat{G}_{\mathfrak{s}}}(f) = 0$. Suppose instead that $e(H^+) \neq 0$. Then we can write $deg_{\widehat{G}_{\mathfrak{s}}}(f) = f_0 y^k + f_{1} y^{k-1} + \cdots + f_k$ for some $k \ge 0$, where $f_0 \neq 0$. Expanding the left hand side of the above equation we see that $a' = k+a$, hence $d = a-a' = -k \le 0$. Now consider the above equation in the group of formal Laurent series of the form $\sum_{j = -\infty}^{n} c_n y^n$ where $n$ is any integer and $c_j \in H^*_{G_\mathfrak{s}}(pt ; \mathbb{Z}_w)$. This group is a module over the ring $H^*_{G_\mathfrak{s}}(pt ; \mathbb{Z})[[y^{-1}]]$ of formal power series in $y$ with coefficients in $H^*_{G_\mathfrak{s}}(pt ; \mathbb{Z})$. Multiplying both sides by $1 + y^{-1} s_{G,1}(V) + y^{-2} s_{G,2}(V) + \cdots $ we obtain
\[
y^a ( f_0 y^{-d} + f_{1} y^{-d-1} + \cdots + f_k ) = e(H^+)( y^{a'} + s_{G,1}(D) y^{a'-1} + s_{G,2}(D) y^{a'-2} + \cdots ).
\]
Equating coefficients we see that $f_j = e(H^+) s_{G,j}(D)$ for $j \le -d$ and $e(H^+) s_{G,j}(D) = 0$ for $j > -d$. Thus
\[
deg_{\widehat{G}_{\mathfrak{s}}}(f) = e(H^+)( y^{-d} + s_{G,1}(D) y^{-d-1} + \cdots + s_{G,-d}(D) ).
\]
Restricting to $G_{\mathfrak{s}} \subset \widehat{G}_{\mathfrak{s}}$ amounts to setting $y=0$ in the above equality, giving
\[
deg_{G_{\mathfrak{s}}}(f) = e(H^+) s_{G,-d}(D).
\]
\end{proof}

\subsection{Abstract gluing formula}\label{sec:agf}

We first prove a general gluing formula for the abtract Seiberg--Witten invariants of a smash product of monopole maps. Let $G$ be a finite group and $G_{\mathfrak{s}}$ an $S^1$-central extension of $G$. Suppose that we have two $G$-monopole maps
\[
f_i : S^{V_i , U_i} \to S^{V'_i , U'_i}, \quad i = 1,2
\]
over $B = pt$ and assume $f_1,f_2$ are $G_{\mathfrak{s}}$-equivariant for the same central extension of $G$. Then the smash product
\[
f = f_1 \wedge f_2 : S^{V,U} \to S^{V',U'}
\]
is again $G_{\mathfrak{s}}$-equivariant, where $V = V_1 \oplus V_2$, $U = U_1 \oplus U_2$, $V' = V'_1 \oplus V'_2$, $U' = U'_1 \oplus U'_2$. As usual we write $D = V - V'$, $H^+ = U' - U$. Similarly define $D_i, H^+_i$ for $i=1,2$. Since $(H^+)^G = (H^+_1)^G \oplus (H^+_2)^G$, we see that $f$ admits a chamber if and only if at least one of $f_1$ and $f_2$ admits a chamber. Suppose that this is the case and let $\phi = (\phi_1 , \phi_2)$ be a chamber. Thus either $\phi_1 \neq 0$ or $\phi_2 \neq 0$.

Our orientation conventions are as follows. We orient $H^+$ according to $H^+ = H^+_1 \oplus H^+_2$, orient $U$ according to $U = U_1 \oplus U_2$ and orient $U'$ according to $U' = U \oplus H^+$. Note that this is different from taking $U'$ to be $U'_1 \oplus U'_2$ by a factor of $(-1)^{b_2 (b_{+,1})}$, where $b_i$ denotes the rank of $U_i$ and $b_{+,i}$ the rank of $H^+_i$.

\begin{theorem}\label{thm:glue}
If $(H^+_1)^G, (H^+_2)^G$ are both non-zero, then the Seiberg--Witten invariants of $f$ are all zero. If $(H^+_1)^G \neq 0$, $(H^+_2)^G = 0$, then
\[
SW_{G,f}^\phi( \theta ) = SW_{G,f_1}^{\phi_1}(e_G(H^+_2) s_{G_{\mathfrak{s}} , -d_2}(D_2) \theta ).
\]
where $s_{G_{\mathfrak{s}} , -d_2}$ is taken to be zero if $d_2 > 0$. In particular, if $G_\mathfrak{s} \cong S^1 \times G$ is the trivial extension then
\[
SW_{G,f}^\phi( \theta ) = \sum_{k+l=-d_2} s_{G,l}(D_2) SW_{G,f_1}^{\phi_1}( x^k e_G( H^+_2 ) \theta ).
\]
\end{theorem}
\begin{proof}
Without loss of generality we can assume $(H^+_1)^G \neq 0$. If we also have $(H^+_2)^G \neq 0$, then $dim( (H^+)^G ) > 1$ and so $SW_{G,f}^{\phi}$ is independent of the choice of chamber. Hence without loss of generality we can assume that $\phi = (\phi_1 , 0)$, where $\phi_1 \neq 0$. We have that $SW_{G,f}^\phi(\theta) = (\pi_{\mathbb{P}(V)})_*( \eta^\phi \theta)$ where $f^*( \tau^\phi_{V',U'}) = (-1)^{b} \delta \tau_U \eta^\phi$. Similarly $SW_{G,f_1}^{\phi_1}(\theta) = (\pi_{\mathbb{P}(V_1)})_*( \eta^{\phi_1}_1 \theta)$ where $f^*( \tau^{\phi_1}_{V'_1,U'_1}) = (-1)^{b_1} \delta \tau_{U_1} \eta^{\phi_1}_1$. Recall that the equivariant degree $deg_{G_\mathfrak{s}}(f_2)$ of $f_2$ is defined by $f_2^*( \tau_{V'_2,U'_2}) = \tau_{V_2,U_2} deg_{G_\mathfrak{s}}(f_2)$.

Consider the external cup product
\[
H^i_{G_{\mathfrak{s}}}( S^{V'_1,U'_1} , S^{U_1} ; \mathbb{Z} ) \times H^j_{G_{\mathfrak{s}}}( S^{V'_2,U'_2} , pt ; \mathbb{Z}) \to H^{i+j}_{G_{\mathfrak{s}}}( S^{V',U'} , S^U ; \mathbb{Z}).
\]
Using this cup product we can write
\[
\tau^\phi_{V',U'} =  (-1)^{ b_{+,1}b_2 }\tau^{\phi_1}_{V'_1,U'_2} \tau_{V'_2,U'_2}.
\]
The sign factor $(-1)^{b_{+,1} b_2}$ arises because our orientation conventions described above. From this, we find
\begin{align*}
(-1)^{b} \delta \tau_U \eta^\phi &= f^*( \tau^\phi_{V',U'} ) \\
&= (-1)^{b_{+,1}b_2} (f_1 \wedge f_2)^*( \tau^{\phi_1}_{V'_1,U'_1} \tau_{V'_2, U'_2} ) \\
&= (-1)^{b_{+,1}b_2} f_1^*( \tau^{\phi_1}_{V'_1,U'_1} ) f_2^*( \tau_{V'_2 , U'_2} ) \\
&= (-1)^{b_{+,1}b_2 + b_1} \delta \tau_{U_1} \eta^{\phi_1}_1  \tau_{V_2 , U_2} deg_{G_\mathfrak{s}}(f_2) \\
&= (-1)^{b_{+,1}b_2 + b_1 + (b_{+,1}+1)b_2} \delta \tau_{U} e_{G_{\mathfrak{s}}}(V_2) \eta^{\phi_1}_1 deg_{G_\mathfrak{s}}(f_2) \\
&= (-1)^{b} \delta \tau_{U} e_{G_{\mathfrak{s}}}(V_2) \eta^{\phi_1}_1 deg(f_2).
\end{align*}
Hence
\begin{equation}
\eta^\phi = e_{G_{\mathfrak{s}}}(V_2) \eta^{\phi_1}_1 deg_{G_\mathfrak{s}}(f_2).
\end{equation}

Let $\iota_1 : \mathbb{P}(V_1) \to \mathbb{P}(V_2)$ be the inclusion map. Since $(\iota_1)_*(1) = e_{G_{\mathfrak{s}}}(V_2)$, we have
\begin{align*}
SW_{G,f}^{\phi}(\theta) &= (\pi_{\mathbb{P}(V)})_*( \eta^\phi \theta ) \\
&= (\pi_{\mathbb{P}(V)})_*( e_{G_{\mathfrak{s}}}(V_2) \eta^{\phi_1}_1 deg_{G_\mathfrak{s}}(f_2) \theta ) \\
&= (\pi_{\mathbb{P}(V)})_* (\iota_1)_*(  \eta^{\phi_1}_1 deg_{G_\mathfrak{s}}(f_2) \theta ) \\
&= (\pi_{\mathbb{P}(V_1)})_*( \eta^{\phi_1}_1 deg_{G_\mathfrak{s}}(f_2) \theta ) \\
&= SW^{\phi_1}_{G,f_1}( deg_{G_\mathfrak{s}}(f_2) \theta ).
\end{align*}
Now if $(H^+_2)^G \neq 0$, then $e(H^+_2) = 0$ and $deg_{G_\mathfrak{s}}(f_2) = 0$ by Proposition \ref{prop:donaldson}. Hence in this case $SW_{G,f}^\phi$ vanishes. If $(H^+_2)^G = 0$, then $d_2 \le 0$ and Proposition \ref{prop:donaldson} gives $deg_{G_\mathfrak{s}}(f_2) = e(H^+_2) s_{G_{\mathfrak{s}} , -d_2}(D_2)$, which gives
\[
SW_{G,f}^\phi( \theta ) = SW_{G,f_1}^{\phi_1}(e_G(H^+_2) s_{G_{\mathfrak{s}} , -d_2}(D_2) \theta ).
\]
Lastly if $G_{\mathfrak{s}} \cong S^1 \times G$ is the trivial extension, then 
\[
s_{G_{\mathfrak{s}} , -d_2}(D_2) = \sum_{k=0}^{-d_2} x^k s_{G , -d_2-k}(D_2)
\]
and hence
\[
SW_{G,f}^\phi( \theta ) = \sum_{k+l=-d_2} s_{G,l}(D_2) SW_{G,f_1}^{\phi_1}( x^k e_G( H^+_2 ) \theta ).
\]
\end{proof}

\subsection{Equivariant connected sum}\label{sec:ecs}

Let $X,Y$ be compact, oriented, smooth $4$-manifolds. Let $G$ be a finite group acting smoothly and orientation preservingly on $X$. Let $H$ be a subgroup of $G$ and assume that $H$ acts smoothly and orientation preservingly on $Y$. We will construct an action of $G$ on the connected sum of $X$ with $|G/H|$ copies of $Y$. For this we need to assume that $X$ and $Y$ contain points $x,y$ whose stabiliser groups are $H$. We also need that $T_x X, T_y Y$ are isomorphic as real representations of $H$ via an orientation reversing isomorphism. Choose such an isomorphism $\psi : T_x X \to T_y Y$. Then we can form the {\em $G$-equivariant connected sum} $X \# ind^G_H(Y)$ as follows. Here $ind^G_H(Y)$ is the disjoint union of $G/H$ copies of $Y$ which is made into a $G$-space by taking $ind^G_H(Y) = G \times_H Y$. We remove from $X$ and $ind^G_H(Y)$ equivariant tubular neighbourhoods $\nu( Gx), \nu(Gy)$ of the orbits $Gx,Gy$ and identify the boundaries of $X \setminus \nu(Gx)$, $ind^G_H(Y) \setminus \nu(Gy)$ by the orientation reversing $G$-equivariant isomorphism $\partial \nu(Gx) \to \partial \nu(Gy)$ determined by $\psi$. We note that the construction of $X \# ind^G_H(Y)$ depends on the choice of points $x,y$ and the isomorphism $\psi$. The underlying smooth manifold of $X \# ind^G_H(Y)$ is $X \# |G/H|Y$, the connected sum of $X$ with $|G/H|$ copies of $Y$. Note also that if $b_1(X) = b_1(Y) = 0$, then $b_1(X \# ind^G_H(Y) ) = 0$.

Suppose that $\mathfrak{s}_X$ is a $G$-invariant spin$^c$-structure on $X$ and $\mathfrak{s}_Y$ is a $H$-invariant spin$^c$-structure on $Y$. Let $S^{\pm}$ denote the spinor bundles corresponding to $\mathfrak{s}_Y$. Since $H$ fixes $y$, we see that $H_{\mathfrak{s}_Y}$ lifts to the fibres $S^{\pm}_y$. Hence the extension class of $H_{\mathfrak{s}_Y}$ is given by $Sq_3^{\mathbb{Z}}( T_y Y ) \in H^3_H( pt ; \mathbb{Z} )$, the third integral Stiefel--Whitney class of $T_y Y$. Similarly, since $H$ fixes $x$, we see that the restriction $G_{\mathfrak{s}_X}|_H$ of the extension $G_{\mathfrak{s}_X}$ to $H$ has extension class $Sq_3^{\mathbb{Z}}( T_x X )$. Since $T_x X$ and $T_y Y$ are orientation reversingly isomorphic, it follows that the extension classes agree (reversing orientation has the effect of exchanging the roles of positve and negative spinors). Therefore $H_{\mathfrak{s}_Y}$ and $G_{\mathfrak{s}_X}|_H$ are isomorphic as extensions of $H$. The isomorphism $\psi : T_x X \to T_y Y$ determines the isomorphism $G_{\mathfrak{s}_X}|_H \to H_{\mathfrak{s}_Y}$. We can induce a $G$-invariant spin$^c$-structure $ind^G_H(\mathfrak{s}_Y)$ on $ind^G_H(Y)$ as follows. We take the spinor bundles of $ind^G_H(\mathfrak{s}_Y)$ to be given by $G_{\mathfrak{s}_X} \times_{H_{\mathfrak{s}_Y}} S^{\pm}$. The underlying isomorphism class of $ind^G_H(\mathfrak{s}_Y)$ is simply given by equipping each connected component of $ind^G_H(Y)$ with a copy of $\mathfrak{s}_Y$. However our construction makes it clear that the extension of $G$ determined by $ind^G_H(\mathfrak{s}_Y)$ is isomorphic to $G_{\mathfrak{s}_X}$. It follows that the spin$^c$-structures $\mathfrak{s}_X$ and $ind^G_H(\mathfrak{s}_Y)$ can be identified equivariantly on the boundaries of $X \setminus \nu(Gx)$ and $ind^G_H(Y) \setminus \nu(Gx)$ to form a spin$^c$-structure $\mathfrak{s}_X \# ind^G_H(\mathfrak{s}_Y)$ on $X \# ind^G_H(\mathfrak{s}_Y)$.

To summarise, given a $G$-invariant spin$^c$-structure $\mathfrak{s}_X$ on $X$ and a $H$-invariant spin$^c$-structure $\mathfrak{s}_Y$ on $Y$, the choice of an orientation reversing isomorphism $\psi : T_x X \to T_y Y$ of representations of $H$ allows us to construct a $G$-invariant spin$^c$-structure $\mathfrak{s} = \mathfrak{s}_X \# ind^H_G(\mathfrak{s}_Y)$ on $X \# ind^G_H(Y)$ and moreover the extensions $G_{\mathfrak{s}}$ and $G_{\mathfrak{s}_X}$ are isomorphic.

Suppose $f : S^{V , U} \to S^{V',U'}$ is an $H$-equivariant monopole map with respect to some extension $H_{\mathfrak{s}}$.  Suppose $H \subseteq G$ and $H_{\mathfrak{s}} = G_{\mathfrak{s}}|_H$ for some extension $G_{\mathfrak{s}}$. We will define an induced monopole map
\[
ind^{G_{\mathfrak{s}}}_{H_{\mathfrak{s}}}(f) : S^{ind^{G_{\mathfrak{s}}}_{H_{\mathfrak{s}}}(V) , ind^{G_{\mathfrak{s}}}_{H_{\mathfrak{s}}}(U) } \to S^{ind^{G_{\mathfrak{s}}}_{H_{\mathfrak{s}}}(V') , ind^{G_{\mathfrak{s}}}_{H_{\mathfrak{s}}}(U') }
\]
as follows. Choose representatives $g_1, \dots , g_n$ for the left cosets of $H$ in $G$ and choose lifts $\tilde{g}_i$ of $g_i$ to $G_{\mathfrak{s}}$. Then $\tilde{g}_1, \dots , \tilde{g}_n$ are representatives for the left cosets of $H_{\mathfrak{s}}$ in $G_{\mathfrak{s}}$. For any representation $W$ of $H_{\mathfrak{s}}$, $ind^{G_{\mathfrak{s}}}_{H_{\mathfrak{s}}}(W) = \mathbb{R}[G_{\mathfrak{s}}] \otimes_{\mathbb{R}[H_{\mathfrak{s}}]} W$. Regard $S^W$ as $W \cup \{\infty\}$ and similarly $S^{ind^{G_{\mathfrak{s}}}_{H_{\mathfrak{s}}}(W)} = ind^{G_{\mathfrak{s}}}_{H_{\mathfrak{s}}}(W) \cup \{\infty\}$. Then $ind^{G_{\mathfrak{s}}}_{H_{\mathfrak{s}}}(f)$ is defined by 
\[
ind^{G_{\mathfrak{s}}}_{H_{\mathfrak{s}}}(f)( \sum_i g_i \otimes w_i ) = \begin{cases} \sum_i g_i \otimes f(w_i) & \text{if } f(w_i) \neq \infty \text{ for all } i, \\ \infty & \text{otherwise}. \end{cases}
\]
This is well-defined because $f$ is $H_{\mathfrak{s}}$-equivariant.

\begin{theorem}\label{thm:ecs}
Let $X,Y$ be compact, oriented, smooth $4$-manifolds with $b_1(X) = b_1(Y) = 0$. Let $G$ be a finite group acting smoothly and orientation preservingly on $X$. Let $H$ be a subgroup of $G$ and suppose that $H$ acts smoothly and orientation preservingly on $Y$. Suppose that there is an $x \in X$ and $y \in Y$ with stabiliser group $H$ and an orientation reversing isomorphism $\psi : T_x X \to T_y Y$ of representations of $H$. Suppose that $\mathfrak{s}_X$ is a $G$-invariant spin$^c$-structure on $X$ and $\mathfrak{s}_Y$ is a $H$-invariant spin$^c$-structure on $Y$. Set $Z = X \# ind^G_H(Y)$, $\mathfrak{s} = \mathfrak{s}_X \# ind^G_H(\mathfrak{s}_Y)$. Then:
\begin{itemize}
\item[(1)]{If $H^+(X)^G, H^+(Y)^H$ are both non-zero then the equivariant Seiberg--Witten invariants of $Z$ vanish.}
\item[(2)]{If $H^+(X)^G \neq 0$ and $H^+(Y)^H = 0$, then for any chamber $\phi \in H^+(X)^G \setminus \{0\}$ we have
\[
SW^\phi_{G , Z , \mathfrak{s}}( \theta ) = SW^\phi_{G , X , \mathfrak{s}_X}( e(ind^G_H(H^+(Y))) s_{G_{\mathfrak{s}} , -d_Y}( ind^{G_{\mathfrak{s}}}_{H_{\mathfrak{s}}}( D_Y ) ) \theta ).
\]
}
\end{itemize}

\end{theorem}
\begin{proof}
Let $f_X, f_Y, f_Z$ denote the equivariant monopole maps for $X,Y$ and $Z$. Bauer's connected sum formula \cite{b2} extends easily to the $G$-equivariant setting. Then by a straightforward extension of \cite[Theorem 3.1]{sung}, we see that $f_Z = f_X \wedge ind^{G_{\mathfrak{s}}}_{H_{\mathfrak{s}}}(f_Y)$.

Set $f = f_Z$, $f_1 = f_X$, $f_2 = ind^{G_{\mathfrak{s}}}_{H_{\mathfrak{s}}}(f_Y)$. We use the notation $H^+, D, H^+_i, D_i$ as in Section \ref{sec:agf}. Then $H^+_1 = H^+(X)$, $H^+_2 = ind^{G_{\mathfrak{s}}}_{H_{\mathfrak{s}}}(H^+(Y))$, $(H^+_1)^G = H^+(X)^G$, $(H^+_2)^G = H^+(Y)^H$. Hence if $H^+(X)^G, H^+(Y)^H$ are both non-zero, then the Seiberg--Witten invariants of $Z$ vanish by Theorem \ref{thm:glue}.

Now suppose $H^+(X)^G \neq 0$ and $H^+(Y)^H = 0$. Theorem \ref{thm:glue} gives
\[
SW^\phi_{G,Z , \mathfrak{s}}( \theta ) = SW^\phi_{G , X , \mathfrak{s}}( e( H^+_2 ) s_{G_{\mathfrak{s}} , -d_Y }( D_2 ) \theta ).
\]
The result follows, since $H^+_2 = ind^{G_{\mathfrak{s}}}_{H_{\mathfrak{s}}}(H^+(Y))$ and $D_2 = ind^{G_{\mathfrak{s}}}_{H_{\mathfrak{s}}}(D_Y)$.
\end{proof}

We consider now the case where $H^+(X)^G = 0$ and $H^+(Y)^H \neq 0$. Here it is more difficult to obtain a general formula for the cohomological invariants so we restrict to the case $G = \mathbb{Z}_p$ and $H = 1$. Let $\mathbb{C}_j$ denote the $1$-dimensional complex representation of $\mathbb{Z}_p = \langle g \rangle$ where $g$ acts as multiplication by $\omega^j$, $\omega = e^{2\pi i /p}$. When $p$ is odd we may choose the orientation on $ind^G_1(H^+(Y))$ such that $ind^G_1(H^+(Y))/H^+(Y) \cong \bigoplus_{j=1}^{(p-1)/2} \mathbb{C}_j^{b_+(Y)}$.

Recall that $H^*_{\mathbb{Z}_p}(pt ; \mathbb{Z}) \cong \mathbb{Z}[v]/(pv)$, $deg(v) = 2$. When $p=2$ to avoid orientability issues we use $\mathbb{Z}_2$-coefficients. We have $H^*_{\mathbb{Z}_2}(pt ; \mathbb{Z}_2) \cong \mathbb{Z}_2[u]$, $deg(u) = 1$. In this case we set $v = u^2$ and we will also denote $u$ by $v^{1/2}$.

\begin{theorem}\label{thm:sump}
Let $X,Y$ be compact, oriented, smooth $4$-manifolds with $b_1(X) = b_1(Y) = 0$. Let $G = \mathbb{Z}_p$ where $p$ is prime act smoothly and orientation preservingly on $X$. Suppose that $\mathfrak{s}_X$ is a $G$-invariant spin$^c$-structure on $X$ and $\mathfrak{s}_Y$ is a spin$^c$-structure on $Y$ with $d(Y,\mathfrak{s}_Y) = 2d_Y - b_+(Y) - 1 = 0$. Set $Z = X \# ind^G_1(Y)$, $\mathfrak{s} = \mathfrak{s}_X \# ind^G_1(\mathfrak{s}_Y)$. Suppose that $H^+(X)^G = 0$ and $b_+(Y) > 0$. Then
\[
SW_{G , Z , \mathfrak{s}}^\phi( x^m ) = (-1)^{d_Y+1} h {\sum_l}' e(H^+(X)) s_{-d_X-l}(D_X) SW(Y , \mathfrak{s}_Y , \phi_Y) v^{m+l - (p-1)/2}
\]
where the sum ${\sum_l}'$ is over $l$ such that $0 \le l \le -d_X$, $m+l > 0$, $m+l = 0 \; ({\rm mod} \; p-1)$ and $h = \prod_{j=1}^{(p-1)/2} j^{b_+(Y)}$ for $p \neq 2$, $h=1$ for $p=2$.
\end{theorem}
\begin{proof}
We give the proof for $p \neq 2$. The case $p=2$ is similar. Let $f_X,f_Y,f_Z$ denote the monopole maps for $X,Y,Z$. We have that $f_Z = ind^G_1(f_Y) \wedge f_X$. Theorem \ref{thm:glue} gives
\[
SW_{G,Z , \mathfrak{s}}^\phi( x^m ) = SW_{G,ind^G_1(f_Y)}^{\phi_Y}(e_G(H^+(X)) s_{G_{\mathfrak{s}} , -d_X}(D_X) x^m ).
\]
Next we will compute $SW_{G , ind^G_1(f_Y)}^{\phi_Y}$ using Theorem \ref{thm:loc}. It is easy to see that for each splitting $s_j : \mathbb{Z}_p \to S^1 \times \mathbb{Z}_p$, we have $(ind^G_1(f_Y))^{s_j} = f_Y$. Furthermore, $ind^G_1(D_Y) \cong \bigoplus_{j=0}^{p-1} \mathbb{C}_j^{d_Y}$. As explained above we can orient $ind^G_1(H^+(Y))$ in such a way that $ind^G_1(H^+(Y))/H^+(Y) \cong \bigoplus_{j=1}^{(p-1)/2} \mathbb{C}_j^{b_+(Y)}$. Then
\begin{align*}
SW^{\phi_Y}_{G,ind^G_1(f_Y)}( x^m ) &= e_G( ind^G_1(H^+(Y))/H^+(Y) ) \sum_{j=0}^{p-1} SW^{\phi_Y}_{G, f_Y}( e_{S^1 \times G}(D/D_j)^{-1} (x+jv)^m ) \\
&= \prod_{r=1}^{(p-1)/2} (rv)^{b_+(Y)} \sum_{j=0}^{p-1} SW^{\phi_Y}_{G,f_Y}\left( \psi_{s_j}^{-1} \left(\prod_{k=1}^{p-1} (x + kv)^{-d_Y} (x-jv)^m \right) \right)
\end{align*}
where $\psi_{s_j}$ was defined in Section \ref{sec:zpact}. Since $d(Y , \mathfrak{s}_Y) = 0$, we have $SW^{\phi_Y}_{G , f_Y}(1) = SW(Y,\mathfrak{s}_Y,\phi_Y)$ and $SW^{\phi_Y}_{G,f_Y}( \psi_{s_j}^{-1}(x^m)) = 0$ for $m > 0$. Hence 
\begin{align*}
SW^{\phi_Y}_{G,ind^G_1(f_Y)}( x^m ) &= (-1)^{d_Y} \prod_{r=1}^{(p-1)/2} (rv)^{b_+(Y)} SW(Y,\mathfrak{s}_Y,\phi_Y) v^{m-d_Y(p-1)} \sum_{j=0}^{p-1} (-j)^m \\
&= (-1)^{d_Y} h \, SW(Y , \mathfrak{s}_Y , \phi_Y) v^{m - (p-1)/2} \sum_{j=0}^{p-1} j^m
\end{align*}
where $h = \prod_{j=1}^{(p-1)/2} j^{b_+(Y)}$. But $\sum_{j=0}^{p-1} j^m$ equals $-1$ mod $p$ if $m > 0$ and $(p-1)$ divides $m$ and is zero otherwise. Hence
\[
SW^{\phi_Y}_{G,ind^G_1(f_Y)}( x^m ) = \begin{cases} (-1)^{d_Y+1} h \, SW(Y,\mathfrak{s}_Y , \phi_Y) & m > 0, m = 0 \; ({\rm mod} \; p-1), \\ 0 & \text{otherwise} \end{cases}
\]
Then using 
\[
s_{G_{\mathfrak{s}} , -d_X}(D_X) = \sum_{l = 0}^{-d_X} x^l s_{-d_X - l}(D_X),
\]
we obtain
\[
SW_{G,Z , \mathfrak{s}}^\phi( x^m ) =  (-1)^{d_Y+1} h {\sum_l}' e(H^+(X)) s_{-d_X-l}(D_X) SW(Y , \mathfrak{s}_Y , \phi_Y) v^{m+l - (p-1)/2}
\]
where the sum is over $l$ with $0 \le l \le -d_X$, $m+l > 0$ and $m+l = 0 \; ({\rm mod} \; p-1)$.
\end{proof}

The following corollary follows immediately from Theorem \ref{thm:sump}.

\begin{corollary}\label{cor:psumx}
Let $p$ be a prime. Let $X$ be a compact, oriented, smooth $4$-manifold with $b_1(X) = 0$ and $b_+(X)>0$. Let $\mathfrak{s}$ be a spin$^c$-structure on $X$ such that $d(X,\mathfrak{s}) = 0$. If $b_+(X)=1$ then fix a chamber. Let $G = \mathbb{Z}_p$ act on the connected sum $\# pX$ of $p$ copies of $X$ by cyclically permuting the summands (to be precise, we take the equivariant connected sum $S^4 \# pX$ where $\mathbb{Z}_p$ acts on $S^4$ by rotation). Then
\[
SW_{G , \# pX , \# p\mathfrak{s}}( x^m ) = \begin{cases} (-1)^{d_X+1} h SW(X , \mathfrak{s}_X) v^{m - (p-1)/2} & m > 0, m = 0 \; ({\rm mod} \; p-1), \\ 0 & \text{otherwise} \end{cases}
\]
where $h = \prod_{j=1}^{(p-1)/2} j^{b_+(X)}$ for $p \neq 2$, $h=1$ for $p=2$.
\end{corollary}

\section{Some examples}\label{sec:ex}

\subsection{Constraints on group actions}

Consider the case where $G = \mathbb{Z}_p$ for a prime $p$. We use notation from Section \ref{sec:zpact}. Combining the divisibility condition Theorem \ref{thm:div} with Theorem \ref{thm:zp}, we obtain the following constraint on smooth $\mathbb{Z}_p$-actions:

\begin{theorem}\label{thm:constrzp}
Let $X$ be a compact, oriented, smooth $4$-manifold with $b_1(X) = 0$. Let $G = \mathbb{Z}_p$ act smoothly on $X$ and let $\mathfrak{s}$ be a spin$^c$-structure preserved by $G$. If $SW(X,\mathfrak{s}) \neq 0 \; ({\rm mod} \; p)$ and $b_0 \neq 1 \; ({\rm mod} \; 2p)$ then there exists an $i$ such that $0 \le 2d_i - b_0 - 1 \le 2(p-2)$.
\end{theorem}
\begin{proof}
If the condition $0 \le 2d_i - b_0 - 1 \le 2(p-2)$ is not satisfied then Theorem \ref{thm:div} implies that $\overline{SW}^{s_i}_{G , X , \mathfrak{s}} = 0 \; ({\rm mod} \; p)$. If this holds for all $i$ then Theorem \ref{thm:zp} implies that $SW_{G,X,\mathfrak{s}} = 0 \; ({\rm mod} \; p)$ and hence $SW(X,\mathfrak{s}) = 0 \; ({\rm mod} \; p)$.
\end{proof}

\begin{corollary}
Let $X$ be a compact, oriented, smooth $4$-manifold with $b_1(X) = 0$. Let $G = \mathbb{Z}_2$ act smoothly on $X$ and let $\mathfrak{s}$ be a spin$^c$-structure preserved by $G$. If $SW(X,\mathfrak{s})$ is odd and $b_0 \neq 1 \; ({\rm mod} \; 4)$ then there exists an $i$ such $2d_i = b_0 + 1$.
\end{corollary}

\begin{example}
Let $K$ denote a $K3$ surface given as a degree $3$ cyclic branched cover of $\mathbb{CP}^1 \times \mathbb{CP}^1$ branched over a smooth curve $\Sigma$ of bi-degree $(3,3)$. This gives an action of $G = \mathbb{Z}_3$ on $K$ with fixed point set a surface of genus $4$ and self-intersection $6$. Similarly we can realise $4(S^2 \times S^2)$ as the branched triple cover of an unknotted surface in $S^4$ of genus $2$ \cite{ak}. This gives an action of $G$ on $4(S^2 \times S^2)$ with fixed point set a surface of genus $2$ and self-intersection zero. Now consider the equivariant connected sum $X_0 = 4(S^2 \times S^2) \# 5 K$ of $4(S^2 \times S^2)$ and five copies of $K$ with the given $\mathbb{Z}_3$-action. This gives a $\mathbb{Z}_3$-action on $X_0$ with fixed point set a single surface of genus $22$ and self-intersection $30$. The $4$-manifold $X_0$ is homeomorphic to the elliptic surface $X = E(10)$, hence the $\mathbb{Z}_3$-action on $X_0$ also defines a continuous, locally linear $\mathbb{Z}_3$-action on $X$. On the other hand we will use Theorem \ref{thm:constrzp} to show that there is no smooth $\mathbb{Z}_3$-action on $X$ with the same fixed point data. Let $\mathfrak{s}$ denote the unique spin structure on $X$. Then $SW(X,\mathfrak{s}) = \binom{8}{4} = 1 \; ({\rm mod} \; 3)$. A smooth $\mathbb{Z}_3$-action on $X$ with fixed point set a surface of genus $22$ and self-intersection $30$ will have $b_0 = 5$ by the $G$-signature theorem and $d_0 = 0$, $d_1=d_2 = 5$ by the $G$-spin theorem. But this contradicts Theorem \ref{thm:constrzp} which requires $3 \le d_i \le 4$ for some $i$. So such an action does not exist.
\end{example}

\subsection{Exotic group actions}

The gluing formula gives a method of contructing group actions which are homeomorphic but not diffeomorphic. Let $X_1, X_2$ be compact oriented smooth $4$-manifolds with $b_1 = 0$ and $b_+ > 1$. Assume that there is a homeomorphism $\varphi : X_1 \to X_2$, but that $X_1$, $X_2$ have different mod $p$ Seiberg--Witten invariants for a prime $p$, so in particular they are not diffeomorphic. More precisely we will assume the following. For simplicity assume $H_1(X_i ; \mathbb{Z}) = 0$ for $i=1,2$ so that spin$^c$-structures can be identified with characteristics. Then we will require that there does not exist an isometry $\psi : H^2(X_1 ; \mathbb{Z}) \to H^2(X_2 ; \mathbb{Z})$ for which $SW(X_1 , \mathfrak{s}_1) = SW(X_2 , \psi(\mathfrak{s}_1)) \; ({\rm mod} \; p)$ for every spin$^c$-structure $\mathfrak{s}_1$ with $d(X_1 , \mathfrak{s}_1) = 0$ and for some choice of orientations of $H^+(X_i)$.

Given $g > 0$, let $Y_g$ denote the degree $p$ cyclic branched cover of an unknotted embedded surface in $S^4$ of genus $g$. Then $Y_p$ is diffeomorphic to $\#^{g(p-1)} S^2 \times S^2$. The branched covering construction defines an action of $G = \mathbb{Z}_p$ on $Y_g$. Moreover $H^2(Y_g ; \mathbb{Z})^G = 0$, in particular $H^+(Y_g)^G = 0$. By \cite{gom} there exists a $k>0$ such that $X_1 \# k(S^2 \times S^2)$ and $X_1 \# k(S^2 \times S^2)$ are diffeomorphic. This also implies that $p X_1 \# k(S^2 \times S^2)$ and $pX_2 \# k(S^2 \times S^2)$ are diffeomorphic. Thus if $g(p-1) \ge k$ then we get two $\mathbb{Z}_p$-actions on $X = pX_1 \# Y_g \cong pX_2 \# Y_g$. The first is obtained by taking $Y_g$ with the $\mathbb{Z}_p$-action described above and attaching $p$ copies of $X_1$ which are permuted by the $\mathbb{Z}_p$-action. The second action on $X$ is given by the same construction but with $X_2$ in place of $X_1$. These two actions are equivariantly homeomorphic since we can apply the homeomorphism $\varphi : X_1 \to X_2$ to each copy of $X_1$. On the other hand they are not equivarianly diffeomorphic. To see this, note that since $H^2(Y_g ; \mathbb{Z})^G = 0$, one finds that the $G$-invariant spin$^c$-structures on $X_i \# Y_g$ are precisely those of the form $ind^{G}_1(\mathfrak{s}_{i}) \# \mathfrak{s}$ where $\mathfrak{s}_i $ is a spin$^c$-structure on $X_i$ and $\mathfrak{s}$ is the unique spin structure on $Y_g$. If $d(X_i , \mathfrak{s}_i) = 0$ then Theorem \ref{thm:sump} gives
\[
SW_{G , X , ind^G(\mathfrak{s}_i) \# \mathfrak{s}}( x^{p-1} ) = (-1)^{g(p-1)/2} SW( X_i , \mathfrak{s}_i) v^{(p-1)/2} \; ({\rm mod} \; p).
\]
Our assumption that $X_1,X_2$ have different mod $p$ Seiberg--Witten invariants then implies that the two different $G$-actions have different equivariant Seiberg--Witten invariants, so they are not diffeomorphic.

\subsection{Obstructions to enlarging group actions}

Let $G$ be a finite group and $H$ a subgroup of $G$. The compatibility of the equivariant Seiberg--Witten invariants with the restriction map from $G$-equivariant cohomology (or $K$-theory) to $H$-equivariant implies a condition for a smooth $H$-action to extend to $G$. This follows immediately from Theorem \ref{thm:propb} (2), but we restate it here from the perspective of extending a group action.

\begin{proposition}\label{prop:extending}
Let $X$ be a compact, oriented smooth $4$-manifold with $b_1(X) = 0$. Suppose that a finite group $H$ acts on $X$ by orientation preserving diffeomorphisms and suppose that $H^+(X)^H \neq 0$. Let $\mathfrak{s}$ be a $H$-invariant spin$^c$-structure and let $\phi \in H^+(X)^H \setminus \{0\}$ be a chamber. Let $G$ be a finite group containing $H$. If the $H$-action on $X$ extends to a smooth, orientation preserving action of $G$ which fixes $\mathfrak{s}$ and $\phi$, then $H_{\mathfrak{s}}$ is the restriction to $H$ of an $S^1$ central extension $G_{\mathfrak{s}} \to G$ and for every $\theta$ in the image of the restriction map $H^*_{G_{\mathfrak{s}}}(pt ; A) \to H^*_{H_{\mathfrak{s}}}(pt ; A)$, we have that $SW_{H,X,\mathfrak{s}}^\phi(\theta)$ is in the image of $H^*_{G}(pt ; A_w) \to H^*_H(pt ; A_w)$.

Furthermore, if $b_+(X)$ is odd and a $H$-equivariant spin$^c$-struture $\mathfrak{o}$ on $H^+(X)$ is given which can be lifted to a $G$-equivariant spin$^c$-structure, then for every $\theta$ in the image of $R(G_{\mathfrak{s}}) \to R(H_{\mathfrak{s}})$, we have that $SW_{H,X,\mathfrak{s}}^{\phi , K}(\theta)$ is in the image of $R(G) \to R(H)$.
\end{proposition}

\begin{example}
Let us consider Proposition \ref{prop:extending} in the case that $H = \mathbb{Z}_p = \langle \sigma \; | \; \sigma^p \rangle$ is cyclic of odd prime order and $G = D_p = \langle \sigma , \tau \; | \; \sigma^p , \tau^2, (\tau \sigma)^2 \rangle$ is the dihedral group of order $2p$. Both $G$ and $H$ have no non-trivial $S^1$ central extensions. Let $w \in H^1_{D_p}(pt ; \mathbb{Z}_2) \cong \mathbb{Z}_2$ be the unique non-trivial element. Recall that $H^{2k}_{\mathbb{Z}_p}(pt ; \mathbb{Z}) \cong \mathbb{Z}_p$ for every $k > 0$. On the other hand a simple calculation shows that $H^{2k}_{D_p}(pt ; \mathbb{Z}) = 0$ for odd $k$ and $H^{2k}_{D_p}(pt ; \mathbb{Z}_w) = 0$ for even $k > 0$. Also $R(\mathbb{Z}_p) \cong \mathbb{Z}[t]/(t^p-1)$ and the image of $R(D_p) \to R(\mathbb{Z}_p)$ is the subring generated by $t+t^{-1}$. From this we obtain non-trivial conditions for a smooth $\mathbb{Z}_p$-action on $X$ to extend to $D_p$, where $X$ is a compact, oriented smooth $4$-manifold with $b_1(X) = 0$ and $H$ acts smoothly and orientation preservingly. Suppose $\mathfrak{s}$ is a $H$-invariant spin$^c$-structure on $X$ and $\phi$ is a $H$-invariant chamber. If the $\mathbb{Z}_p$-action extends to a smooth orientation preserving action of $D_p$ which preserves $\mathfrak{s}$ and $\phi$. Assume $b_+(X)$ is odd, so $d(X,\mathfrak{s})$ is even. Then the condition is that $SW^\phi_{\mathbb{Z}_p , X , \mathfrak{s}}( x^m ) = 0$ whenever $2m-d(X,\mathfrak{s})$ is positive and equals $2$ mod $4$ if $\sigma$ preserves orientation on $H^+$, or equals $0$ mod $4$ if $\sigma$ reverses orientation on $H^+$.

Consider for instance the case that $X = \# pY$ is the connected sum of $p$ copies of a $4$-manifold $Y$ and $\sigma$ cyclically permutes the summands. Then by Corollary \ref{cor:psumx}, $SW_{\mathbb{Z}_p , \# pY , \# p \mathfrak{s}_Y}^{\phi}( x^{p-1} )$ is a non-zero multiple of $SW(Y,\mathfrak{s}_Y , \phi) v^{(p-1)/2}$. We obtain the following non-existence result: if $SW(Y,\mathfrak{s}_Y , \phi) \neq 0 \; ({\rm mod} \; p)$ then there does not exist an extension of the $\mathbb{Z}_p$-action on $X$ to a smooth action of $D_p$ which fixes $\# p \mathfrak{s}_Y$ and $\phi$ and for which $\tau$ preserves orientation on $H^+(X)$ if $p = 3 \; ({\rm mod} \; 4)$ or reverses orientation on $H^+(X)$ if $p = 1 \; ({\rm mod} \; 4)$.
\end{example}

\subsection{Obstructions to equivariant connected sum decompositions}

Let $X$ be a compact, oriented, smooth $4$-manifold with $b_1(X)=0$. Consider a smooth $G$-action on $X$. If $X$ can be written as an equivariant connected sum $X = X_1 \# X_2$ where $b_+(X_1)^G, b_+(X_2)^G > 0$, then the equivariant Seiberg--Witten invariants vanish by Theorem \ref{thm:ecs}. This can be used to limit the possible ways in which $X$ can be an equivariant connected sum.

\begin{example}
We will construct two smooth $\mathbb{Z}_3$-actions on $X = 5 \mathbb{CP}^2 \# 23 \overline{\mathbb{CP}}^2$ with the same fixed point data. One of these actions will decompose as an equivariant connected sum $X = X_1 \# X_2$ with $b_+(X_1)^G , b_+(X_2)^G > 0$, the other will not decompose in this way. The actions on $X$ will be constructed from the following:
\begin{itemize}
\item[(1)]{Let $K$ be the Fermat quartic $\{ [z_0,z_1,z_2,z_3] \in \mathbb{CP}^3 \; | \; z_0^4 + z_1^4 + z_2^4 + z_3^4 = 0\}$ with $\mathbb{Z}_3$-action given by $[z_0,z_1,z_2,z_3] \mapsto [z_0 , z_2 , z_3 , z_1]$.}
\item[(2)]{Let $\mathbb{CP}^2_{(a)}$ denote $\mathbb{CP}^2$ with $\mathbb{Z}_3$-action $[z_0,z_1,z_2] \mapsto [z_0 , \omega z_1 , \omega^a z_2]$, where $\omega = e^{2\pi i/3}$.}
\item[(3)]{Let $2(S^2 \times S^2)$ be the branched triple cover of an unknotted torus in $S^4$.}
\end{itemize}

The action of $\mathbb{Z}_3$ on the tangent space of an isolated fixed point will be orientation preservingly isomorphic to $(z_1,z_2) \mapsto ( \omega^{-1} z_1 , \omega^a z_2)$, where $a$ is either $1$ or $-1$. Let $n_{\pm}$ denote the number of isolated fixed points of type $a = \pm 1$. Then $K$ has $(n_+ , n_-) = (6,0)$, $\mathbb{CP}^2_{(1)}$ has $(n_+ , n_-) = (0,1)$, $\mathbb{CP}^2_{(2)}$ has $(n_+ , n_-) = (3,0)$ and $2(S^2 \times S^2)$ has $(n_+ , n_-) = 0$.

Let $Y$ be the equivariant sum $Y = \overline{\mathbb{CP}}^2_{(2)} \# \overline{\mathbb{CP}}^2_{(1)} \# 2(S^2 \times S^2)$, where the first and second summands are connected along isolated fixed points and the second and third summands are connected along non-isolated fixed points. Then $X$ is diffeomorphic to $K \# Y$ and so we obtain a $\mathbb{Z}_3$-action on $X$ by considering $K \# Y$ as an equivariant connected sum. We have that $b_+(Y)^G = 0$ and if we equip $2(S^2 \times S^2)$ with its unique spin structure and equip each $\overline{\mathbb{CP}}^2$ summand of $Y$ with a spin$^c$-structure satisfying $c(\mathfrak{s})^2 = -1$, then we obtain an invariant spin$^c$-structure $\mathfrak{s}_Y$ on $Y$ for which $d_Y = 0$. If $\mathfrak{s}_K$ denotes the unique spin structure on $K$, then Theorem \ref{thm:ecs} gives
\[
SW_{\mathbb{Z}_3 , K \# Y , \mathfrak{s}_K \# \mathfrak{s}_Y}(1) = SW_{\mathbb{Z}_3 , K , \mathfrak{s}_K}(1) = SW(K , \mathfrak{s}_K) = 1.
\]
So $K \# Y$ has a non-zero equivariant Seiberg--Witten invariant and can't be obtained as an equivariant connected sum of the form $X_1 \# X_2$ with $b_+(X_1)^G, b_+(X_2)^G > 0$.

Next let $K'$ be defined as follows. First take the equivariant connected sum $K'_0 = 3 \mathbb{CP}^2_{(2)} \# \overline{\mathbb{CP}}^2_{(2)}$ where the three $\mathbb{CP}^2_{(2)}$ summands are attached to the three isolated fixed points of $\overline{\mathbb{CP}}^2_{(2)}$. Then let $K' = K'_0 \# 18 \overline{\mathbb{CP}}^2$, where the $\mathbb{Z}_3$-action permutes the $18$ copies of $\overline{\mathbb{CP}}^2$ in six $3$-cycles. It is easily seen that the $\mathbb{Z}_3$-action on $K'$ has the same fixed point data as $K$, namely $(n_+ , n_-) = (6,0)$ and no non-isolated fixed points. Hence the $\mathbb{Z}_3$-actions on $K \# Y$ and $K' \# Y$ have the same fixed points as well. Moreover, $K' \# Y$ is diffeomorphic to $X$ so this gives another $\mathbb{Z}_3$-action on $X$ with the same fixed point data. Unlike the first action this one can be decomposed into an equivariant connected sum $X_1 \# X_2$ with $b_+(X_1)^G, b_+(X_2)^G > 0$. This is clear because each of the three $\mathbb{CP}^2_{(2)}$-summands in $K'_0$ have $b_+^G = 1$.
\end{example}

\subsection{Obstructions to equivariant positive scalar curvature}

The $4$-manifolds of the form $a \mathbb{CP}^2 \# b\overline{\mathbb{CP}}^2$ admit metrics of positive scalar curvature. On the other hand the equivariant Seiberg--Witten invariants can be used to find many examples of actions on such manifolds for which there is no invariant metric of positive scalar curvature. Let $X$ be a simply-connected $4$-manifold with $b_+(X) > 1$ on which $\mathbb{Z}_p$ acts smoothly. Assume the action has a non-isolated fixed point and that $X$ has a non-zero mod $p$ Seiberg--Witten invariant. We assume $X$ is not spin (if necessary we can replace $X$ by a blow up to achieve this). Then for some $g>0$ we have that $X \# g(p-1)(S^2 \times S^2)$ is diffeomorphic to $a \mathbb{CP}^2 \# b\overline{\mathbb{CP}}^2$, where $a = b_+(X)+g(p-1)$, $b = b_-(X)+g(p-1)$. Now recall that $Y_g = g(p-1)(S^2 \times S^2)$ can be realised as the degree $p$ cyclic branched cover of an unknotted surface in $S^4$ of genus $g$. We have $b_+(Y_g)^{\mathbb{Z}_p} = 0$ and then $X \# Y_g$ has non-zero equivariant Seiberg--Witten invariants by Theorem \ref{thm:ecs}. This gives a smooth $\mathbb{Z}_p$-action on $a \mathbb{CP}^2 \# b\overline{\mathbb{CP}}^2$ which does not admit an invariant positive scalar curvature metric.


\bibliographystyle{amsplain}

\end{document}